%% file: main_neutral.tex
\documentclass[12pt]{article}
\usepackage{setspace}
\usepackage[margin=1in]{geometry}
\usepackage{amsmath}
\usepackage{amsfonts}
\usepackage{amssymb}
\usepackage{amsthm}
\usepackage{comment}
\usepackage{caption}
\usepackage{subcaption}
\usepackage{booktabs} 
\usepackage{makecell}
\usepackage{cellspace} 
\usepackage{multirow}
\usepackage{xcolor}
\definecolor{C0}{HTML}{3182ce}
\definecolor{forestgreen}{rgb}{0,0,0}
\usepackage[pagebackref=true]{hyperref}
\hypersetup{
	colorlinks=true,
	urlcolor=C0,
    citecolor=C0,
	linktoc=all,
	linkcolor=C0}
	
\renewcommand*{\backrefalt}[4]{%
	\ifcase #1 \footnotesize{(Not cited.)}%
	\or        \footnotesize{(Cited on page~#2.)}%
	\else      \footnotesize{(Cited on pages~#2.)}%
	\fi}

\usepackage{graphicx}
\usepackage{float}
\usepackage{natbib}
\setcitestyle{citesep={,}}
\usepackage{bm}
\usepackage{enumitem}
\usepackage{bbding}
\usepackage{colortbl}
\usepackage{algorithm}
\usepackage{algorithmic}
\usepackage [latin1,utf8]{inputenc}

\newtheorem{Thm}{Theorem}

\newtheorem{Prop}{Proposition}

\counterwithin*{Def}{section}

\newtheorem{Lem}{Lemma}

\counterwithin*{Lem}{section}

\newtheorem{Cond}{Condition}

\counterwithin*{Hyp}{section}

\newcommand{\R}{\mathbb{R}}
\newcommand{\X}{\mathcal{X}}
\newcommand{\Z}{\mathcal{Z}}
\newcommand{\A}[1]{\mathcal{A}_{#1}(n)}
\newcommand{\e}{\mathrm{e}}

\begin{document}

\begin{center}
{\Large
\textbf{Bayesian mixture models (in)consistency \\
for the number of clusters}
}
\bigskip

 Louise Alamichel$^{1,\dagger}$\footnote{This work has been partially supported by the LabEx PERSYVAL-Lab (ANR-11-LABX-0025-01) funded by the French program Investissements d'Avenir.\\
 \indent 
 $^\dagger$Equal contribution.}, Daria Bystrova$^{1,\dagger}$, Julyan Arbel$^{1}$ \& Guillaume Kon Kam King$^{2}$ 
\bigskip

{\it
$^{1}$Univ. Grenoble Alpes, CNRS, Inria, Grenoble INP, LJK, 38000 Grenoble, France
}\\
{\tt \{louise.alamichel, daria.bystrova, julyan.arbel\}@inria.fr}

{\it
$^{2}$Université Paris-Saclay, INRAE, MaIAGE, 78350 Jouy-en-Josas, France
}\\
{\tt guillaume.kon-kam-king@inrae.fr}

\end{center}
\bigskip


\input{article}

\section*{Supporting Information}
Additional information for this article is available \href{https://github.com/dbystrova/BNPconsistency}{online}, corresponding to the code used for the simulations and the figures.

\bibliographystyle{template/ba}
\bibliography{sample.bib}


\newpage
\input{appendix}


\end{document}

%% file: article.tex

\begin{abstract}
    \begin{onehalfspace}
    Bayesian nonparametric mixture models are common for modeling complex data. While these models are well-suited for density estimation, recent results proved posterior inconsistency of the number of clusters when the true number of components is finite, for the Dirichlet process and Pitman--Yor process mixture models. We extend these results to additional Bayesian nonparametric priors such as Gibbs-type processes and finite-dimensional representations thereof. The latter include the Dirichlet multinomial process, the recently proposed Pitman--Yor, and normalized generalized gamma multinomial processes. We show that mixture models based on these processes are also inconsistent in the number of clusters and discuss possible solutions. Notably, we show that a post-processing algorithm introduced for the Dirichlet process can be extended to more general models and provides a consistent method to estimate the number of components.
     \end{onehalfspace}
\end{abstract}

\noindent{\small \textbf{Keywords}: {Clustering; Finite mixtures; Finite-dimensional BNP representations; Gibbs-type process}
}



\section{Introduction}

\paragraph{Motivation.} Mixture models appeared as a natural way to model heterogeneous data, where observations may come from different populations. Complex probability distributions can be broken down into a combination of simpler models for each population. Mixture models are used for density estimation, model-based clustering \citep{fraley_model-based_2002} and regression \citep{muller1996bayesian}. Due to their flexibility and simplicity, they are widely used in many applications such as healthcare \citep{Ramrez_2019, ullah2019bayesian}, econometrics \citep{FrhwirthSchnatter_2012}, ecology \citep{attorre2020finite} and many others \citep[further examples in][]{fruhwirth-schnatter_handbook_2019}.

In a mixture model, data $X_{1:n} = (X_1,\ldots,X_n)$, $X_i \in \mathcal{X}\subset \mathbb{R}^p$ are modeled as coming from a $K$-components mixture distribution. If the \textit{mixing measure} $G$ is discrete, i.e. $G =\sum_{i=1}^K w_i\delta_{\theta_i}$ with positive weights $w_i$ summing to one and atoms $\theta_i$, then the \textit{mixture density} is
 \begin{equation} \label{mixture def}
     f^X(x) = \int f(x\mid\theta)G(\mathrm{d}\theta)=  \sum_{k=1}^K w_k f(x\mid\theta_k), 
\end{equation}
where $f(\cdot\mid\theta)$ represents a component-specific kernel density parameterized by $\theta$. We denote the set of parameters by $\theta_{1:K}= (\theta_1,\ldots, \theta_K)$, where each $\theta_k \in \mathbb{R}^d,\, k = 1,\ldots,K$.
Model~\eqref{mixture def} can be equivalently represented through latent allocation variables $z_{1:n} = \left(z_1, \ldots, z_n\right)$, $z_i \in \{1,\ldots,K\}$. Each $z_i$ denotes the component from which observation $X_i$ comes: $p(X_i\mid \theta_k) = p(X_i \mid z_i = k)$ with $w_k = P(z_i = k)$. Allocation variables $z_i$ define a clustering such that $X_i$ and $X_j$ belong to the same cluster if $z_i = z_j$. Moreover, $z_1,\ldots,z_n$ define a partition $A = (A_1,\ldots,A_{K_{n}})$ of $\{1,\ldots,n\}$, where $K_{n}$ denotes the number of clusters. 
    
It is important to distinguish between the number of components $K$, which is a model parameter, and the number of clusters $K_{n}$, which is the number of components from which we observed at least one data point in a dataset of size $n$ \citep{fruhwirth2021generalized, argiento2019infinity, greve2020spying}.
For a data-generating process with $K_0$ components, inference on \textcolor{forestgreen}{$K$} is typically done by considering the number of clusters $K_n$ and the present article investigates to what extent this is warranted.

Although mixture models are widely used in practice, they remain the focus of active theoretical investigations, owing to multiple challenges related to the estimation of mixture model parameters. These challenges stem from identifiability problems \citep{fruhwirth2006finite}, label switching \citep{celeux2000computational}, and computation complexity due to the large dimension of parameter space. 

Another critical question, which is the main focus of this article, regards the number of components and clusters, and whether it is possible to infer them from the data. This question is even more crucial when the aim of inference is clustering. 
The typical approach to estimating the number of components in a mixture is to fit models of varying complexity and perform model selection using a classic criterion such as the Bayesian Information Criterion (BIC), the Akaike Information Criterion (AIC), etc. 
This approach is not entirely satisfactory in general, because of the need to fit many separate models and the difficulty of performing a reliable model selection. 
Therefore, several methods that bypass the need to fit multiple models have been proposed. They define a single flexible model accommodating various possibilities for the number of components: mixtures of finite mixtures, Bayesian nonparametric mixtures, and overfitted mixtures. These methods have been prominently proposed in the Bayesian framework, where the specification of prior information is a powerful and versatile method to avoid overfitting by unduly complex mixture models.

\paragraph{Three types of discrete mixtures.} 
Although we consider discrete mixing measures, $G$ could be any probability distribution \citep[for continuous mixing measures, see for instance Chapter 10 in][]{fruhwirth-schnatter_handbook_2019}. 
Depending on the specification of the mixing measure, there exist three main types of discrete mixture models: \textit{finite mixture} models where the number of components $K$ is considered fixed (known, equal to $K_0$, or unknown), \textit{mixture of finite mixtures} (MFM) where $K$ is random and follows some specific distribution, and \textit{infinite mixtures} where $K$ is infinite.  Under a Bayesian approach, the latter category is often referred to as Bayesian nonparametric (BNP) mixtures.

Specification of the number of components $K$ is different for the three types of mixtures. When $K$ is unknown, the Bayesian approach provides a natural way to define the number of components by considering it random and adding a prior for $K$ to the model, as is done for mixtures of finite mixtures. Inference methods for MFM were introduced by \cite{nobile_1994, Richardson_1997}. 

Using Bayesian nonparametric (BNP) priors for mixture modeling is another way to bypass the choice of the number of components $K$. This is achieved by assuming an infinite number of components, which adapts the number of clusters found in a dataset to the structure of the data. The most commonly used BNP prior is the Dirichlet process introduced by \cite{ferguson1973bayesian} and the corresponding Dirichlet process mixture was first introduced by \cite{lo1984class}. The success of the Dirichlet process mixture is based on its ease of implementation and computational tractability. 
However, in some cases the Dirichlet process prior may be restrictive, so more flexible priors such as the Pitman--Yor process can be used.  Gibbs-type processes, introduced by \cite{gnedin2006exchangeable}, 
form an important general class of priors, which contain Dirichlet and Pitman--Yor processes and have flexible clustering properties while maintaining mathematical tractability, see \cite{lijoi_prunster_2010,de_blasi_are_2015} for a review. 
Compared to the Dirichlet process, Gibbs-type priors exhibit a predictive distribution which involves more information, that is, sample size and number of clusters \citep[refer to the sufficientness postulates for Gibbs-type priors of][]{Sufficientness}.
\textcolor{forestgreen}{The class of Gibbs-type priors encompasses BNP processes which} are widely used, for instance in species sampling problems \citep{lijoi_species_sampling2007,favaro_2009,favaro2012new, CESARI201479,arbel2017bayesian}, survival analysis \citep{Jara_2010}, network inference  \citep{caron2017sparse, Legramanti_2022}, linguistics \citep{teh2010hierarchical} and mixture modeling \citep{Ishwaran_2001, Lijoi_2005,lijoi2007controlling}.  \cite{miller2018,fruhwirth2021generalized,argiento2019infinity} study the connection between the mixtures of finite mixtures and BNP mixtures with Gibbs-type priors.
A common approach to inferring the number of clusters in Bayesian nonparametric models is through the posterior distribution of the number of clusters. 

Finally, finite mixture models are considered when $K$ is assumed to be finite. 
We distinguish two cases, depending on whether the number of components is known or unknown. The case when the number of components is known, say $K=K_0$, is referred to as the exact-fitted setting. An appealing way to handle the other case ($K_0$ unknown) is to use a chosen \textcolor{forestgreen}{upper} bound on $K_0$, i.e. to take the number of components $K$ such that $K\geq K_0$, yielding the so-called overfitted mixture models. A classic overfitted mixture model is based on the Dirichlet multinomial process, which is a finite approximation of the Dirichlet process \citep[see][for instance]{ishwaran_exact_2002}. Generalizations of the Dirichlet multinomial process were recently introduced by \cite{lijoi2020finite-dim,lijoi2020pitman}, which lead to more flexible overfitted mixture models.

\paragraph{Asymptotic properties of Bayesian mixtures.}
A minimal requirement for the reliability of a statistical procedure is that it should have reasonable asymptotic properties, such as consistency. This consideration also plays a role in the Bayesian framework, where asymptotic properties of the posterior distribution may be studied. In Table~\ref{Tab:summary}, we provide a summary of existing results of posterior consistency for the three types of mixture models, when it is assumed that data come from a finite mixture and that the kernel $f(\cdot\mid \theta)$ correctly describes the data generation process (i.e. the so-called \textit{well-specified setting}). We denote by $K_0$ the true number of components, $G_0$ the true mixing measure, and $f_0^X$ the true density written in the form of \eqref{mixture def}. For finite-dimensional mixtures, Doob's theorem provides posterior consistency in density estimation \citep{nobile_1994}.  However, this is a more delicate question for BNP mixtures. Extensive research in this area provides consistency results for density estimation under different assumptions for Bayesian nonparametric mixtures,  such as for Dirichlet process mixtures \citep{ghosal1999,ghosal2007posterior, kruijer_adaptive_2010} and other types of BNP priors \citep{lijoi2005consistency}. 
 In the case of MFM, posterior consistency in the number of clusters as well as in the mixing measure follows from Doob's theorem and was proved by \cite{nobile_1994}. Recently, \cite{miller2022consistency} provided \textcolor{forestgreen}{a} new proof with simplified assumptions.

For finite mixtures and Bayesian nonparametric mixtures, under some conditions of identifiability, kernel continuity, and uniformity of the prior, \cite{nguyenConvergenceLatentMixing2013} proves consistency for mixing measures and provides corresponding contraction rates.
These results only guarantee consistency for the mixing measure and do not imply consistency of the posterior distribution of the number of clusters. In contrast, posterior inconsistency of the number of clusters for Dirichlet process mixtures and Pitman--Yor process mixtures is proved by \cite{miller2014inconsistency}. 
To the best of our knowledge, this result was not shown to hold for other classes of priors. We fill this gap and provide an extension of \cite{miller2014inconsistency} results for Gibbs-type process mixtures and some of their finite-dimensional representations. 

Inconsistency results for mixture models do not impede real-world applications but suggest that inference about the number of clusters must be taken carefully. 
On the positive side, and in the case of overfitted mixtures, \cite{rousseau2011asymptotic} establish that the weights of extra components vanish asymptotically under certain conditions. 
Additional results by  \cite{chambaz:hal-00262054} establish posterior consistency for the mode of the number of clusters. 
\cite{guha2021posterior} propose a post-processing procedure that allows consistent inference of the number of clusters in mixture models. They focus on Dirichlet process mixtures and we provide an extension for Pitman--Yor process mixtures and overfitted mixtures in this article. Another possibility to solve the problem of inconsistency is to add flexibility for the prior distribution on a mixing measure through a prior on its hyperparameters. For Dirichlet multinomial process mixtures,
\cite{malsiner-walli_model-based_2016} observe empirically that adding a prior on the $\alpha$ parameter helps with centering the posterior distribution of the number of clusters on the true value (see their Tables 1 and 2). A similar result is proved theoretically by \cite{ascolani2022clustering} for Dirichlet process mixtures under mild assumptions. 

As a last remark, although we focus on the well-specified case, an important research line in mixture models revolves around misspecified-kernel mixture models, when data are generated from a finite mixture of distributions that do not belong to the kernel family $f(\cdot \mid \theta)$. \cite{miller2018robust} shows how so-called coarsened posteriors allow performing inference on the number of components in MFMs with Gaussian kernels when data come from skew-normal mixtures. 
 \cite{cai2021finite} provide theoretical results for MFMs, when the mixture component family is misspecified, showing that the posterior distribution of the number of components diverges. 
 Misspecification is of course a topic of critical importance in practice, however, the well-specified case is challenging enough to warrant its own extensive investigation.

\paragraph{Contributions and outline.}
In this rather technical landscape, it can be difficult for the non-specialist to keep track of theoretical advances in Bayesian mixture models. 
This article aims to provide an accessible review of existing results, as well as the following novel contributions (see Table~\ref{Tab:summary}):
\begin{itemize}
    \item We extend \cite{miller2014inconsistency} results to additional Bayesian nonparametric priors such as Gibbs-type processes (Proposition~\ref{prop:Gibbs}) and finite-dimensional representations of them (including the Dirichlet multinomial process and Pitman--Yor and normalized generalized gamma multinomial processes, Proposition~\ref{prop:PYM});
    \item We discuss possible solutions. In particular, we show that the \cite{rousseau2011asymptotic} result regarding emptying of extra clusters holds for the Dirichlet multinomial process and Pitman--Yor multinomial process (Proposition~\ref{prop:over}). Second, we establish that the post-processing algorithm introduced by \cite{guha2021posterior} for the Dirichlet process extends to more general models and provides a consistent method to estimate the number of components (Propositions~\ref{prop:consPY} and \ref{prop:consO}).
    \item We also provide insight into the non-asymptotic efficiency and practical application of these solutions through an extensive simulation study, and investigate alternative approaches which add flexibility to the prior distribution of the number of clusters.
\end{itemize}


\begin{table}[H]
\caption{Results on consistency for different mixture models and quantities of interest in the case where kernel densities are well-specified and data comes from a finite mixture. Consistency is indicated with \textcolor{green!50!black}{\Checkmark} and inconsistency with \textcolor{red!80!black}{\XSolidBrush}. Our contributions regard the shaded cells. The references cited are [RGL19] \citet[][Theorem 4.1]{rousseau2019bayesian}; [GvdV17] \citet[][Theorem 7.15]{ghosal2017fundamentals}; [KRV10] \citet[][]{kruijer_adaptive_2010}; [HN16] \citet[][]{ho_strong_2016}; [Ngu13] \citet[][]{nguyenConvergenceLatentMixing2013}; [Nob94] \citet[][]{nobile_1994}; [MH14] \cite{miller2014inconsistency}; [GHN21] \cite{guha2021posterior}.}
\centering
\begin{tabular}{ScScScScSc}
    \toprule
    \multirow{2}{*}[-0.3em]{Quantity of interest} &
    \multicolumn{2}{c}{Finite} & 
    Infinite & 
    MFM
    \\
    \cmidrule(l){2-5}
    {} &
    $K=K_0$  &
    $K \geq K_0$ &
    $K = \infty$ &
    $K$ random 
    \\ \midrule
    Density $f^X_0$ &
    \textcolor{green!50!black}{\Checkmark} [RGL19] &
    \textcolor{green!50!black}{\Checkmark} [RGL19] &
    \textcolor{green!50!black}{\Checkmark} [GvdV17] &
    \textcolor{green!50!black}{\Checkmark} [KRV10]
    \\ 
    Mixing measure $G_0$ &
    \textcolor{green!50!black}{\Checkmark} [HN16] &
    \textcolor{green!50!black}{\Checkmark} [HN16] &
    \textcolor{green!50!black}{\Checkmark} [Ngu13] &
    \textcolor{green!50!black}{\Checkmark} [Nob94]
    \\
    Nb of components $K_0$ &
    N/A &
    \cellcolor{gray!20!white}\textcolor{red!80!black}{\XSolidBrush} [ours] / \textcolor{green!50!black}{\Checkmark}  &
    \cellcolor{gray!20!white}\textcolor{red!80!black}{\XSolidBrush} [MH14, ours] / \textcolor{green!50!black}{\Checkmark}  &
    \textcolor{green!50!black}{\Checkmark} [GHN21]
    \\\bottomrule
\end{tabular}
 \label{Tab:summary}
\end{table}

The structure of the article is as follows: we start by introducing the notion of a partition-based mixture model and by presenting Gibbs-type processes and finite-dimensional representations of BNP processes in Section \ref{sec:priors}. 
We then recall in Section \ref{sec:inconsistency} the inconsistency results of \cite{miller2014inconsistency} on Dirichlet process mixtures and Pitman--Yor process mixtures and present our generalization.
We discuss some consistency results and a post-processing procedure in Section \ref{sec:consistency}. We conclude with a simulation study illustrating some of our results in Section \ref{sec:simulation} \textcolor{forestgreen}{and a real data analysis in Section \ref{sec:real}}, while the appendix contains proofs and additional details on the simulation \textcolor{forestgreen}{and real data} study.

\section{Bayesian mixture models and mixing measures}
\label{sec:priors}
We introduce or recall some notions useful for the rest of the paper. We start by defining the mixture model considered. It is based on a partition, whose distribution determines important aspects of the mixture. 
We introduce different types of priors on the partition, the Gibbs-type process, and some finite-dimensional representations of nonparametric processes such as the Pitman--Yor multinomial process. We conclude this section by recalling the notions of posterior consistency and contraction rate.

\subsection{Partition-based mixture model}
\label{sec:partition}

We consider partition-based mixture models as in \cite{miller2014inconsistency}. Let $\A{k}$ be the set of ordered partitions of $\{1,\ldots,n\}$  into $k  \in \{1,\ldots, n\}$ nonempty sets:
\begin{equation*}
\A{k} := \left\{ (A_1,\ldots,A_k) : A_1,\ldots,A_k \text{ disjoint, } \bigcup_{i=1}^k A_i = \{1,\ldots,n\}, \;  |A_i|\geq 1 \; \forall i\right\}.
\end{equation*}

We denote by $n_i :=|A_i|$ the cardinality of set $A_i$. We consider a partition distribution $p(A)$ on $\bigcup_{k=1}^n \A{k}$, which induces a distribution $p(k)$ on $\{1,\ldots,n\}$. We denote by $\pi$ a prior density on the parameters $\theta \in \Theta \subset\R^d$ and $f(\cdot \mid \theta)$ a parametrized component density. 
The hierarchical structure of a \textit{partition-based mixture model} is:
\begin{align*}
    p(\theta_{1:k} \,|\, A, k) &= \prod_{i = 1}^k\pi(\theta_i), \\ 
    p(X_{1:n} \,|\, A,k,\theta_{1:k}) &= \prod_{i=1}^k \prod_{j\in A_i}f(X_j \mid \theta_i),
\end{align*}
where $X_{1:n} = (X_1,\ldots,X_n)$ with $X_i\in\X$, $\theta_{1:k} = (\theta_1,\ldots,\theta_k)$ with $\theta_i\in\Theta$, and $A\in\A{k}$.
In the rest of the article, we denote by $K_n$ the number of clusters in a dataset of size $n$, which is denoted $k$ in this section for ease of presentation. $K_n$ highlights this quantity's random nature and dependence on $n$.

The distribution $p$ on the set of ordered partitions determines the type of the mixture model. Here, we consider two types of prior distributions on the partition: nonparametric ones as a Dirichlet process or a Gibbs-type process, and finite-dimensional ones as a Pitman--Yor multinomial process or a normalized infinitely divisible multinomial process.

\subsection{Gibbs-type processes}
Gibbs-type processes are a natural generalization of the Dirichlet process and Pitman--Yor process \citep[see for example][]{de_blasi_are_2015}. Gibbs-type processes of type $\sigma \in (-\infty,1)$ can be characterized through the probability distribution of the induced random ordered partition $A \in \A{k}$, which has the following form:
\begin{equation}
\label{eq:EPPF}
    p(A) = p(n_1, \ldots, n_k) =\frac{V_{n,k}}{k!}\prod_{j=1}^k (1-\sigma)_{n_j-1},  
\end{equation}
where $(x)_{n} = x(x+1)\cdots(x+n-1)$ is the ascending factorial and $(x)_{0} = 1$ by convention. $V_{n,k}$ are nonnegative numbers that satisfy the recurrence relation: 
\begin{equation}
    \label{eq:rec_Vnk}
    V_{n,k} = (n-\sigma k)V_{n+1,k}+V_{n+1,k+1}, \quad V_{1,1} = 1.
\end{equation}
The probability distribution for the unordered partition $\tilde{A}$ can be deduced from \eqref{eq:EPPF} multiplying by $k!$ to adjust for order: $p(\tilde{A}) = V_{n,k} \prod_{j=1}^k (1-\sigma)_{n_j-1}$.
Parameters $V_{n,k}$ admit the following form \citep[see][]{pitman2003poisson,gnedin2006exchangeable}:
\begin{equation}
\label{eq:Vnk2}
    V_{n,k}=\frac{\sigma^{k}}{\Gamma(n-k\sigma)}\int_{0}^{+\infty}\int_{0}^{1}t^{-k\sigma}p^{n-k\sigma-1}h(t)f_{\sigma}((1-p)t)\mathrm{d}t\mathrm{d}p,
\end{equation}
with $\Gamma$ the gamma function, $f_\sigma$ the density of a positive $\sigma$-stable random variable and $h$ a non-negative function. We limit ourselves to the case $0<\sigma<1$.

Gibbs-type processes are a general class including the Dirichlet and Pitman--Yor processes and some stable processes. 
The Pitman--Yor family can be defined by the probability $p$ in \eqref{eq:EPPF} with parameters
\[ V_{n,k} = \frac{\prod_{i=1}^{k-1} (\alpha+i\sigma)}{(\alpha+1)_{n-1}},
\]
where $\sigma \in [0,1)$ and $\alpha\in(-\sigma,\infty)$. 
If $\sigma=0$, we obtain the Dirichlet process for which $V_{n,k} = \alpha^k/(\alpha)_n$.

The normalized generalized gamma process \citep[NGG,][]{lijoi2007controlling} is another particular case of Gibbs-type processes, with  parameters 
\begin{equation}
    \label{eq:Vnk_NGG}
    V_{n,k} = \frac{\e^\beta\sigma^{k-1}}{\Gamma(n)}\sum_{i=0}^{n-1} \binom{n-1}{i}(-1)^i\beta^{i/\sigma}\Gamma\left(k-\frac{i}{\sigma};\beta\right), 
\end{equation} 
where $\sigma\in(0,1)$, $\beta>0$ and $\Gamma(\cdot;\cdot)$ is the following incomplete gamma function: $\Gamma(x;a)=\int_x^\infty s^{a-1}\e^{-s} \mathrm{d}s$. If $\beta = 0$ we obtain the normalized $\sigma$-stable process. Furthermore, if $\sigma\to 0$, then we also recover the Dirichlet process (see Figure~\ref{fig:processes-graphs}(a) for a graphical representation of the relations between these BNP processes).


\subsection{Finite-dimensional representations}
\label{sec:finite-dim}
Finite-dimensional representations for BNP priors have been developed to deal with situations where the increase of the number of clusters with the sample size is unrealistic, such as when \textcolor{forestgreen}{an upper} bound on the number of clusters is known. They are convenient and tractable models that share many properties of their infinite-dimensional counterparts, such as a clear interpretation of their parameters and efficient sampling algorithms. They naturally approximate their associated nonparametric priors as their dimension increases. See Figure~\ref{fig:processes-graphs}(b) for a graphical representation of the relations between these multinomial mixing measures.
\paragraph{Dirichlet multinomial process.} The simplest example of such a finite-dimensional representation is the Dirichlet multinomial distribution \citep[see for instance][]{muliere1995note,ishwaran2000markov}. A Dirichlet multinomial process with concentration parameter $\alpha>0$, number of components $K$,  and base measure $P$, is a random discrete measure $G= \sum_{k=1}^K w_k\delta_{\theta_k}$ characterized by a Dirichlet distribution on the weights with parameter $\alpha/K$: $(w_1,\ldots, w_n) \sim \text{Dir}(\alpha/K, \ldots, \alpha/K)$ and, as usual, location parameters $\theta_k$ are distributed according to the base measure $P$.  \cite{muliere2003weak} proves that the Dirichlet multinomial process with parameters $\alpha$, $K$,  and $P$  approximates the Dirichlet process with parameters $\alpha$ and $P$, in the sense of the weak convergence, when $K \to \infty$.
Recent works by \cite{lijoi2020finite-dim,lijoi2020pitman} develop finite-dimensional versions of the Pitman--Yor process and normalized random measures with independent increments \citep[][]{regazzini_distributional_2003}. The latter include the Dirichlet and normalized generalized gamma multinomial processes as special cases.

\paragraph{Pitman--Yor multinomial process.}
The Pitman--Yor multinomial process is based on the Pitman--Yor process. Fix some integer $K\geq1$, base measure $P$, and parameters $\alpha, \sigma$ as in the Pitman--Yor process case above.
The Pitman--Yor multinomial process is defined by \cite{lijoi2020pitman} as a discrete random probability measure $p_K$ such that 
\[ G_K\mid G_{0,K}\sim \textsc{PY}(\sigma,\alpha;G_{0,K}), \quad G_{0,K} = \frac{1}{K}\sum_{k=1}^K\delta_{\tilde\theta_k}, \]
where $\tilde\theta_k \overset{\text{iid}}{\sim} P$.  
For all $A\in\A{k}$, the partition distribution for the Pitman--Yor multinomial process is:
\begin{equation}
    \label{eq:EPPF_PYM}
    p(A)  = \binom{K}{k}\frac{1}{(\alpha+1)_{n-1}} \sum_{(\ell_1,\ldots,\ell_k)} \frac{\Gamma(\alpha/\sigma+|\ell^{(k)}|)}{\sigma\Gamma(\alpha/\sigma+1)}\prod_{i=1}^{k} \frac{C(n_i,\ell_i;\sigma)}{K^{\ell_i}},
\end{equation}
where $k = |A|$ and the sum runs over the vectors $\ell^{(k)} = (\ell_1,\ldots,\ell_k)$ such that $\ell_i \in \{1,\ldots,n_i\}$ and $|\ell^{(k)}| = \ell_1+\cdots+\ell_k$. Coefficients $C(n,k;\sigma)$ are the generalized factorial coefficients defined as
\begin{equation}
    \label{eq:Cnk}
    C(n,k;\sigma) = \frac{1}{k!}\sum_{j=0}^k(-1)^j \binom{k}{j}(-j\sigma)_n
\end{equation} 
As with the Pitman--Yor process, the random probability measure $p_K$ of the Pitman--Yor multinomial process reduces to the Dirichlet multinomial process when $\sigma=0$.
The Pitman--Yor multinomial process is thus a generalization of the Dirichlet multinomial process. 
As the latter, the Pitman--Yor multinomial process approximates the Pitman--Yor process, as the Pitman--Yor process is obtained as a limiting case when $K \to \infty$ \citep[see Theorem 5 in][]{lijoi2020pitman}. In addition, it is also more flexible than the Dirichlet multinomial process. 
It can be used as an effective computational tool in a nonparametric setting by replacing the stick-breaking construction in the classic Gibbs sampler \citep[see more details in][]{lijoi2020pitman}.

\paragraph{Normalized infinitely divisible multinomial process.}
Normalized infinitely divisible multinomial (\textsc{NIDM}) processes are introduced by \cite{lijoi2020finite-dim} and can be seen as a finite approximation for normalized random measures with independent increments (\textsc{NRMI}), see for instance \cite{regazzini_distributional_2003,james2009posterior}. 
\textsc{NIDM} processes can be described as \textsc{NRMI} measures using a hierarchical structure similar to the previous section
\[( G_K \mid  G_{0,K}) \sim \textsc{NRMI}(c, \rho;   G_{0,K}), \quad   G_{0,K} = \frac{1}{K} \sum_{k=1}^{K} \delta_{\tilde\theta_k},\]
where $\tilde\theta_k \overset{\text{iid}}{\sim}  P$ a base measure. In this expression, $\rho$ is a function that characterizes the \textsc{NRMI} process used. 
The choice $\rho(s) = s^{-1}e^{-s}$ corresponds to the Dirichlet process. It yields the Dirichlet multinomial process whose distribution for all $A\in\A{k}$ is defined as
\begin{equation}
    \label{eq:EPPF_DPM}
    p(A) =  \binom{K}{k}\frac{1}{(\alpha)_n}\prod_{j=1}^{k}(\alpha/K)_{n_j},    
\end{equation}
where $k = |A|$. Similarly, choosing $\rho(s) = \frac{1}{\Gamma(1-\sigma)}s^{-1-\sigma}e^{-\beta s}$, $0\leq\sigma<1$ and $\beta\geq0$ amounts to considering an \textsc{NGG} characterized by \eqref{eq:Vnk_NGG}. We then get the \textsc{NGG} multinomial process. In this case, for all $A\in\A{k}$ the probability is
\begin{equation}
    \label{eq:EPPF_NGGM}
    p(A) = \binom{K}{k} \sum_{(\ell_1,\ldots,\ell_k)} \frac{V_{n,|\ell^{(k)}|}}{K^{|\ell^{(k)}|}}\prod_{i=1}^{k} \frac{C(n_i,\ell_i;\sigma)}{\sigma^{\ell_i}},  
\end{equation}
where $k = |A|$ and $C(n,k;\sigma)$ are defined in \eqref{eq:Cnk} and the sum over $\ell^{(k)} = (\ell_1,\ldots,\ell_k)$ is as in the PY case. Parameters $V_{n,k}$ are defined in \eqref{eq:Vnk_NGG} for the particular case of \textsc{NGG} processes, which depend on $\beta$ and $\sigma$. 

\begin{figure}
    \centering
    \begin{tabular}{m{5cm} m{3.4cm} m{5cm}}
    \includegraphics[height = 4cm]{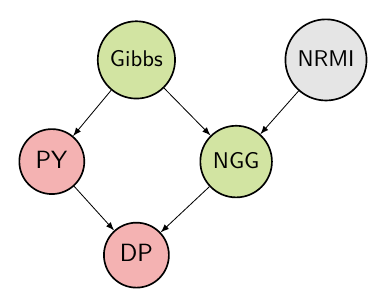} &
    $\begin{array}{c}
         \xrightarrow{\text{``Multinomialization''}}  \\
         \\
         \xleftarrow{\text{Weak limit as } K\to\infty} 
    \end{array}$ &
    \includegraphics[height = 4cm]{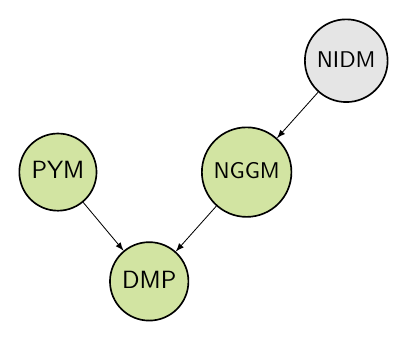} \\
     (a) BNP processes &   & (b) Multinomial processes
    \end{tabular}
    \caption{Graphical representation of the relationships between the discrete mixing measures considered in this article. An arrow indicates that the target is a special case of the origin. (a) BNP processes: Gibbs-type priors (Gibbs), normalized random measures with independent increments (\textsc{NRMI}), Pitman--Yor process (\textsc{PY}), normalized generalized gamma process (\textsc{NGG}), and Dirichlet process (\textsc{DP}). (b) Multinomial processes (finite-dimensional approximations of their respective BNP counterparts in the left panel): normalized infinitely divisible multinomial process (\textsc{NIDM}), Pitman--Yor multinomial process (\textsc{PYM}), normalized generalized gamma multinomial process (\textsc{NGGM}), and Dirichlet multinomial process (\textsc{DMP}). Going from left to right can be done according to a ``multinomialization'' of the BNP processes as described in Section~\ref{sec:finite-dim}, while the reverse direction is achieved by taking a weak limit as $K\to\infty$.  Our contributions generalize results known for mixing measures in red to mixing measures in green. The case of mixing measures in gray remains an open problem.}
    \label{fig:processes-graphs}
\end{figure}

\subsection{Posterior consistency}
Posterior consistency is an asymptotic property of the posterior. As in frequentist inference, we can consider that there exists a true value for the parameter of the distribution of the data. Then the posterior is said to be consistent if it converges to the true parameter as the sample size increases to infinity. 

More formally, given a prior distribution $\pi$ on the parameter space $\Theta$, we denote by $\Pi(\cdot\mid X_{1:n})$ the posterior distribution with $X_{1:n}$ a given sample of the data.
The posterior distribution is said to be consistent at $\theta_0\in\Theta$ if \smash{$\Pi(U^c\mid X_{1:n})\underset{n\to\infty}{\longrightarrow} 0$} in $P_{\theta_0}$-probability for all neighborhoods $U$ of $\theta_0$. 
For instance, in our case, we consider mixture models for densities. In this type of model, the posterior density is said to be consistent at $f_0^X$ if, for a distance $d$ on the parameter space, $\Pi(d(f,f_0^X)\geq\varepsilon\mid X_{1:n})\underset{n\to\infty}{\longrightarrow} 0$ in $P_{f_0^X}$-probability for all $\varepsilon>0$.
It is also possible to define posterior consistency for quantities of interest such as the number of clusters. 
The posterior number of clusters $K_n$ is said to be consistent at $K_0$ if $\Pi(K_n=K_0\mid X_{1:n})\underset{n\to\infty}{\longrightarrow} 1$ in $P_{f_0^X}$-probability. 

A refinement in the study of posterior consistency is to evaluate the speed at which a posterior distribution concentrates around the true parameter. The quantity that measures this speed is called a posterior contraction rate. More formally, the parameter space $\Theta$ is supposed to be a metric space with a metric $d$.
A sequence $\varepsilon_n$ is a posterior contraction rate at the parameter $\theta_0$ with respect to the metric $d$ if for every $M_n \to \infty$, $\Pi(d(\theta,\theta_0)\geq M_n\varepsilon_n\mid X_{1:n})\underset{n\to\infty}{\longrightarrow} 0$ in $P_{\theta_0}$-probability.

For more details on posterior consistency or contraction rates, the reader could refer to \citet[][Chapters 6 to 9]{ghosal2017fundamentals}.

\section{Inconsistency results} 
\label{sec:inconsistency}
In this section, we generalize the inconsistency results by \cite{miller2014inconsistency}. Under the context defined previously, \cite{miller2014inconsistency} states sufficient conditions that imply posterior inconsistency of the number of clusters and also proves that these conditions are satisfied for the Dirichlet process and Pitman--Yor process mixture models. For completeness, we first recall here this inconsistency result and then prove that it also applies to the different models introduced in Section \ref{sec:priors}.

\subsection{Inconsistency theorem of \cite{miller2014inconsistency}}
\label{sec:Th}
The central result of~\citet[Theorem~6]{miller2014inconsistency} is reproduced below as Theorem \ref{thm}. This result depends on two conditions which are discussed thereafter.

We start with some notations. 
For $A\in\A{k}$, we define $R_A = \bigcup_{i:|A_i|\geq2}A_i$, the union of all clusters except singletons.
For index $j\in R_A$, we define $B(A,j)$ as the ordered partition $B\in\A{k+1}$ obtained by removing $j$ from its cluster $A_\ell$ and creating a new singleton for it. Then $B_\ell = A_\ell\setminus \{j\}$, and $B_{k+1} = \{j\}$. Let $\Z_A := \{B(A,j): j\in R_A\}$, for $n>k\geq1$, we define
\[c_n(k) := \frac{1}{n} \max_{A\in\A{k}}\max_{B\in\Z_A} \frac{p(A)}{p(B)},\]
with the convention that $0/0 = 0$ and $y/0 = \infty$ for $y>0$.
\begin{Cond}
\label{cond:part}
    Assume $\limsup_{n\to \infty} c_n(k)<\infty$, given some particular $k\in \{1,2,\ldots\}$.
\end{Cond}
\noindent \cite{miller2014inconsistency} show that this condition holds for any $k\in \{1,2,\ldots\}$ for the Pitman--Yor process, and thus for the Dirichlet process.

The second condition, named Condition 4 in \cite{miller2014inconsistency}, controls the likelihood through the control of single-cluster marginals. 
The single-cluster marginal for cluster $i$ is $m(X_{A_i}) = \int_\Theta \left(\prod_{j\in A_i} f(X_j\mid\theta)\right) \pi(\theta)\mathrm{d}\theta$. 
This condition induces, for example, that as $n\to \infty$, there is always a non-vanishing proportion of the observations for which creating a singleton cluster increases its cluster marginal. 
This condition only involves the data distribution and is shown to hold for several discrete and continuous distributions, such as the exponential family \citep[see Theorem 7 in][]{miller2014inconsistency}. 
In the following, we assume that this condition is satisfied and mainly focus on Condition \ref{cond:part}.

\begin{Thm}[\citealp{miller2014inconsistency}]
\label{thm}
    Let $X_1,X_2,\ldots\in\X$ be a sequence of random variables and consider a partition-based model. Then, if Condition 4 from \cite{miller2014inconsistency} holds, and Condition~\ref{cond:part} above holds for any $k\geq1$, we have for any $k\geq1$
    \[\limsup_{n\to \infty} \Pi(K_n = k\mid X_{1:n})<1 \; \text{ with probability } 1. \]
\end{Thm}
As said previously, Condition~\ref{cond:part}  is only related to partition distribution, while Condition 4 from \cite{miller2014inconsistency} only involves the data distribution and single-cluster marginals. Hence, to generalize this inconsistency result to other processes, it is enough to show that Condition \ref{cond:part} also holds for these different processes. This is the focus of the next section, for Gibbs-type processes and finite-dimensional discrete priors.

\subsection{Inconsistency of Gibbs-type and multinomial processes}\label{sec:incons-gibbs}

We extend the inconsistency result for all the processes introduced in Section \ref{sec:priors} by proving that Condition \ref{cond:part} holds.
\begin{Prop}[Gibbs-type processes]
    \label{prop:Gibbs}
    Consider a Gibbs-type process with $0 \leq\sigma<1$, then Condition \ref{cond:part} holds for any $k\in\{1,2,\ldots\}$, and so does the inconsistency of Theorem \ref{thm}.
\end{Prop}

\begin{Prop}[Multinomial processes]
    \label{prop:PYM}
    Consider any of the following priors: 
    Dirichlet multinomial process,
    Pitman--Yor multinomial process and
    normalized generalized gamma multinomial process, with $K$ components.
    Then Condition \ref{cond:part} holds for any  $k<\min(n,K)$, and so does the inconsistency of Theorem \ref{thm}.
\end{Prop}

The proofs of Propositions \ref{prop:Gibbs} and  \ref{prop:PYM} are provided in Appendix~\ref{an:proof}. Note that although the Dirichlet multinomial process is a particular case of the Pitman--Yor multinomial process and the normalized generalized gamma multinomial process, we include it as a separate case in the statement as the proof for this case differs from the proofs for its generalizations.

More precisely, the proof \cite[as in][Proposition 5]{miller2014inconsistency} consists in controlling the ratio of probability \smash{$\frac{1}{n}~ p(A)/p(B)$}, where $B = B(A,j)$ is defined in Section \ref{sec:Th}. 
For the Gibbs-type process, as the ratio of probability is raised by \smash{$(V_{n,k}/V_{n,k+1})$}, it is enough to show that the sequence \smash{$(V_{n,k}/V_{n,k+1})_{n\geq1}$} is bounded. Since there is no simple formula for $V_{n,k}$ in the general case of the Gibbs-type process, we prove this using a Laplace approximation.
The idea of the original proof of \cite{miller2014inconsistency} is the same but this ratio simplifies as they consider Pitman--Yor process.

For the Pitman--Yor multinomial process and the \textsc{NGG} multinomial process, the partition distribution depends on a sum over the vectors $\ell^{(k)}=(\ell_1,\ldots,\ell_k)$ such that $\ell_i\in\{1,\ldots,n_i\}$ and $|\ell^{(k)}| = \ell_1+\cdots+\ell_k$. 
We write this sum as $k$ different sums over each $\ell_i$. As in the nonparametric case, we consider the ratio of probability \smash{$\frac{1}{n}~ p(A)/p(B)$}. By definition of partition $B$, if $j\in A_k$ then the sum over $\ell_k$ is different for $p(A)$ and $p(B)$, one is of $n_k$ elements and the other of $n_k-1$ elements. We separate the sum of $n_k$ elements into two sums, the first one of $n_k-1$ elements and the second one of one element. 
In this way, we can use some known properties of the generalized factorial coefficients and some specific properties of each process to conclude.

\begin{figure}[ht]
    \centering
    \begin{tabular}{cccc}
        \includegraphics[width = 0.26\textwidth]{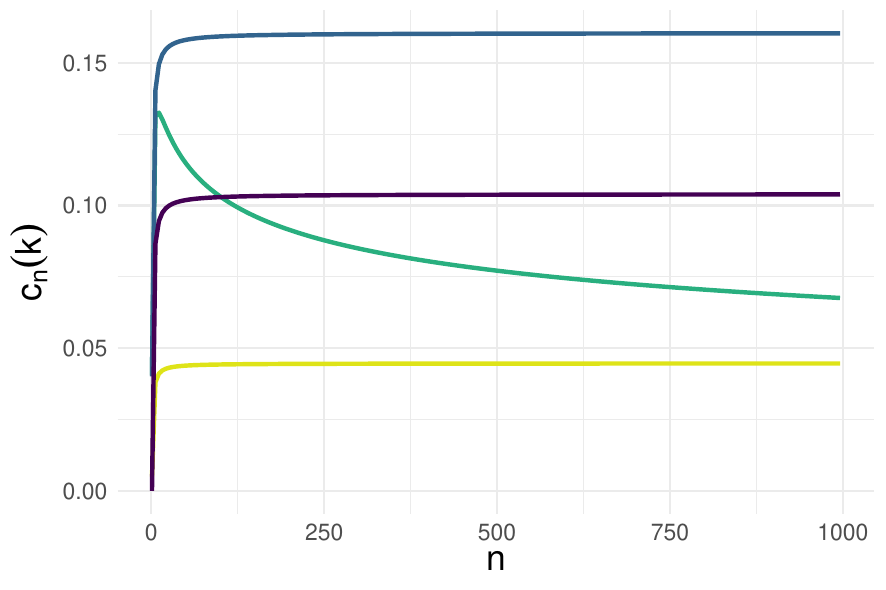} & \includegraphics[width = 0.26\textwidth]{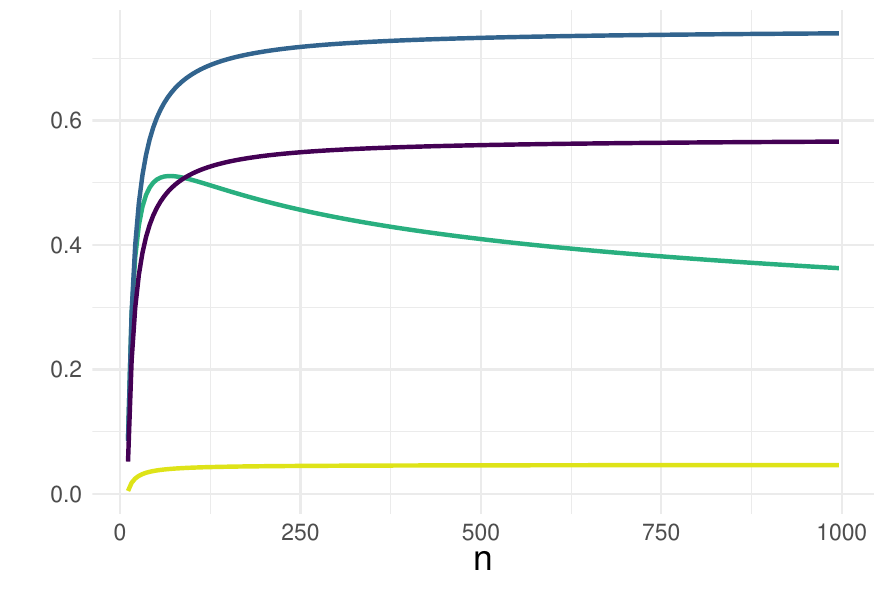} & \includegraphics[width = 0.26\textwidth]{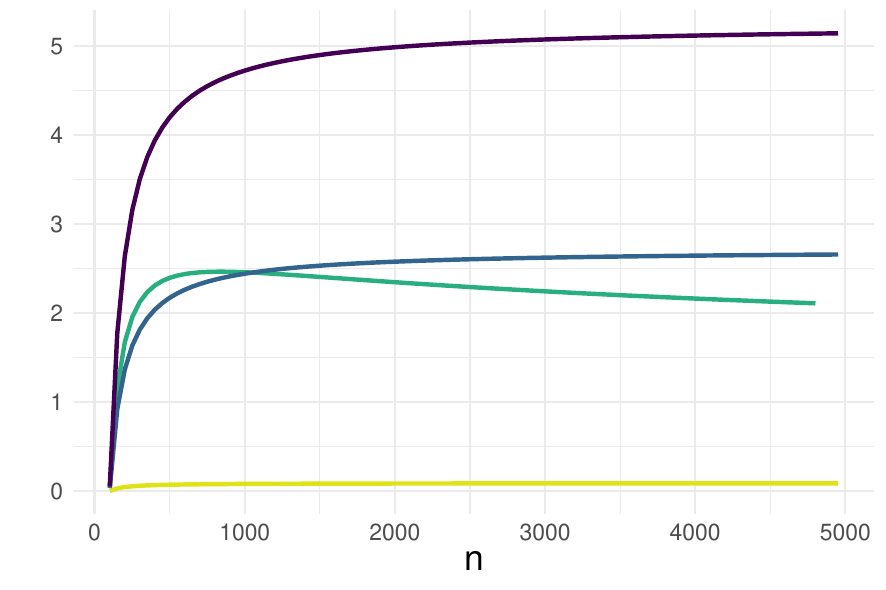}& \includegraphics[width = 0.1\textwidth]{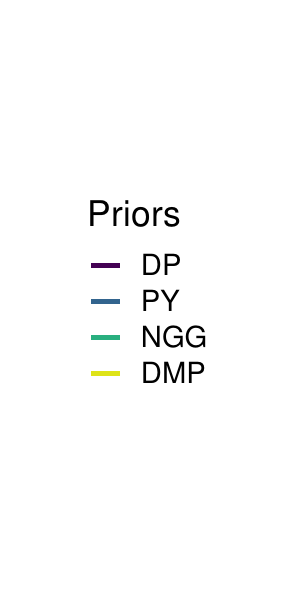}\\
        (a) $k=1$ & (b) $k=10$ & (c) $k=100$ \\
        \hspace{0.1cm}
        \includegraphics[width = 0.26\textwidth]{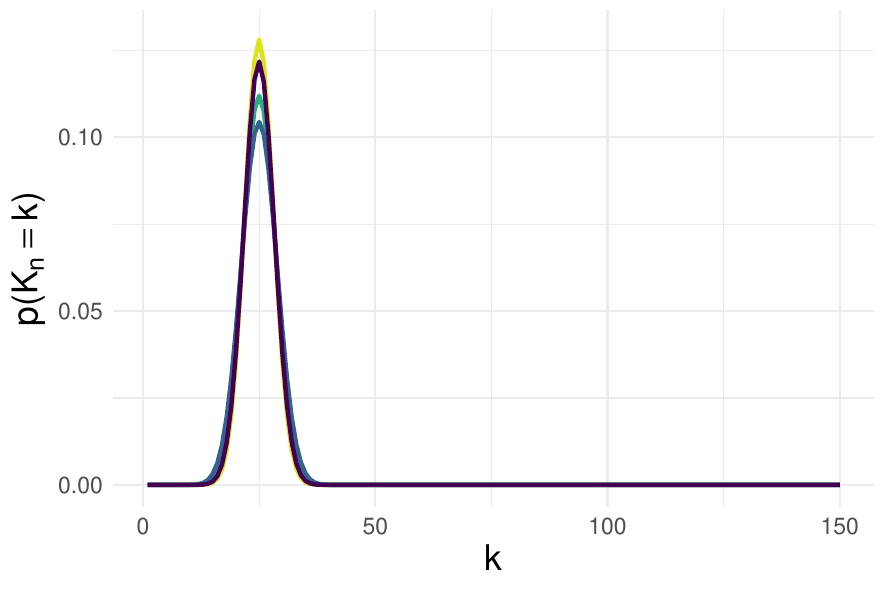} & \includegraphics[width = 0.26\textwidth]{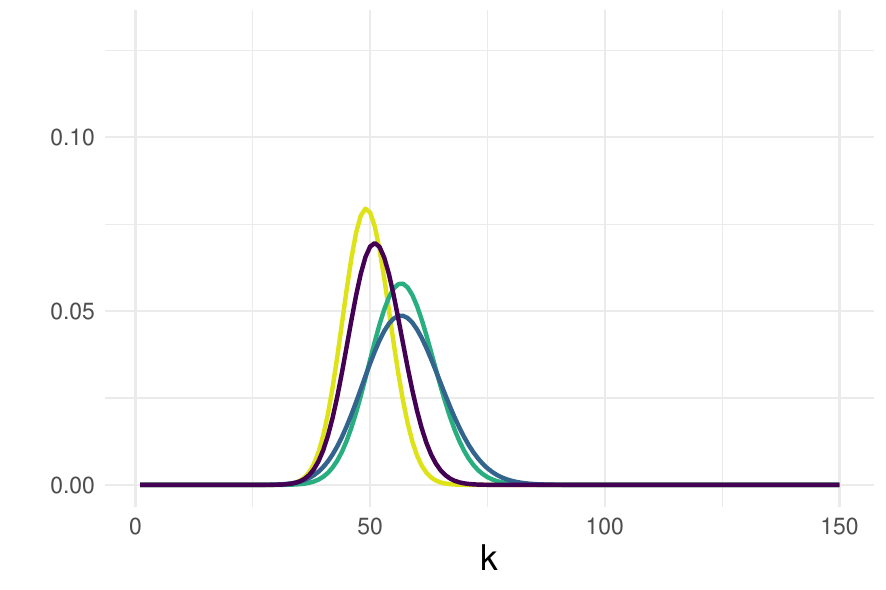} & \includegraphics[width = 0.26\textwidth]{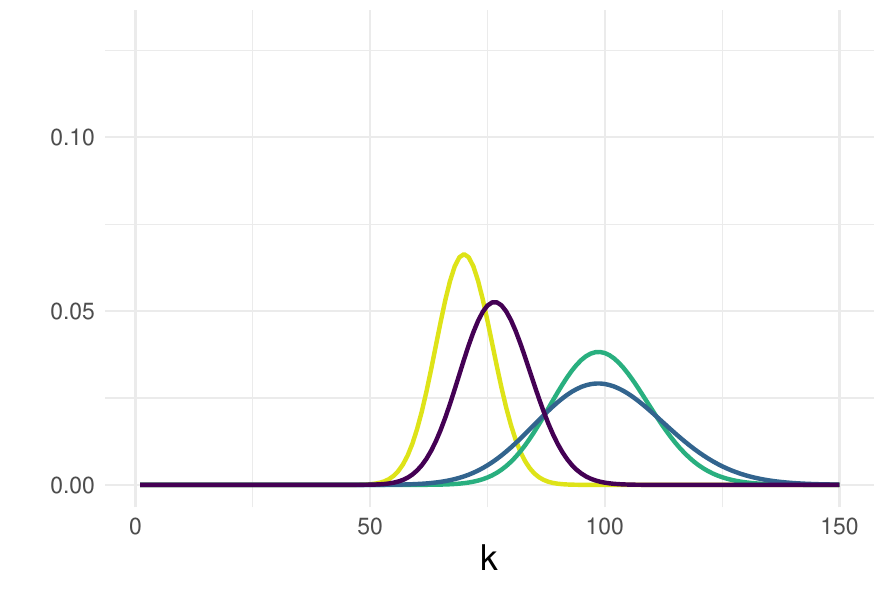} & \includegraphics[width = 0.1\textwidth]{figures/fig2.pdf}\\
        (d) $n=50$ & (e) $n=250$ & (f) $n=1000$
    \end{tabular}
    \caption{(\textit{Top row}) Illustrations of Condition~\ref{cond:part}, for $k\in\{1,10,100\}$: the function $n\mapsto c_n(k)$ reaches a plateau for large values of $n$ for a range of priors (see infra).\\
    (\textit{Bottom row}) Prior probability on the number of clusters for different processes and different values of $n$. In both rows, parameters are fixed such that $\mathbb{E}[K_{50}] = 25$: for Dirichlet process \textsc{DP}$(\alpha = 19.2)$, for Pitman--Yor process \textsc{PY}$(\sigma = 0.25, \alpha = 12.2)$, for \textsc{NGG} process \textsc{NGG}$(\sigma = 0.25, \beta = 48.4)$ and for Dirichlet multinomial process \textsc{DMP}$(\alpha = 22.5, K = 200)$. Illustrations are made using the package \href{https://github.com/konkam/GibbsTypePriors}{\texttt{GibbsTypePriors}}.
    }
    \label{fig:cond1}
\end{figure}

The top row of Figure \ref{fig:cond1} illustrates Condition \ref{cond:part} for different partition distributions, such as the Dirichlet multinomial process (\textsc{DMP}), the Dirichlet process (\textsc{DP}), the Gibbs-type process for the normalized generalized gamma process (\textsc{NGG}) special case and the Pitman--Yor process (\textsc{PY}). In all these cases, we represent the function $c_n(k)$ defined in Section \ref{sec:inconsistency} for different values of $k$, $k\in\{1,10,100\}$, with $n\in\{1,\ldots,5000\}$ and for some fixed parameters chosen such that \smash{$\mathbb{E}[K_{50}] = 25$}. 
We draw all the priors we considered for this choice of the parameters in Figure \ref{fig:cond1} bottom row. We also illustrate how the priors vary depending on $n$, fixing the priors parameters such that $\mathbb{E}[K_{50}] = 25$ then we made $n$ varying, $n \in \{50,250,1000\}$.
In Figure \ref{fig:cond1} top row, we can see $n\mapsto c_n(k)$ function reaches a plateau, thus indicating its boundedness for every process and values of $k$.

\section{Consistency results}
\label{sec:consistency}
\textcolor{forestgreen}{The previous results imply that the posterior distribution of the number of clusters for some Bayesian nonparametric mixture models and some overfitted mixture models is inconsistent and thus do not provide good estimates for the number of components in a finite mixture.
In these cases, the posterior distribution of the number of clusters is not the relevant summary to consider.}
Instead, results by \cite{rousseau2011asymptotic, nguyenConvergenceLatentMixing2013, scricciolo2014adaptive} suggest that it might be better to focus on the \textcolor{forestgreen}{latent} mixing measure.
In particular, recent works on consistency can be extended to the models we consider. 
In this \textcolor{forestgreen}{section}, we consider the framework of \cite{rousseau2011asymptotic} \textcolor{forestgreen}{for overfitted mixtures} and investigate to which extent it might apply to some models we have been considering, the Dirichlet multinomial process and Pitman--Yor multinomial process mixture models.
\textcolor{forestgreen}{Moreover} \cite{guha2021posterior} introduce a post-processing procedure, the Merge-Truncate-Merge (MTM) algorithm, for which the output, the number of clusters, is consistent. \cite{guha2021posterior} proved that this algorithm can be applied to the Dirichlet process mixture model so that there is consistency for the number of clusters after applying this algorithm. 
We extend this result and prove that we can apply the algorithm to overfitted mixture models and the Pitman--Yor process mixture model. 

\subsection{Emptying extra clusters}\label{sec:mixing-weight}
Overfitted mixtures can be constructed based on the Dirichlet multinomial process or the Pitman--Yor multinomial process. 
\cite{rousseau2011asymptotic} show in their Theorem 1 that overfitted mixtures, under some conditions on the kernel and the mixture model, have the desirable property that in the mixing measure the weights of extra components tend to zero as the sample size grows. 
This result only concerns the weights and not the number of clusters, but a near-optimal posterior contraction rate for the mixing measure can be deduced from it \citep[see section 3.1 in][]{guha2021posterior}.
To be more precise, \cite{rousseau2011asymptotic} consider a prior $\pi$ on the mixture weights $w$ written as follows
\[\pi(w) = C(w) w_1^{\alpha_1-1}\cdots w_k^{\alpha_k-1},\]
with specific properties for the function  $C(w)$  recalled in Condition \ref{hyp:5}.
Two types of prior hyper-parameter configurations are studied, which lead to opposite conclusions: merging or emptying of extra components. 
Let $d$ be the dimension of the component-specific parameter $\theta$. 
If $\alpha_\mathrm{max} = \max_j(\alpha_j)$ is such that $\alpha_\mathrm{max}<d/2$, then the posterior expectation for the weights of the extra components tends to zero. This is the case where extra components are emptied. 
The other case corresponds to $\alpha_\mathrm{min}= \min_j(\alpha_j)>d/2$. In this case, the extra components are merged with non-negligible weight, which means that they become identical to an existing component and inadvertently borrow some of its weight. This case is less stable as there are different merging possibilities. 
It is therefore preferable to choose parameters of the prior that belong to the first case. 
The result stated in Theorem 1 in \cite{rousseau2011asymptotic}, depends on five conditions. The first one, Condition \ref{hyp:1} below, is a posterior contraction condition on the mixture density. The following three conditions, Condition~\ref{hyp:2}, Condition~\ref{hyp:3}, and Condition~\ref{hyp:4} in Appendix \ref{an:bg}, are standard conditions on the kernel density, respectively on regularity, integrability, and strong identifiability. Finally, Condition \ref{hyp:5} below represents a classic continuity property for the prior density. 
More details on this result are provided in Appendix \ref{an:bg} where the assumptions on the kernel are recalled and Theorem 1 of \cite{rousseau2011asymptotic} is stated.

To apply Theorem 1 in \cite{rousseau2011asymptotic} to our case, as the kernel is not the focus of this article, we only need to check the conditions on the mixture model. We recall here these two conditions, Condition~\ref{hyp:1} and Condition~\ref{hyp:5}, which correspond to conditions respectively on the posterior contraction of the mixing measure and the prior density. 

\begin{Cond}[\citealp{rousseau2011asymptotic}, Assumption 1]
\label{hyp:1}
    There exists $\varepsilon_n \leq \log (n)^q/\sqrt{n}$, for some $q\geq0$, such that
    \[ \lim_{M\to\infty} \limsup_n \left\{ \mathbb{E}^n_0\left[\Pi(\|f^X-f_0^X\|_1\geq M\varepsilon_n\mid X_{1:n})\right]\right\} = 0, \]
    where $f_0^X$ is the true mixture density.
\end{Cond}

\begin{Cond}[\citealp{rousseau2011asymptotic}, Assumption 5]
\label{hyp:5}
    The prior density with respect to Lebesgue measure on the cluster-specific parameter $\theta$ is continuous and positive on $\Theta$, and the prior $\pi(w)$ on $w=(w_1,\ldots,w_K)$ satisfies
    \[\pi(w) = C(w) w_1^{\alpha_1-1}\cdots w_K^{\alpha_K-1},\]
    where $C(w)$ is a continuous function on the simplex bounded from above and from below by positive constants.
\end{Cond}

\begin{Prop}
\label{prop:over}
    Assume that the kernel considered satisfies Conditions \ref{hyp:2}, \ref{hyp:3}, and \ref{hyp:4} \textcolor{forestgreen}{(see Appendix \ref{an:bg})}. Let $G$ be a Dirichlet multinomial process. 
    Then, Conditions \ref{hyp:1} and \ref{hyp:5} are satisfied, and Theorem 1 of \cite{rousseau2011asymptotic} holds.
\end{Prop}
The proof of this proposition can be found in Appendix \ref{an:prop4}. It relies on Theorem 4.1 from \cite{rousseau2019bayesian} through which Condition \ref{hyp:1} holds for mixture models based on the Dirichlet multinomial process.
This theorem gives a result on density consistency for finite mixture models in the exact setting, which remains true in the overfitted mixture case. 
The proof in Appendix \ref{an:prop4} consists mainly of proving that Condition \ref{hyp:5} holds for the different priors we consider.


\textcolor{forestgreen}{
We have also studied the Pitman--Yor multinomial process, which is an interesting prior and a natural extension of the Dirichlet multinomial process. As an overfitted mixture, it could be expected that the result in \cite{rousseau2011asymptotic} would also apply to this prior.
However, even in the special case $\sigma=\frac{1}{2}$ where a prior density for the weights is available in closed form, it can be proven that Condition \ref{hyp:5} will never be satisfied. More precisely, there exists no $\alpha_1,\ldots,\alpha_K$ such that the function $C(w)$ defined in Condition~\ref{hyp:5} is bounded from above and below by positive constants. Hence, the \cite{rousseau2011asymptotic} framework cannot provide any guarantee for the Pitman--Yor multinomial process. Refer to Appendix~\ref{an:prop4} for a detailed description.
} 

\subsection{Merge-Truncate-Merge \textcolor{forestgreen}{algorithm}}\label{sec:MTM}
We assume throughout this section as in \cite{guha2021posterior} that the parameter space $\Theta$ is compact. 
We denote by $W_r(\cdot,\cdot)$ the Wasserstein distance of order $r$, $r\geq1$. 
We recall in Theorem \ref{thm:cons} the following result by \cite{guha2021posterior}.
\begin{Thm}[\citealp{guha2021posterior}, Theorem 3.2.]
\label{thm:cons}
    Let $G$ be a posterior sample from the posterior distribution of any Bayesian procedure, namely $\Pi(\cdot\mid X_{1:n})$ such that for all $\delta>0$ 
    \[\Pi\left(G\,:\, W_r(G,G_0)\leq\delta\omega_n\mid X_{1:n}\right)\overset{p_{G_0}}{\longrightarrow} 1,\]
    with $\omega_n=o(1)$ a vanishing rate, $r\geq1$.
    Let $\Tilde{G}$ and $\Tilde{K}$ be the outcome of the Merge-Truncate-Merge algorithm \citep[][]{guha2021posterior} applied to $G$. Then the following holds as $n\to \infty$: 
    \begin{enumerate}[label=(\alph*)]
        \item $\Pi(\Tilde{K}=K_0\mid X_{1:n})\longrightarrow 1$ in $P_{G_0}$-probability.
        \item For all $\delta>0$, $\Pi(G\,:\, W_r(\Tilde{G},G_0)\leq\delta\omega_n\mid X_{1:n})\longrightarrow1$ in $P_{G_0}$-probability.
    \end{enumerate}
\end{Thm}
The Merge-Truncate-Merge algorithm is described in Appendix \ref{an:bg}.

\begin{Prop}[Pitman--Yor process]
    \label{prop:consPY}
    Let $G$ be a posterior sample from the posterior distribution of a Pitman--Yor process mixture. Under conditions of Lemma \ref{lem:Contract}, Theorem \ref{thm:cons} applies to $G$.
\end{Prop}

\begin{Prop}[Overfitted mixtures]
    \label{prop:consO}
    Let $G$ be a posterior sample from the posterior distribution of an overfitted mixture. Under conditions of second-order identifiability and uniform Lipschitz continuity of the kernel \citep{nguyenConvergenceLatentMixing2013, ho_strong_2016}, Theorem \ref{thm:cons} applies to $G$ with $r\leq2$.
\end{Prop}

To prove Proposition \ref{prop:consPY}, we introduce a lemma which derives from Theorem 1 in \cite{scricciolo2014adaptive}.  
We assume a location mixture with a known scale parameter $\tau_0$, as stated in Equation~\eqref{eq:loc_mix} in Appendix \ref{an:bg}. The location parameter is univariate, $\Theta\subset\mathbb{R}$.
There are three standard conditions, described in Appendix \ref{an:bg} as Condition \ref{hyp:A1}, Condition \ref{hyp:A2} and Condition \ref{hyp:A3}, for the theorem. Condition \ref{hyp:A1} is a condition on the kernel density, Condition \ref{hyp:A2} is a tail condition on the true mixing distribution, and Condition \ref{hyp:A3} is a condition on the base measure.
Theorem 1 from \cite{scricciolo2014adaptive} is also recalled in Appendix \ref{an:bg}.
To state the lemma, we also need a condition on the kernel $f(\cdot\mid\theta)$. We suppose that for some constants $0<\rho<\infty$ and $0<\eta\leq2$, the Fourier transform $\hat{f}$ of $f(\cdot\mid\theta)$ satisfies $|\hat{f}(t)|\sim e^{-\left(\rho|t|\right)^\eta}$.
\begin{Lem}
\label{lem:Contract}
Assuming the model is a location mixture as in \textcolor{forestgreen}{Equation~\eqref{eq:loc_mix}}, the scale parameter $\tau_0$ is known and $\Theta\subset\mathbb{R}$ is bounded.
Under Conditions \ref{hyp:A1}, \ref{hyp:A2}, and \ref{hyp:A3}, with $G$ the posterior mixing measure of a Pitman--Yor process mixture model, with $\sigma\in[0,1)$, then for every $1\leq r<\infty$, there exists a sufficiently large constant $M$ and some $0<\eta\leq2$ such that
\[\Pi(G:W_r(G,G_0)\geq M\,\textcolor{forestgreen}{\log (n)}^{-1/\eta}\mid X^{(n)})\to 0 \text{ in } P_{G_0}\text{-probability.}\]
\end{Lem}
The proof of this lemma can be found in Appendix \ref{an:prop4}. This lemma is similar to Corollary 2 from \cite{scricciolo2014adaptive} which applies to the special case of the Dirichlet process. With this lemma, we can now prove Proposition \ref{prop:consPY}.
\begin{proof}[Proof of Proposition \ref{prop:consPY}.]
    Theorem \ref{thm:cons} holds when the posterior $G$ is such that for all $\delta>0$, there exists a vanishing rate $\omega_n$ such that
    \[\Pi(G\,:\, W_r(G,G_0)\geq\delta\omega_n\mid X_{1:n})\longrightarrow 0 \text{ in } P_{G_0}\text{-probability}.\]
    Under the conditions of Lemma \ref{lem:Contract}, we have 
    \[\Pi(G:W_r(G,G_0)\geq M\,\textcolor{forestgreen}{\log (n)}^{-1/\eta}\mid X_{1:n})\to 0 \text{ in } P_{G_0}\text{-probability,}\]
    so that $\delta\omega_n = M\left(\textcolor{forestgreen}{\log (n)}\right)^{-1/\eta}$.
    
    Hence, the consistency results of Theorem \ref{thm:cons} hold for a Pitman--Yor process mixture model satisfying the conditions of Lemma \ref{lem:Contract}.
\end{proof}

In the case of Proposition \ref{prop:consO}, we also need a contraction rate for the mixing measure of overfitted mixture models. To ensure the existence of a contraction rate, two conditions on the kernel are required. 
These conditions are described in Appendix \ref{an:bg} as Condition~\ref{hyp:id} and Condition~\ref{hyp:lip}. 
Let $G$ be the mixing measure of any overfitted mixture model.
It is known that under some conditions on the kernel, there exists a rate of contraction for $G$ \citep[see Equation (5)][]{guha2021posterior},
\begin{equation}
    \label{eq:contractO}
    \Pi(G\,:\, W_2(G,G_0)\gtrsim \left(\textcolor{forestgreen}{\log (n)}/n\right)^{1/4}\mid X_{1:n})\longrightarrow 0 \text{ in } P_{G_0}\text{-probability}.
\end{equation}
This rate can be suboptimal for some overfitted mixture models but is sufficient to prove Proposition \ref{prop:consO}.

\begin{proof}[Proof of Proposition \ref{prop:consO}.]
    The proof of Theorem \ref{thm:cons} is the same in the case of overfitted mixtures as in the Bayesian nonparametric case. This theorem holds when the posterior $G$ is such that for all $\delta>0$, there exists a vanishing rate $\omega_n$ such that
    \[\Pi(G\,:\, W_r(G,G_0)\geq\delta\omega_n\mid X_{1:n})\longrightarrow 0 \text{ in }P_{G_0} \text{-probability}.\]
    We use Equation~\eqref{eq:contractO} to conclude with $\delta\omega_n \leq \left(\textcolor{forestgreen}{\log (n)}/n\right)^{1/4}$ and $r=2$.
    
    Hence, the consistency results of Theorem \ref{thm:cons} hold for a Pitman--Yor process mixture model satisfying the conditions of Lemma \ref{lem:Contract}.
\end{proof}

The work of \cite{guha2021posterior} can be applied to different Bayesian procedures. The only condition is to have a contraction rate for the mixing measure under the Wasserstein distance. 
However, this condition is not easy to prove, here we prove it for the Pitman--Yor process but there is no direct generalization for Gibbs-type processes. 
In the overfitted mixtures case, there is a general contraction rate for the mixing measure under the Wasserstein distance \citep[see][]{nguyenConvergenceLatentMixing2013,ho_strong_2016}. This rate could be suboptimal for some procedures as it is an upper bound but it guarantees the consistency of the Merge-Truncate-Merge algorithm.


\section{Simulation study}
\label{sec:simulation}
We consider a simulation study to illustrate the three parts of our theoretical results pertaining to (i) inconsistency of the posterior distribution of $\tilde K_n$ (Section~\ref{sec:incons-gibbs}), (ii) emptying of extra clusters (Section~\ref{sec:mixing-weight}), and (iii) the Merge-Truncate-Merge algorithm (Section~\ref{sec:MTM}). We study the familiar case of a Dirichlet multinomial mixture of multivariate normals. The simulated data was generated using a Gaussian location mixture, with a  parameter setting similar to the one of \cite{guha2021posterior} for the Dirichlet Process. 
More precisely, we have $K_0=3$ clusters and Gaussian kernels such that:
\begin{equation*}
    f_0^X(x) = \sum_{i=1}^3 w_i \mathcal{N}(x\mid \mu_i, \Sigma),
\end{equation*}
where $w= (w_1,w_2,w_3)$ are the weights, which we fix as $w= (0.5, 0.3, 0.2)$, and $N(x \mid \mu_i, \Sigma)$ is a multivariate Gaussian distribution with mean $\mu_i$ and covariance matrix $\Sigma$. We considered the following parameters for the mean and the covariance matrix:
\begin{equation*}
    \mu_1 = (0.8,0.8), \mu_2 =(0.8, -0.8) , \mu_3 =(-0.8, 0.8) \text{ and }\Sigma = 0.05I_{2}.
\end{equation*}
Here, the dimension of the kernel parameter $\theta = (\mu,\Sigma)$ is $d=5$ (2 for $\mu$ and 3 for $\Sigma$). 
In this setting, we generated a sequence of datasets with  $n=\{20, 200, 2000, 20000\}$, such that the smaller datasets are subsets of the larger ones. The number of components of the Dirichlet multinomial process is set to $K=10$, thus satisfying the overfitted condition $K\geq K_0$. 
 We chose the maximum parameter of the Dirichlet distribution, $\Bar{\alpha} = \alpha/K$, according to the intuition of \cite{rousseau2011asymptotic}  results. To obtain vanishing weights for extra components, the parameter  $\bar  \alpha$ should be less than $d/2=2.5$. We consider the following values: $\Bar{\alpha} \in \{0.01, 1, 2.5, 3\}$. 
We used the Markov chain Monte Carlo (MCMC) sampler proposed by \cite{malsiner-walli_model-based_2016}\footnote{The code is available at \href{https://github.com/dbystrova/BNPconsistency}{https://github.com/dbystrova/BNPconsistency}.}. Although the proposed algorithm allows us to use a hyperprior on the parameter $\alpha$ and shrinkage priors on the component means, we have used the basic version with standard priors on parameters. \textcolor{forestgreen}{See details on the number of iterations and simulation practical information in Appendix~\ref{an:simulation}}. Two situations are considered. In the first case, the prior expected number of clusters is fixed, which leads to decreasing parameter $\alpha$ at a rate asymptotically equivalent to $\log(n)^{-1}$.  In the second case, we introduce a prior distribution on $\Bar{\alpha}$.

\paragraph{Posterior inconsistency on $K_n$.} 
In Figure \ref{fig:DPM-prior-and-posterior}, we present the posterior distribution of the number of clusters for different values of parameter $\Bar{\alpha}$ and different sizes of the dataset $n$. In addition, we present the prior distribution on the number of clusters for the corresponding $\Bar{\alpha}$ and $n$. Table~\ref{Tab:prior_kn} summarizes the values of the parameters $\Bar{\alpha}$ and sample sizes $n$ used in the simulation study and displays the associated prior and posterior expected number of clusters $K_n$. 
As proved in Proposition \ref{prop:PYM}, the posterior distribution diverges with $n$. 
This \textcolor{forestgreen}{lack of concentration} is visible for three of the considered values $\Bar{\alpha} \in \{1, 2.5, 3\}$ in our experiments. For $\Bar{\alpha} = 0.01$, the posterior distribution stays concentrated around the true value $K_0= 3$ for the range of sample sizes $n$.
Interestingly, Figure \ref{fig:DPM-prior-and-posterior} makes it clear that the prior with fixed $\Bar{\alpha}$ puts increasing mass towards $K_n=10$ as the sample size grows, which is probably one of the root causes for posterior inconsistency. 
Allowing $\Bar{\alpha}$ to vary, as investigated on Figure \ref{fig:DPM-varying-alpha}, induces a much less informative prior on the number of clusters and the posterior deterioration as the sample size grows appears much less striking.


\begin{table}[H]
 \caption{\label{Tab:prior_kn}
 Prior and posterior expected number of clusters $K_n$ for the various values of $\Bar{\alpha}$ considered in our experiments.}
\centering
\footnotesize
\begin{tabular}{ScScScScScScScScSc}
    \toprule
    \multirow{2}{*}[-0.3em]{$n$} &
    \multicolumn{4}{c}{Prior $\mathbb{E}[K_{n}]$} &
    \multicolumn{4}{c}{Posterior $\mathbb{E}[K_{n}\vert X_{1:n}]$} 
    \\
    \cmidrule(l){2-9}
    {} &
    $\Bar{\alpha} =0.01$  &
    $\Bar{\alpha} =1$ &
    $\Bar{\alpha} =2.5$ &
    $\Bar{\alpha} =3$ 
    &
    $\Bar{\alpha} =0.01$  &
    $\Bar{\alpha} =1$ &
    $\Bar{\alpha} =2.5$ &
    $\Bar{\alpha} =3$
    \\ \midrule
    20 &
    1.3 &
    6.9 &
    7.9 &
   8
   &
    2.8 &
    4.9&
   5.9 &
   6.0
    \\ 
   200 &
   1.5 &
   9.6 &
   9.9 &
    9.98
    &
   3.04&
   6.9 &
   9.5 &
    9.7
    \\
    2000 &
   1.7 &
   9.9&
   $\approx$  10 &
    $\approx$ 10
    &
   3.07&
   8.1 &
   9.98 &
    9.99
    \\
     20000 &
   1.9 &
   9.99&
  $\approx$  10 &
   $\approx$  10
    &
   3.01&
   8.7 &
  $\approx$  10 &
   $\approx$  10
    \\ \bottomrule
\end{tabular}
\end{table}

\def\textlegendDPM{under a Dirichlet multinomial process mixture with fixed parameter $K=10$, and various choices of $\Bar{\alpha} = \alpha/K$ and $n$}

\paragraph{Emptying of extra clusters.} 
We are also interested to see how the posterior distribution of the component weights behaves in our simulation setting. Figure \ref{fig:DPM-weights} illustrates the posterior distribution of the weights of the components for different specifications of the parameter $\Bar{\alpha}$ and $n$, and is similar to Figure 1 and Figure 2 in \cite{rousseau2011asymptotic}. In our case, we sort the weights in decreasing order to 
alleviate the label-switching problem.  
For the minimal values of $\Bar{\alpha}=0.01$, we can see that the posterior weights with growing $n$ are concentrated at the true values of mixture weights, except the largest $n$. When $\Bar{\alpha}=1$, we can observe the concentration trend, but convergence is slower than in the first case. For $\Bar{\alpha}=2.5$ there are no clear dynamics. And for  $\Bar{\alpha}=3$  we can see that the weights become more uniformly distributed, which can be related to the merging weights regime. Specification of our simulation study does not allow to apply the \cite{rousseau2011asymptotic} theory directly, as in our case the support of $\theta$ is not bounded. However, we can see that the simulation results are still consistent with the theory, suggesting broader applicability.

\paragraph{Merge-Truncate-Merge.}
We applied the Merge-Truncate-Merge algorithm proposed by \cite{guha2021posterior} to the posterior distribution of the mixing measure in our simulation setting and illustrate the posterior distribution of the number of clusters $\Tilde{K}$ on Figure \ref{fig:DPM-MTM}. To use the Merge-Truncate-Merge algorithm, we need to know the Wasserstein convergence rates of the corresponding mixing measure. We use the convergence rate for overfitted mixtures  $\omega_n =  \left(\log(n)/n\right)^{1/4}$ \citep{guha2021posterior}. Note that for this convergence rate the prior on the kernel parameters should be bounded, which is not the case in our simulation (see details in Appendix \ref{an:simulation}), so as in the previous section, we apply Merge-Truncate-Merge out of its theoretically proven domain. The Merge-Truncate-Merge algorithm depends on the specification of a positive scalar $c$. As there is no explicit guideline for computing $c$, we tested a range of values $c \in \{0.1, 0.5, 1, 2\}$, see Figure~\ref{fig:DPM-MTM}. We can note that for each value of $n$, there exists some value of $c$ such that the posterior distribution of the number of clusters remains concentrated around the true number of components $K_0 =3$. At the same time, some values of $c$ are too restrictive or do not eliminate extra clusters. For example, $c=0.01$ for $\Bar{\alpha} =1$ does not allow the number of components to be correctly estimated. 
Conversely, too large a value of $c$ makes the Merge-Truncate-Merge algorithm also fail in the sense that it outputs zero values for $\tilde K$. This is \textcolor{forestgreen}{because the second step in the algorithm} truncates all clusters at once, \textcolor{forestgreen}{which corresponds to the case where the set $\mathcal{A}$ of the MTM algorithm recalled in Appendix \ref{an:bg} is empty and the set $\mathcal{N}$ contains everything}. 
This suggests interpreting $c$ as a regularization parameter, with the estimated number of clusters decreasing with increasing $c$. Following this intuition, we can draw (Figure \ref{fig:DPM-reg-path}) so-called ``regularization paths'' plots for $c$. More specifically, they represent the posterior mean and maximum a posteriori (MAP) for the posterior distribution of $\Tilde{K}$ for a range of values  $[0,c_\mathrm{max} ]$  for the parameter $c$, where $c_\mathrm{max}$ is defined as the value of $c$ for which the number of clusters given by the MTM algorithm $\tilde K$ is equal to $1$.  In other words, $c_\mathrm{max}$ coincides with the value where all the clusters have been merged or truncated by the MTM post-procedure into a single cluster. We can see that for all specifications of parameter $\Bar{\alpha}$ for large $n\geq 2000$, there always exists a region where the posterior mean and the MAP remain approximately constant (exactly constant for the MAP). 
This suggests a heuristic to use the Merge-Truncate-Merge algorithm: explore regularly spaced values in $[0, c_{\max}]$ and look for a plateau. In the absence of a plateau, the sample size should be increased.



\begin{figure}[H]
\centering
\begin{tabular}{cc}
  \includegraphics[width=.5\textwidth]{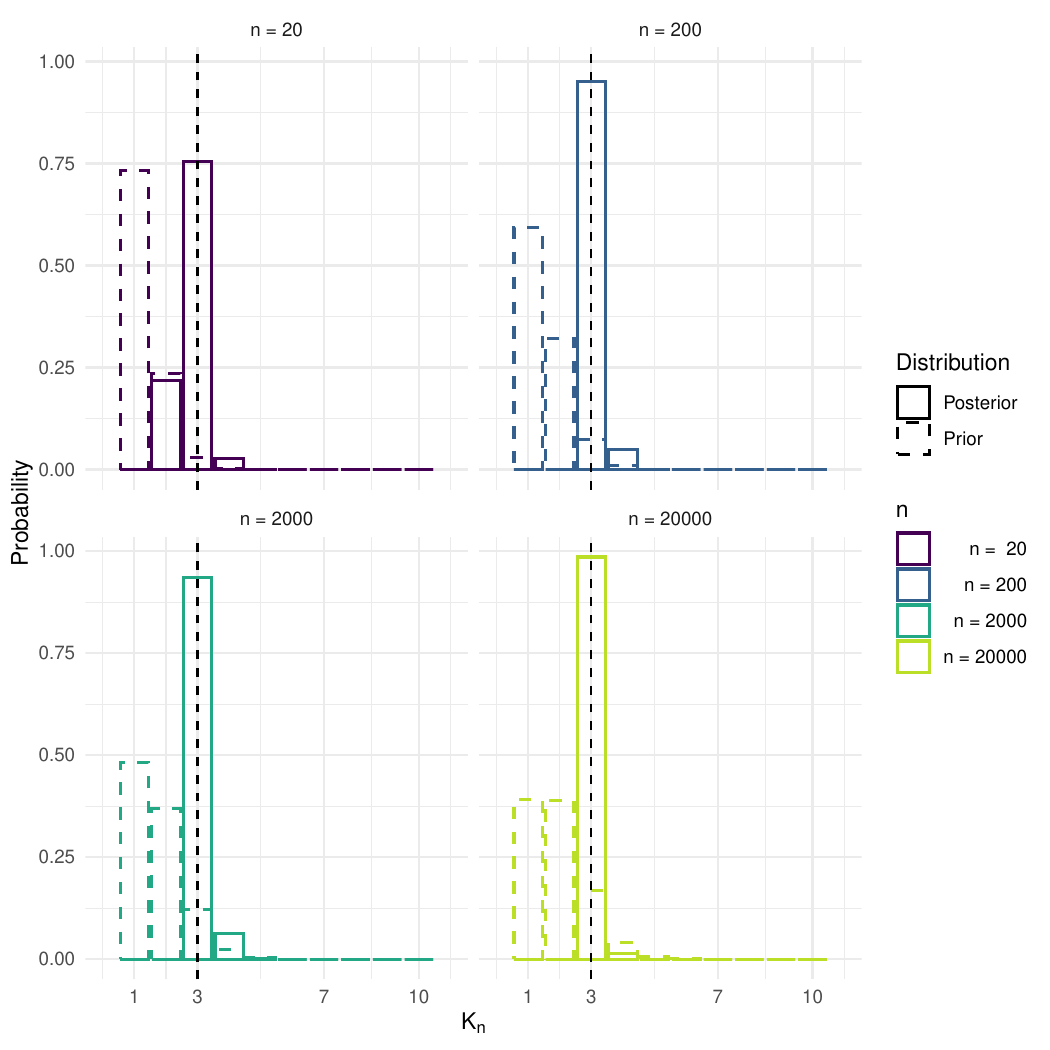} &   \includegraphics[width=.5\textwidth]{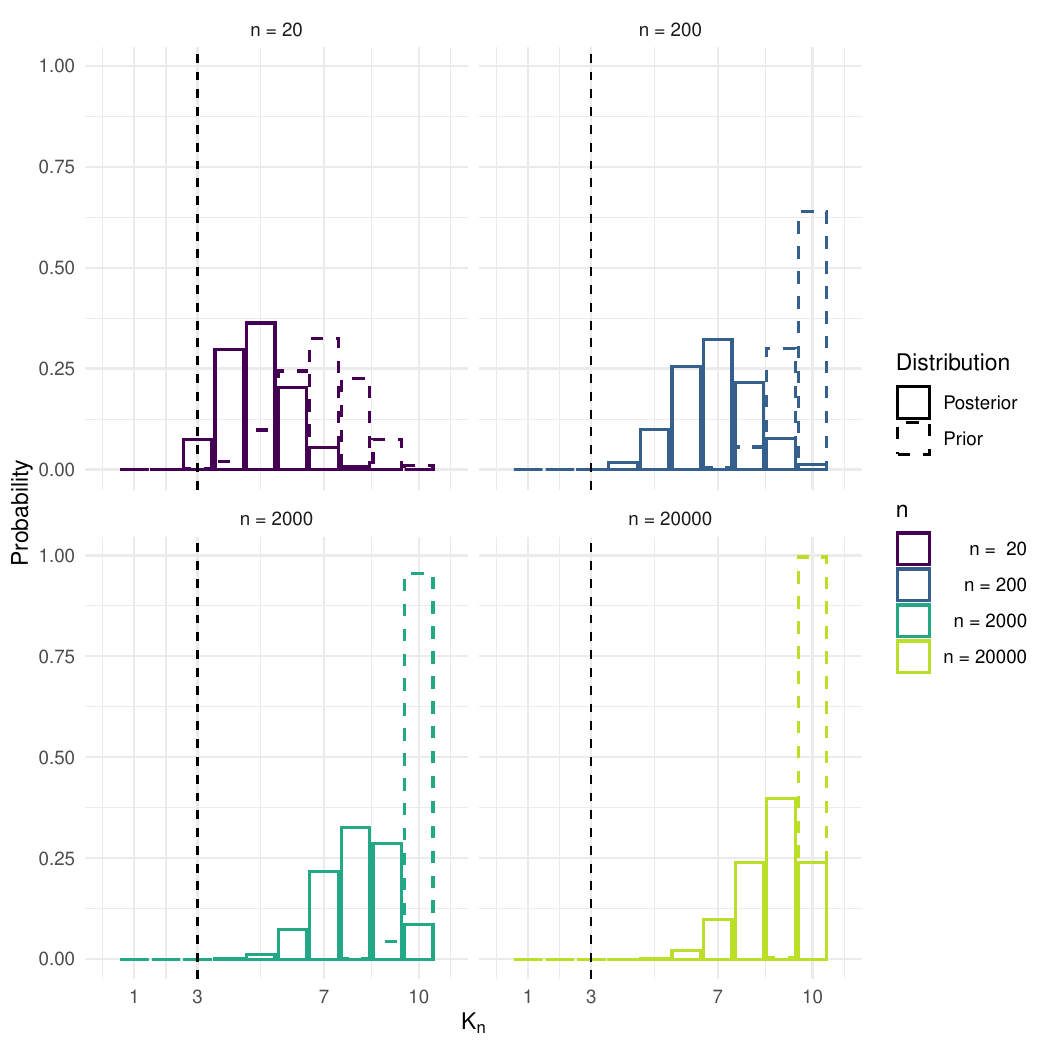}  \\
(a) Fixed $\Bar\alpha = 0.01$  & (b)  Fixed  $\Bar{\alpha} = 1$ \\[6pt]
 \includegraphics[width=.5\textwidth]{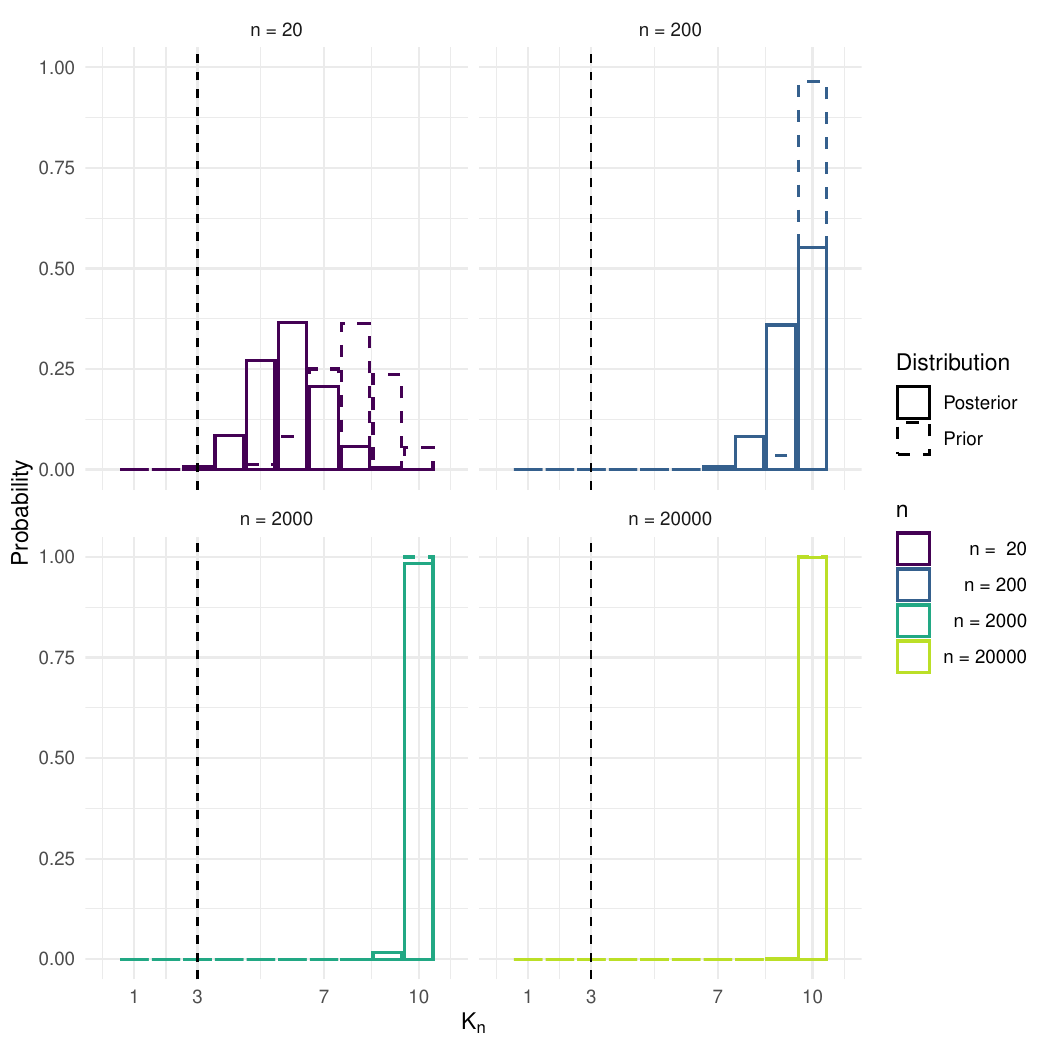} &
 \includegraphics[width=.5\textwidth]{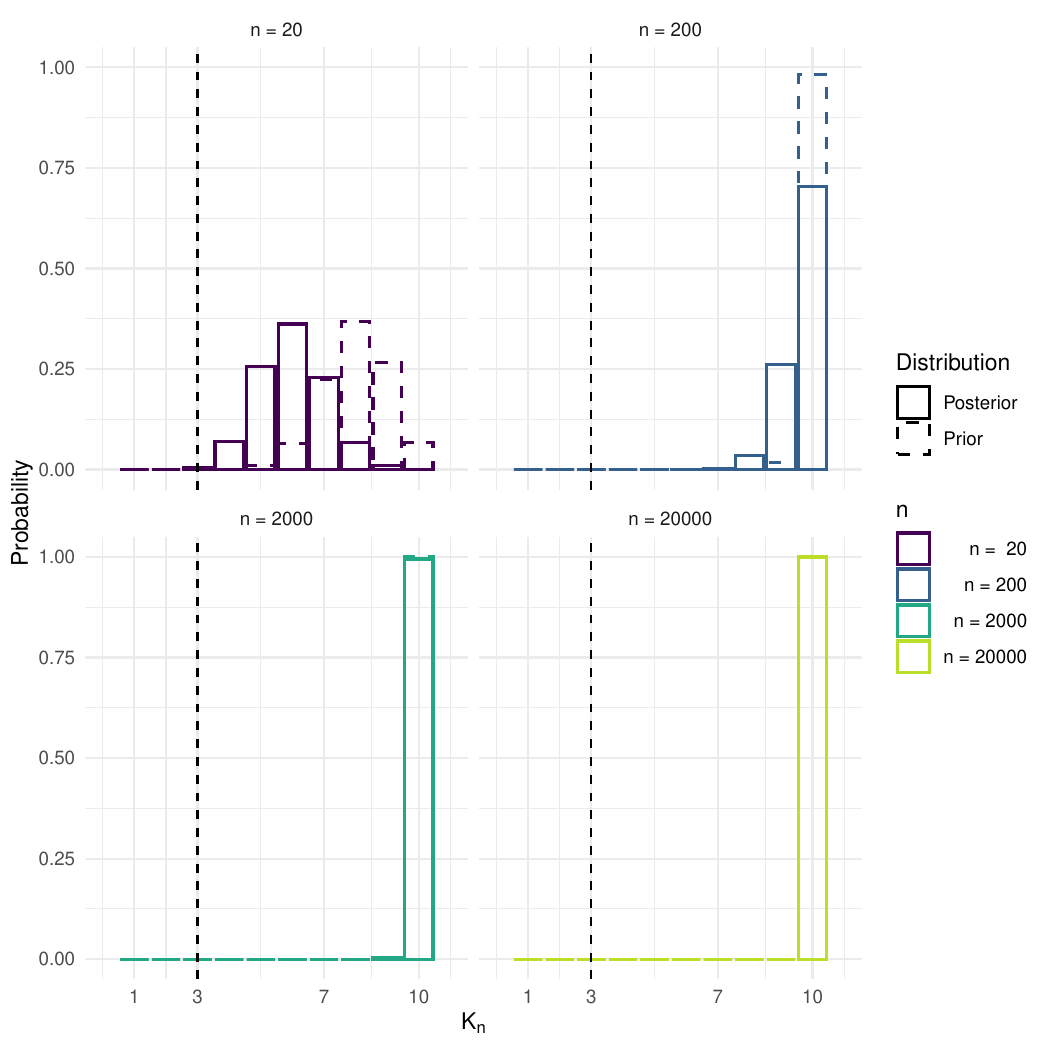} \\
 (c) Fixed $\Bar{\alpha} = 2.5$ & (d)   Fixed $\Bar{\alpha} = 3$  
\end{tabular}
\caption{Prior and posterior distributions of the number of clusters $K_n$ \textlegendDPM. The value $\Bar{\alpha} = 2.5$ corresponds to \cite{rousseau2011asymptotic}'s threshold.}
\label{fig:DPM-prior-and-posterior}
\end{figure}


\begin{figure}[H]
\centering
\begin{tabular}{cc}
  \includegraphics[width=.5\textwidth]{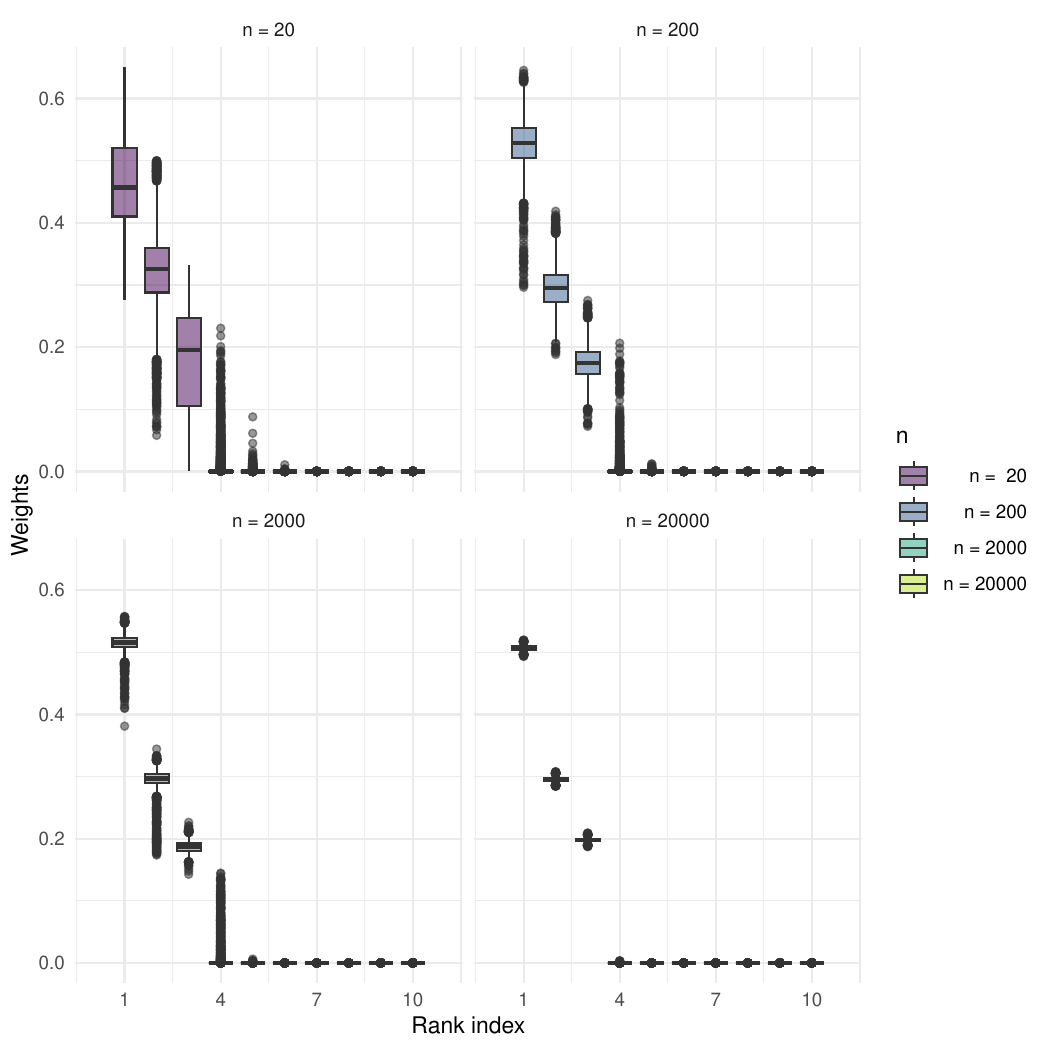} &   \includegraphics[width=.5\textwidth]{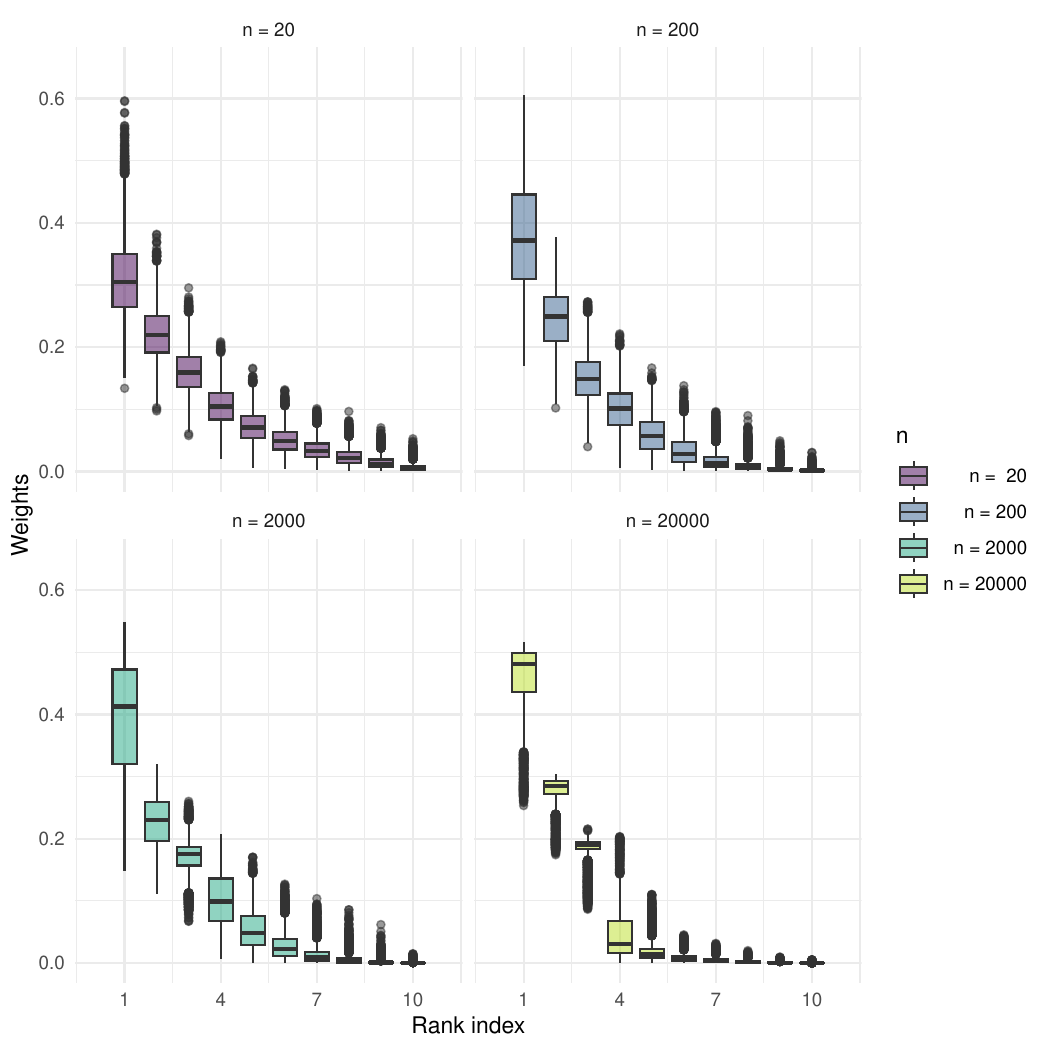}  \\
(a) Fixed $\Bar\alpha = 0.01$  & (b)  Fixed  $\Bar{\alpha} = 1$ \\[6pt]
 \includegraphics[width=.5\textwidth]{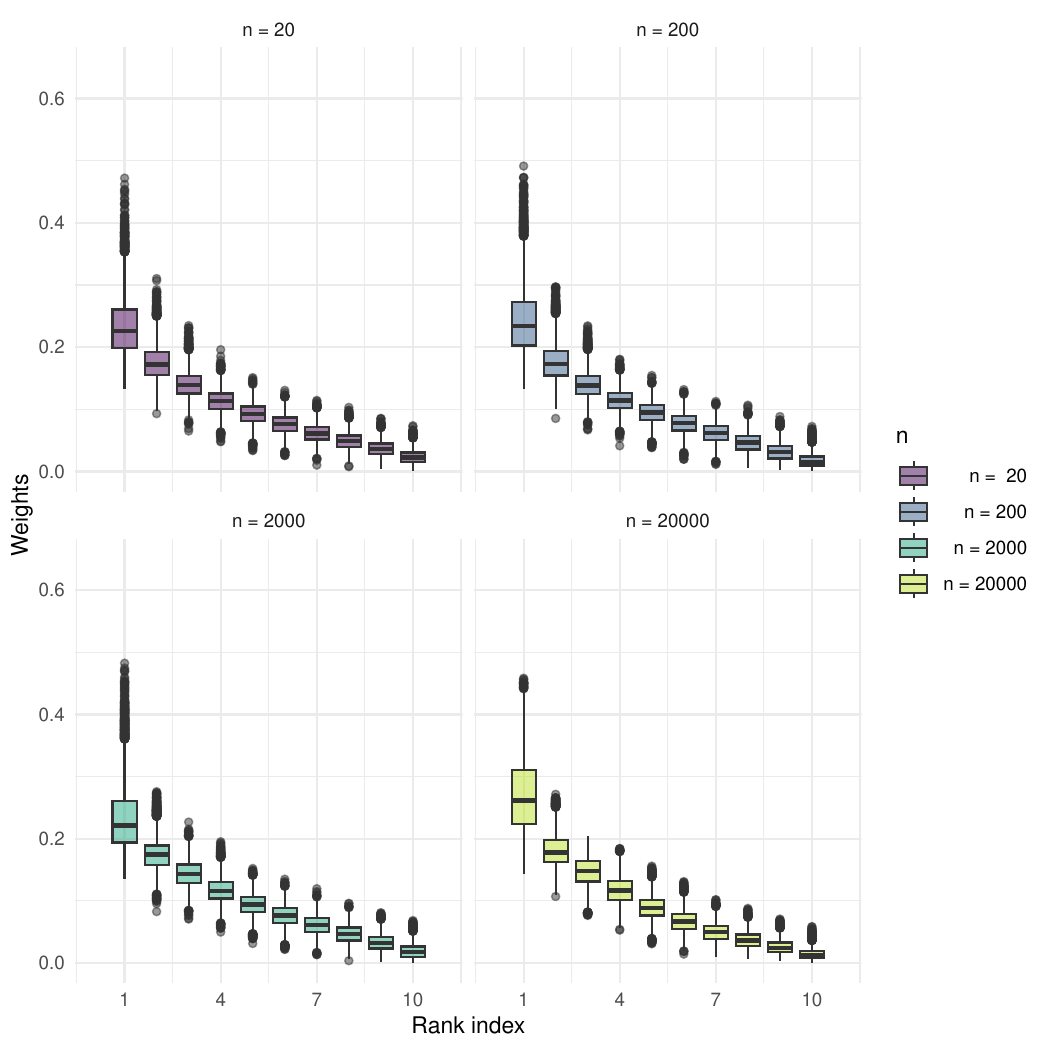} &
 \includegraphics[width=.5\textwidth]{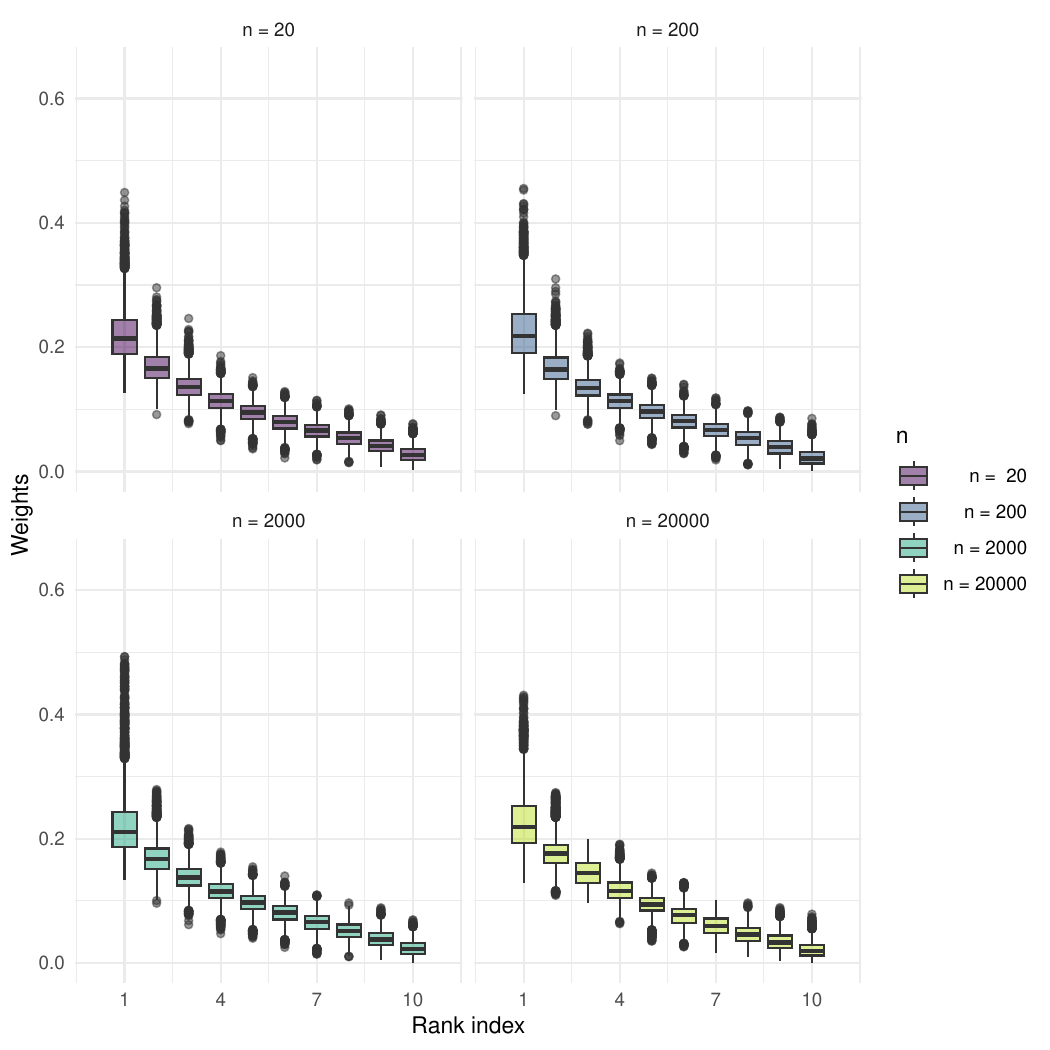} \\
 (c) Fixed $\Bar{\alpha} = 2.5$ & (d)   Fixed $\Bar{\alpha} = 3$  
\end{tabular}
\caption{Mixture weights \textlegendDPM. }
\label{fig:DPM-weights}
\end{figure}


\begin{figure}[H]
\centering
\begin{tabular}{cc}
  \includegraphics[width=.5\textwidth]{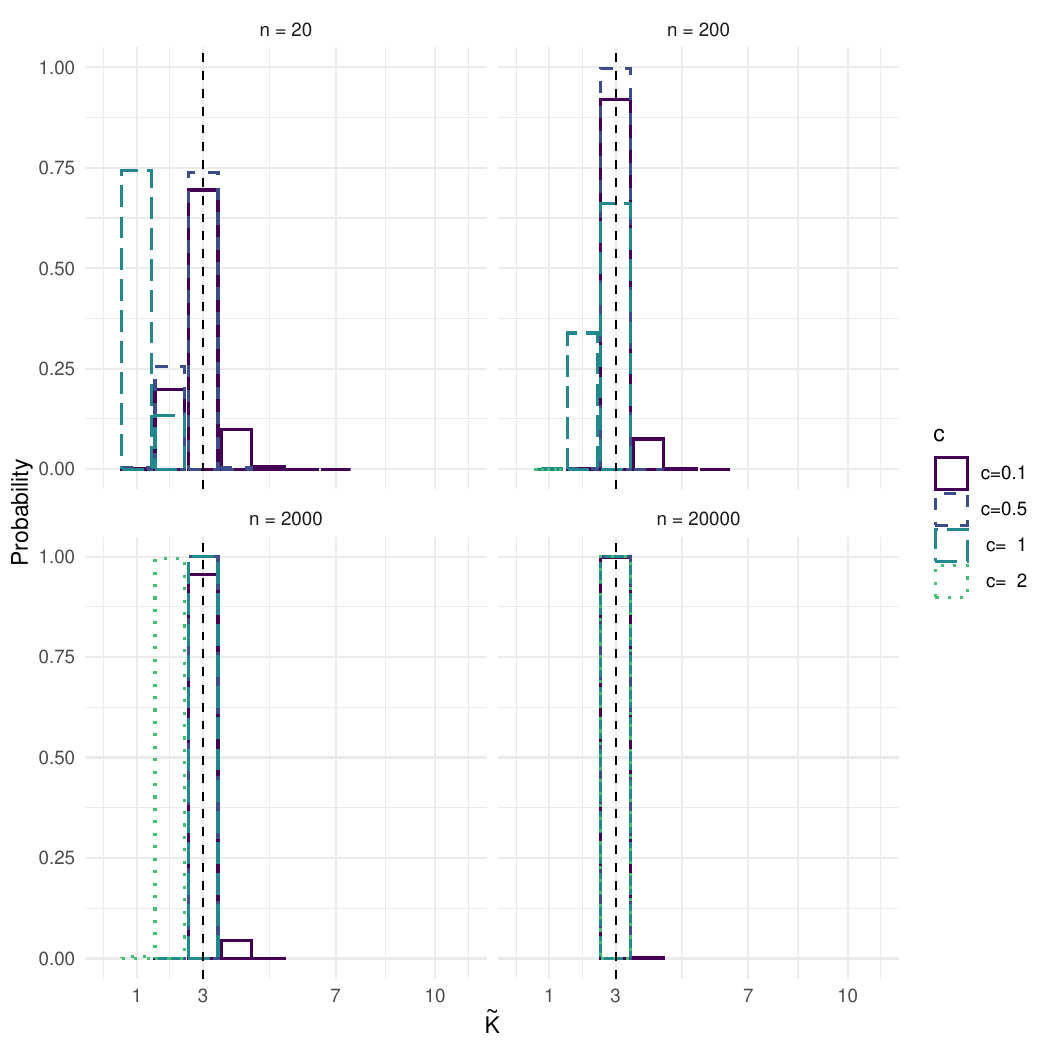} &   \includegraphics[width=.5\textwidth]{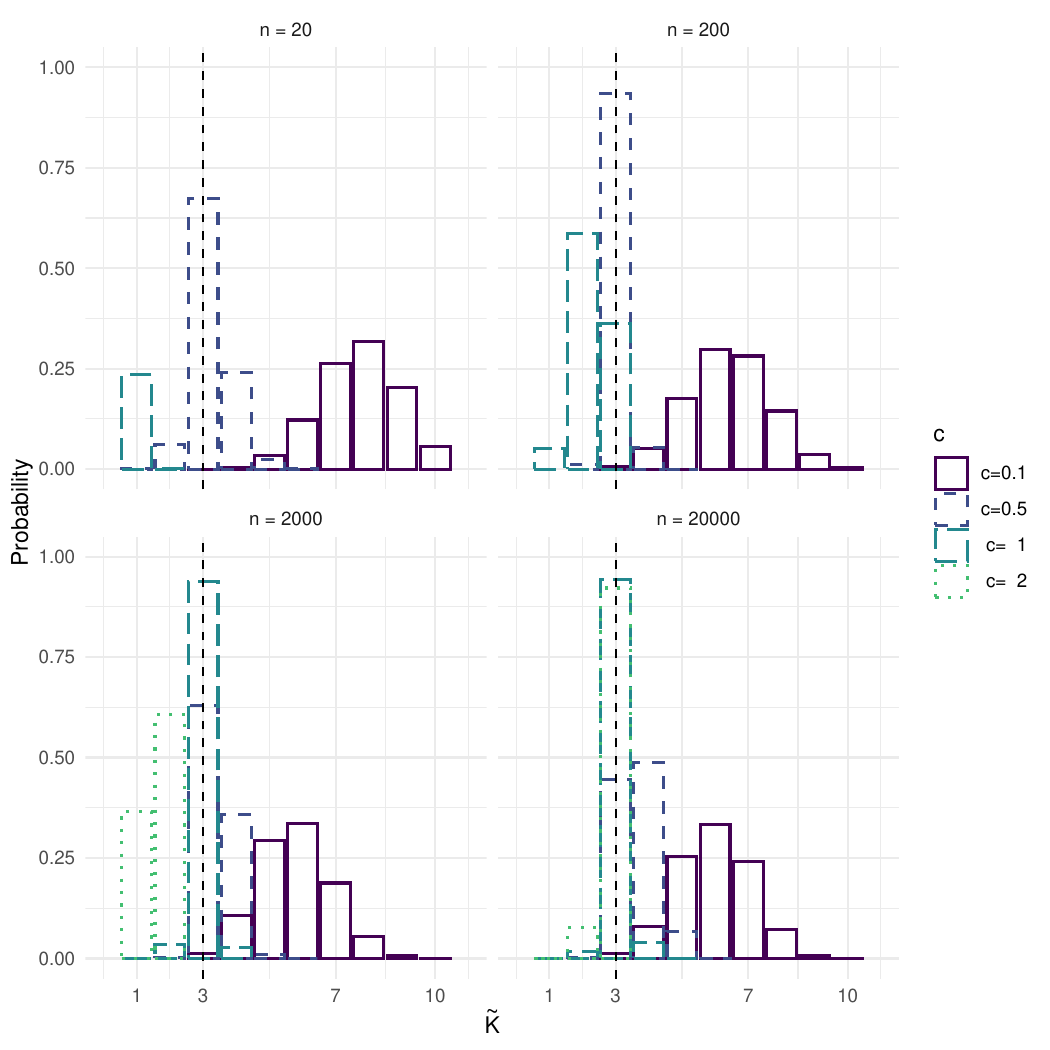}  \\
(a) Fixed $\Bar\alpha = 0.01$  & (b)  Fixed  $\Bar{\alpha} =1$ \\[6pt]
 \includegraphics[width=.5\textwidth]{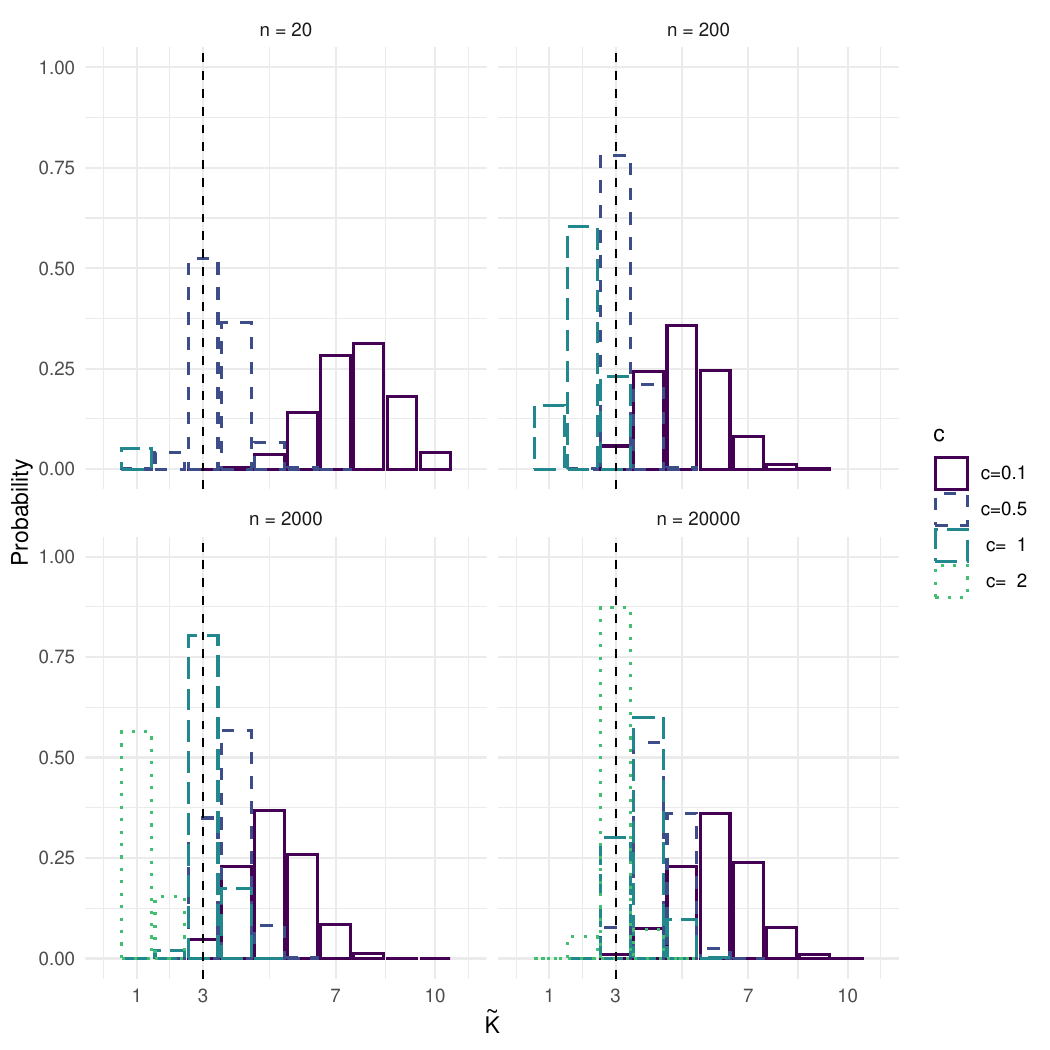} &
 \includegraphics[width=.5\textwidth]{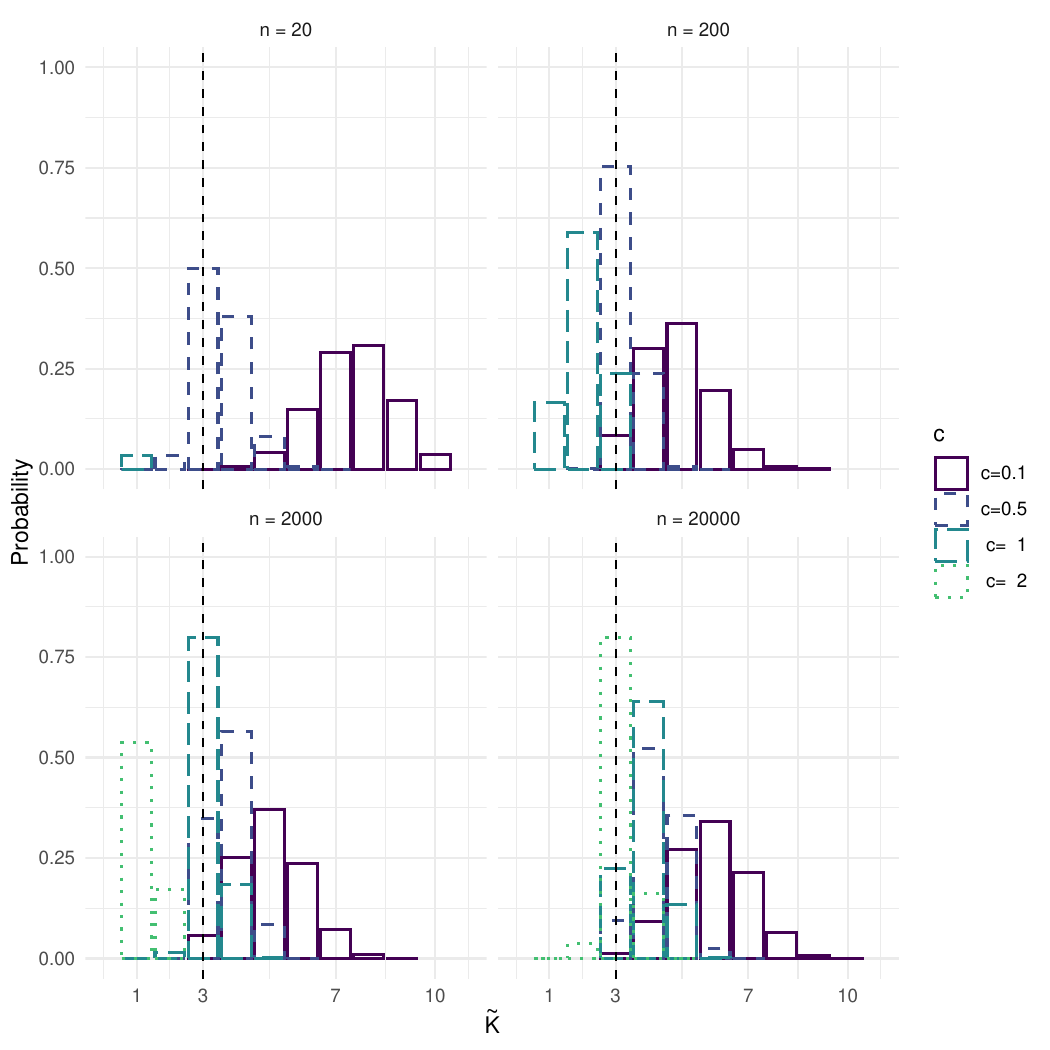} \\
 (c) Fixed $\Bar{\alpha} = 2.5$ & (d)   Fixed $\Bar{\alpha} = 3$  
\end{tabular}
\caption{Distribution of $\tilde K$, that is the posterior number of clusters after applying the Merge-Truncate-Merge algorithm of \cite{guha2021posterior}, with $c$ parameter in $\{0.1,0.5,1,2\}$, \textlegendDPM. }
\label{fig:DPM-MTM}
\end{figure}


\begin{figure}[H]
\centering
\begin{tabular}{cc}
  \includegraphics[width=.5\textwidth]{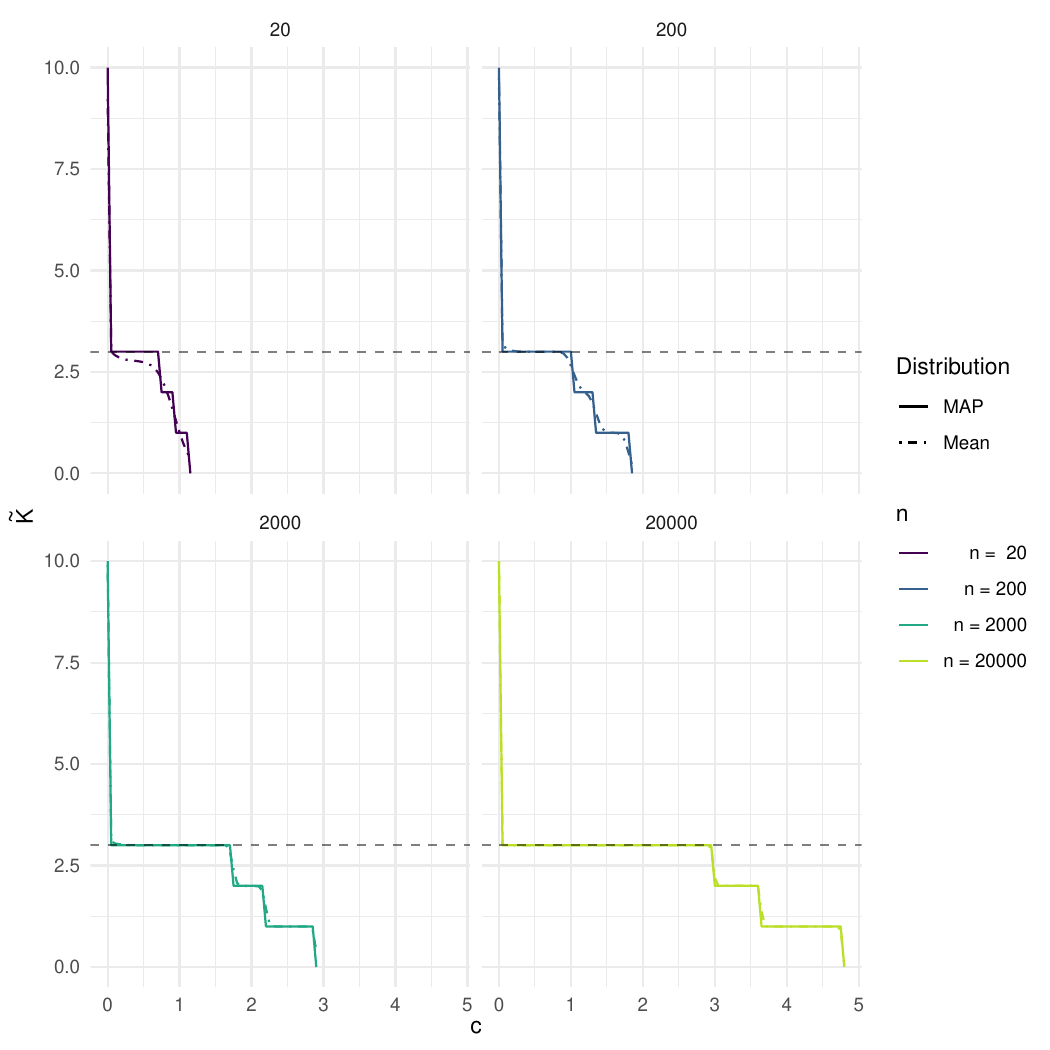} &   \includegraphics[width=.5\textwidth]{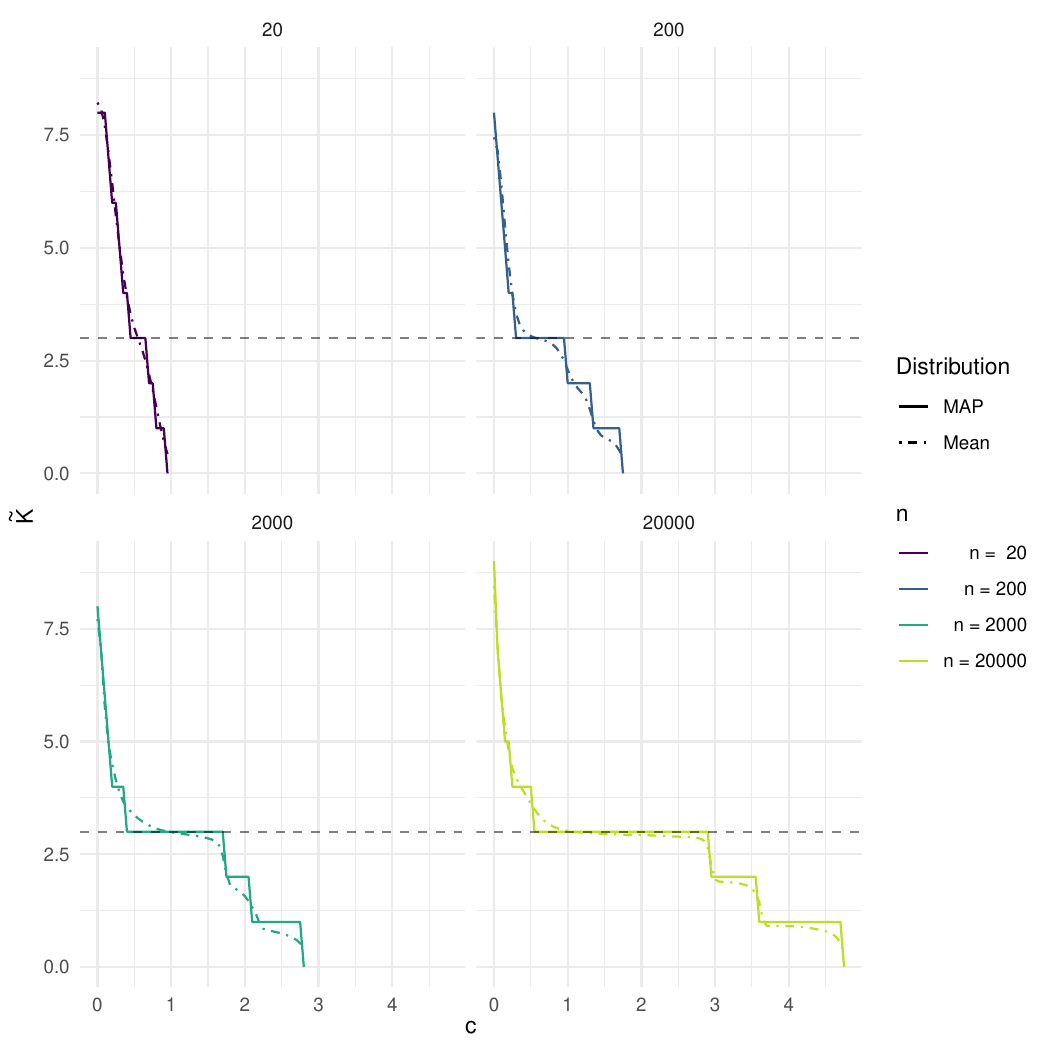}  \\
(a) Fixed $\Bar\alpha = 0.01$  & (b)  Fixed  $\Bar{\alpha} = 1$ \\[6pt]
 \includegraphics[width=.5\textwidth]{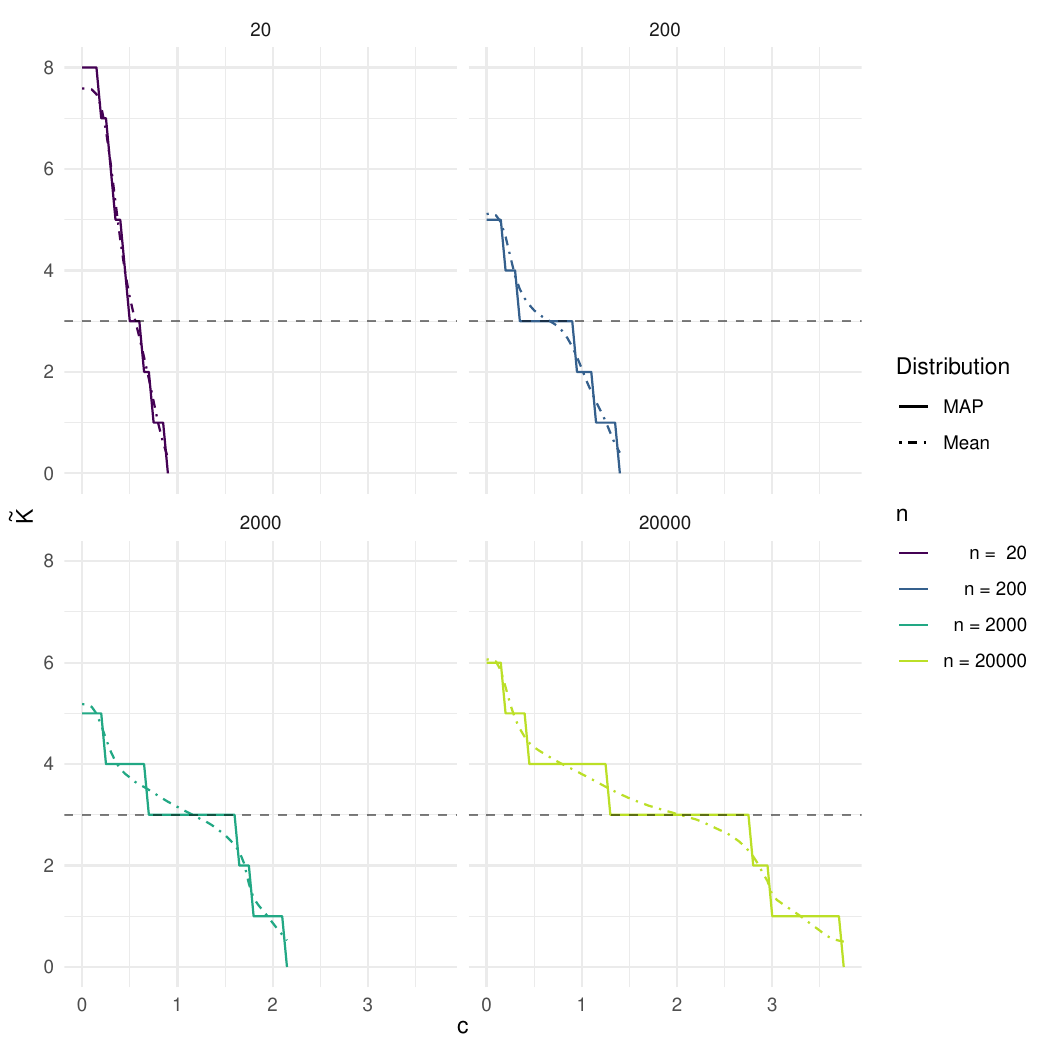} &
 \includegraphics[width=.5\textwidth]{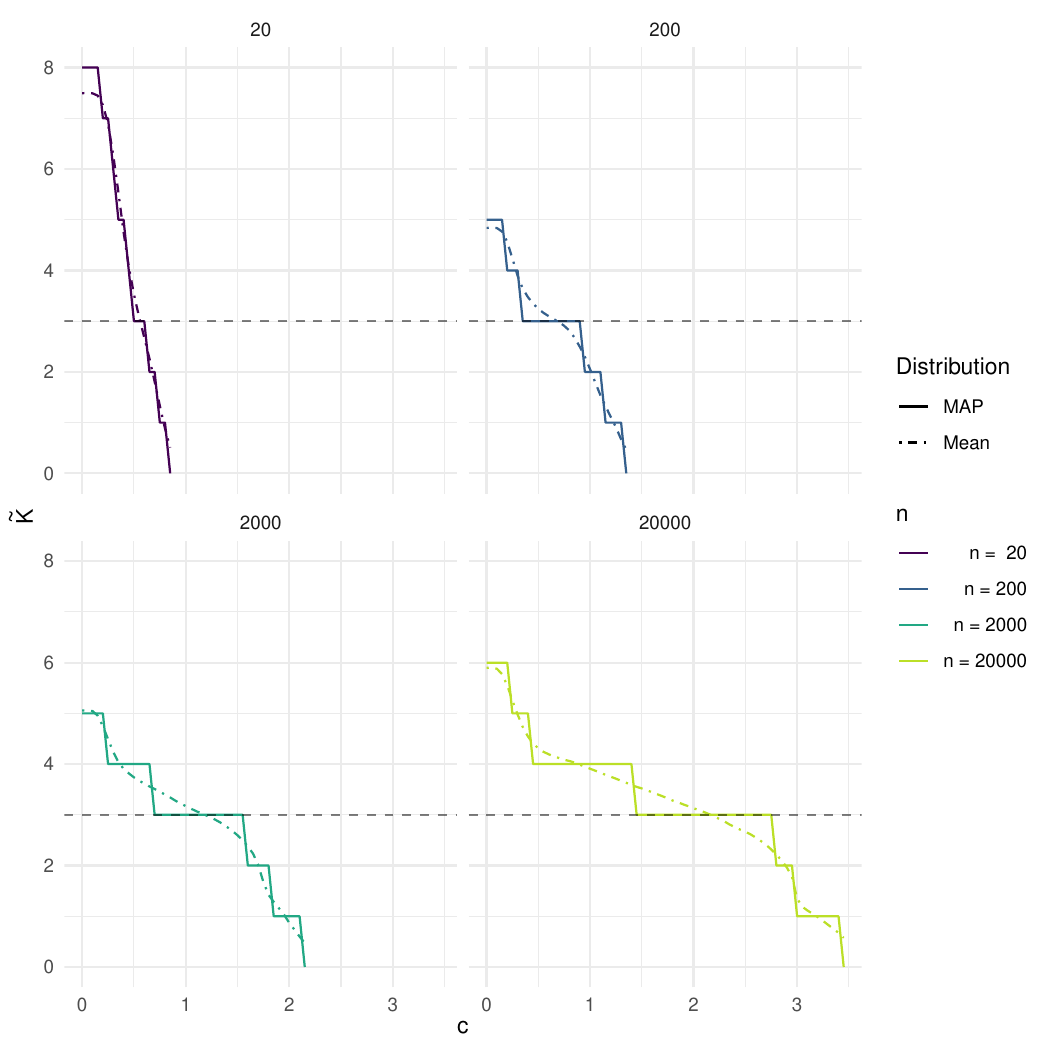} \\
 (c) Fixed $\Bar{\alpha} = 2.5$ & (d)   Fixed $\Bar{\alpha} = 3$  
\end{tabular}
\caption{``Regularization path'' for $\tilde K$, that is the posterior number of clusters after applying the Merge-Truncate-Merge algorithm of \cite{guha2021posterior}, with \textcolor{forestgreen}{parameter $c$ in $[0,c_{\max}]$}, \textlegendDPM. The dotted dashed curves represent the posterior mean while the solid curves represent the maximum a posteriori (MAP) and the dotted horizontal line represents $K_0 =3$.}
\label{fig:DPM-reg-path}
\end{figure}

\section{Real-data Analysis}
\label{sec:real}
We now consider the Sodium-Lithium Countertransport (SLC) dataset introduced by \cite{dudley_assessing_1991}. This dataset is composed of $190$ individual measurements of SLC level. 
This dataset was studied by \cite{miller_nonparametric_nodate} with a location-scale Gaussian mixture of finite mixture (MFM) model.
In the following, we consider the maximum a posteriori of the number of clusters found in \cite{miller_nonparametric_nodate} using MFM as ground truth, hence the true number of components is assumed to be $K_0=2$.

We used two different models to study this dataset: a Pitman--Yor process mixture model with various values of  $\alpha \in\{0.01,0.5\}$ and $\sigma\in\{0.1,0.25\}$, and a Dirichlet multinomial process mixture model with $K=10$ components and various choices of parameter $\Bar{\alpha}=\alpha/K\in\{0.01,0.5,1,2\}$. 
In both cases, we used a location-scale mixture model described in detail in Appendix \ref{an:real}.

In Figure \ref{fig:SLC}, we present the posterior distribution of the number of clusters and the so-called ``regularization paths'' for the two models. 
More precisely, Figure \ref{fig:SLC} (a) presents the posterior distribution of the number of clusters for a Pitman--Yor process mixture model.
We used the marginal sampler from \texttt{BNPmix} package proposed in \cite{corradin_bnpmix_2021}.
In Figure \ref{fig:SLC} (b), we illustrate the application of the Merge-Truncate-Merge algorithm \citep[MTM,][]{guha2021posterior} to the posterior distribution of the mixing measure for the Pitman--Yor process mixture model with the ``regularization paths'' plots for the parameter $c$ from the algorithm.
It is worth noticing that even if the contraction rate given in Lemma \ref{lem:Contract} is only valid for a location mixture with the scaling parameter known, here we used this rate for a location-scale mixture model as the scaling parameter is unknown. More precisely, we used the rate of Lemma \ref{lem:Contract} with $\eta=2$ as the kernel is Gaussian.
In the same way, in Figure \ref{fig:SLC} (c), we present the posterior distribution of the number of clusters for a Dirichlet multinomial process mixture model with the number of components $K=10$ and various choices of parameter $\Bar{\alpha}$.  
In Figure \ref{fig:SLC} (d), we illustrate the application of the Merge-Truncate-Merge algorithm \citep{guha2021posterior} to the posterior distribution of the mixing measure for the Dirichlet multinomial process mixture model with the ``regularization paths'' plots for the parameter $c$ from the algorithm.

\begin{figure}
    \centering
    \begin{tabular}{cc}
  \includegraphics[width=.5\textwidth]{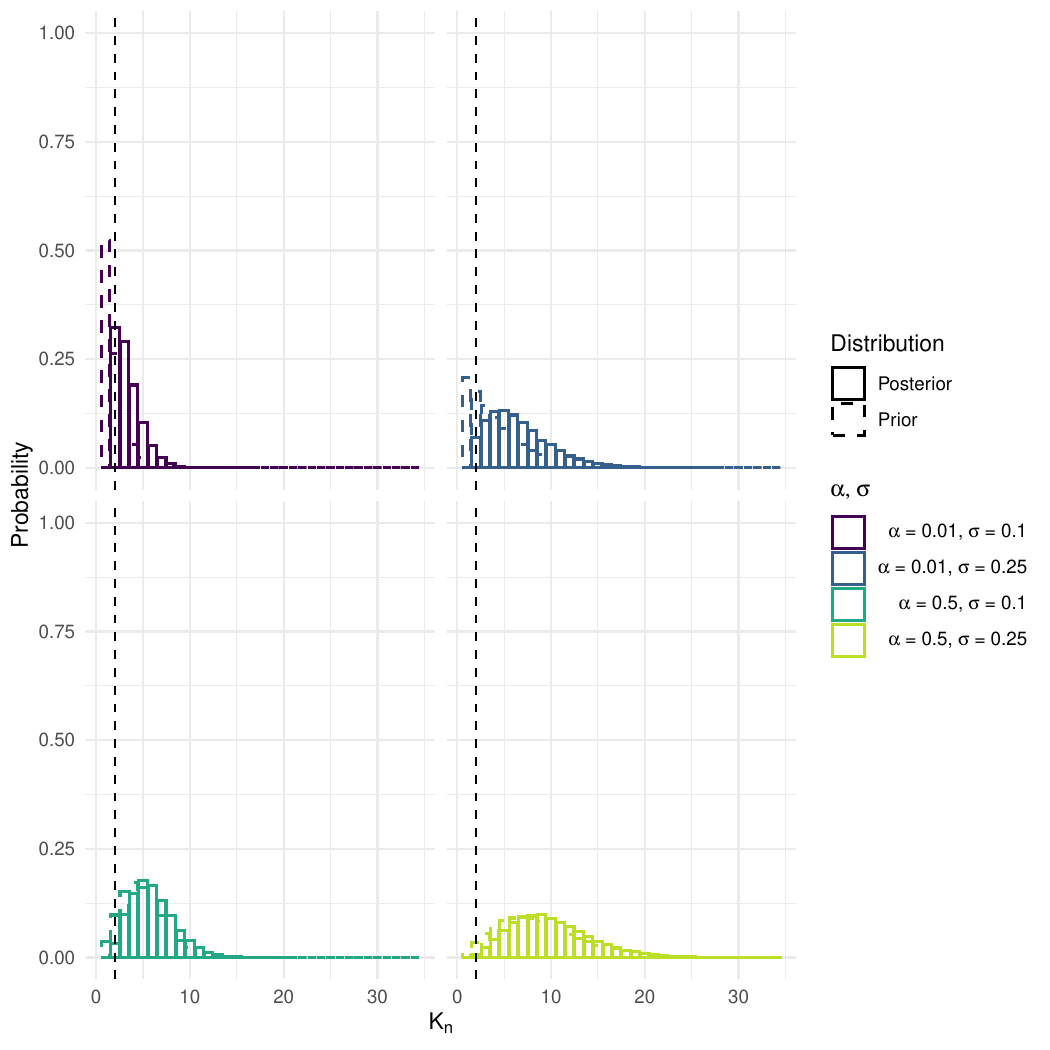} &   \includegraphics[width=.5\textwidth]{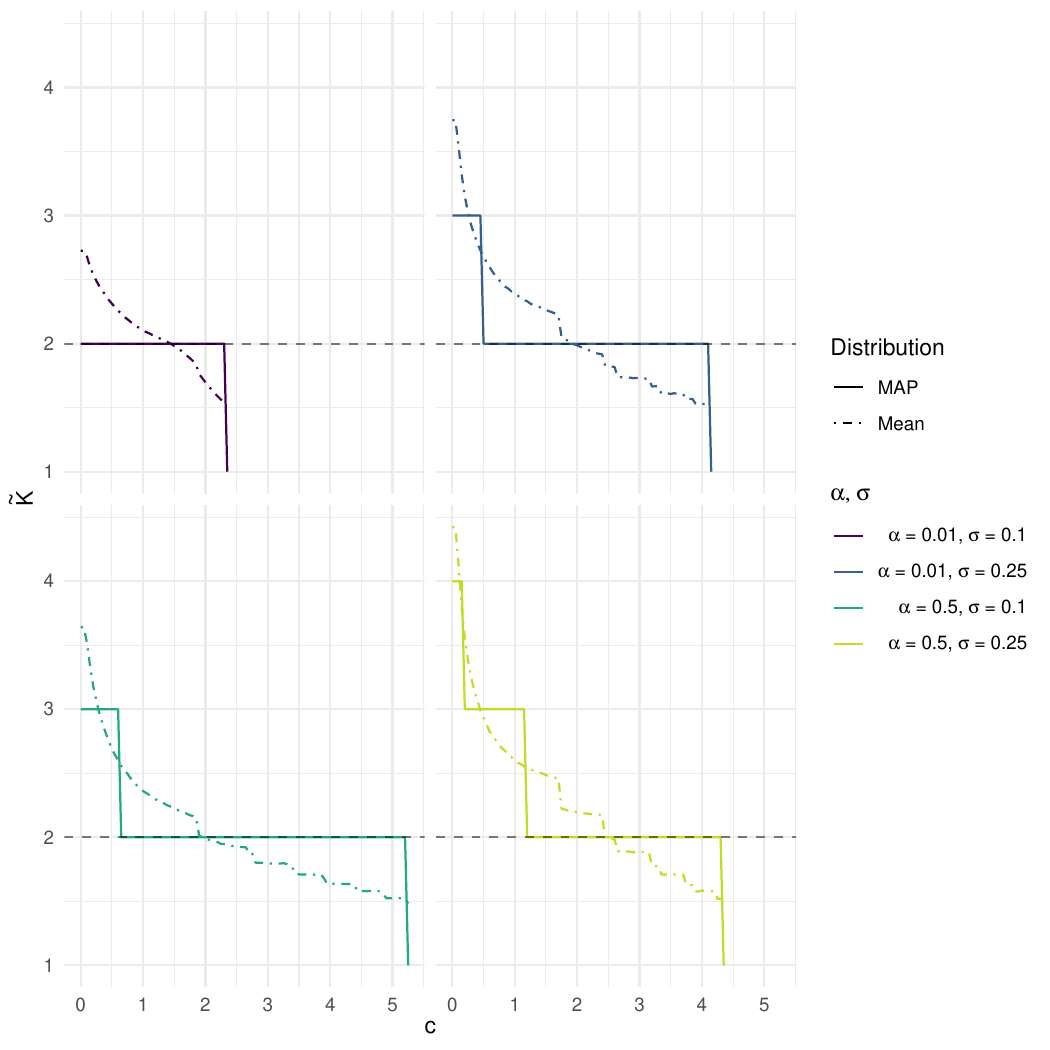}  \\
(a) Pitman--Yor process model & (b) Pitman--Yor process model   \\[6pt]
 \includegraphics[width=.5\textwidth]{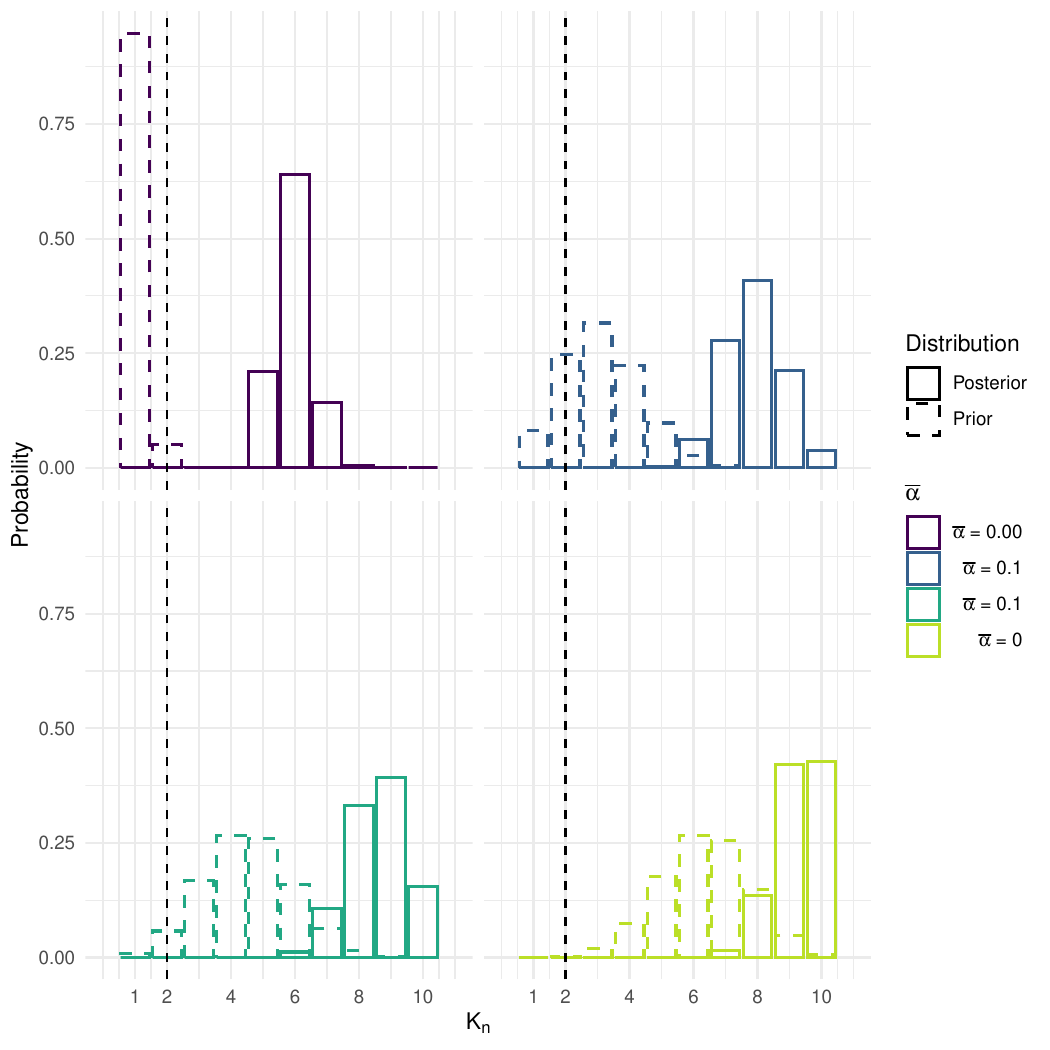} &
 \includegraphics[width=.5\textwidth]{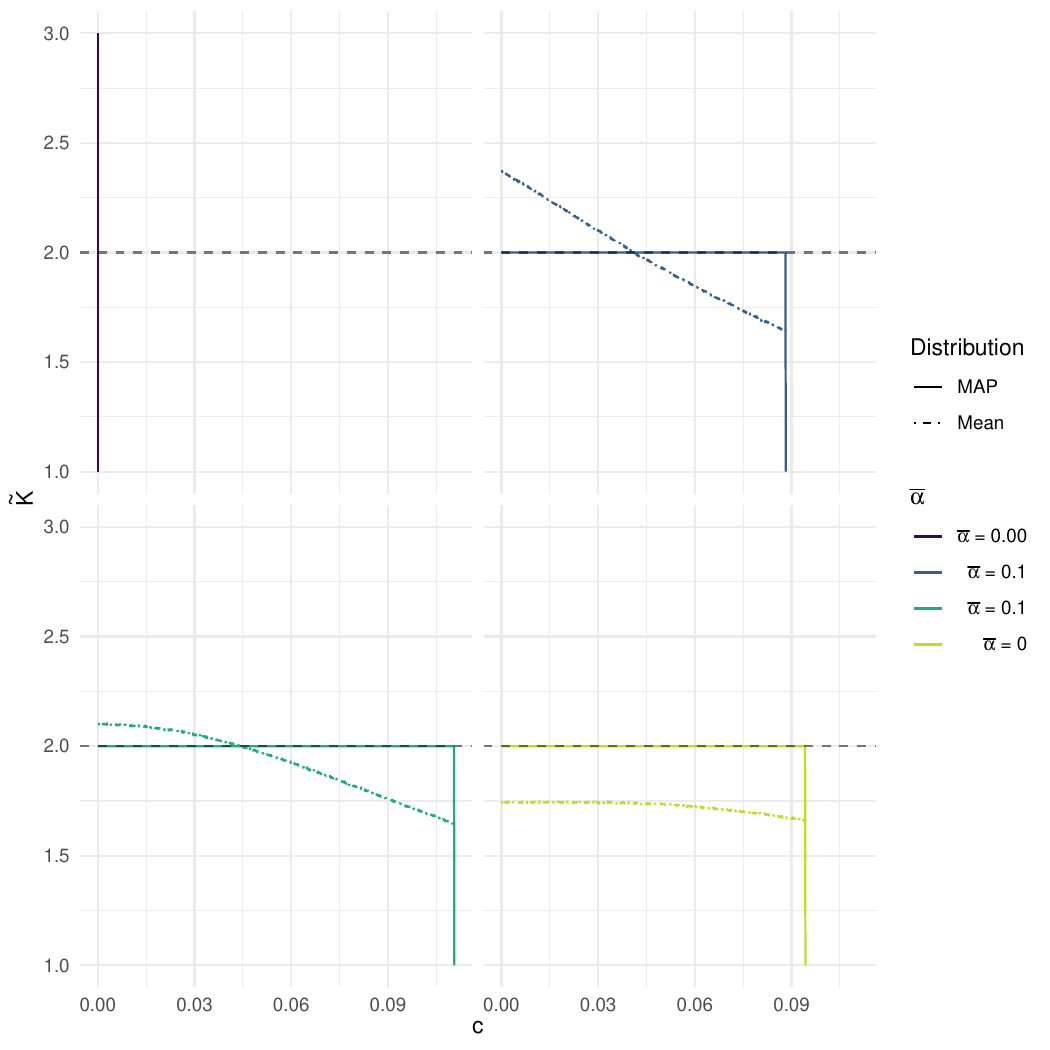} \\
 (c) Dirichlet multinomial process model & (d) Dirichlet multinomial process model   
\end{tabular}
    \caption{(a) Prior and posterior distributions of the number of clusters $K_n$ under a Pitman--Yor process model with various choices of $\alpha$ and $\sigma$ applied on the SLC dataset. (b) Corresponding ``regularization path'' for $\tilde K$, that is the posterior number of clusters after applying the Merge-Truncate-Merge algorithm of \cite{guha2021posterior}, with $c$ parameter in $[0,6]$. \\
    (c) Prior and posterior distributions of the number of clusters $K_n$ under a Dirichlet multinomial process mixture with fixed parameter $K = 10$, and various choices of $\Bar{\alpha}=\alpha/K$ applied on the SLC dataset. (d) Corresponding ``regularization path'' for $\tilde K$, that is the posterior number of clusters after applying the Merge-Truncate-Merge algorithm of \cite{guha2021posterior}, with $c$ parameter in $[0,0.12]$.\\ 
    For (b) and (d) the dotted dashed curves represent the posterior mean while the solid curves represent the maximum a posteriori (MAP). For (a) and (c) the solid curves represent the posterior distribution of $K_n$ while the dotted curves represent the prior distribution. The dotted line represents $K_0=2$.}
    \label{fig:SLC}
\end{figure}

For both models, in Figure \ref{fig:SLC} (a) and (c), the posterior distribution of the number of clusters is not centered around the ground truth $K_0=2$. This aligns with the inconsistency results in Section \ref{sec:inconsistency}. 
In Figure \ref{fig:SLC} (b), for each value of $\alpha$ and $\sigma$ a plateau can be observed on the true value $K_0=2$ for the maximum a posteriori as a function of the parameter $c$ of the MTM algorithm. 
In Figure \ref{fig:SLC} (b) and (d), the range of values for the parameter $c$ is $[0,c_\mathrm{max} ]$ where $c_\mathrm{max}$ is defined as in Section~\ref{sec:simulation}, such that for the greater value of $c$  $\tilde K$ is equal to $1$.
On the other hand, $c=0$ is such that only the first stage of MTM algorithm is performed, see Appendix \ref{an:bg} for more details.
In Figure \ref{fig:SLC} (d), the ranges of values for $c$ are very small, illustrating the fact that the first stage of the MTM algorithm alone already has a very strong merging effect. Plateaus on the true value $K_0=2$ for the maximum a posteriori as a function of $c$ can still be observed for $\Bar{\alpha}>0.01$.
For $\Bar{\alpha}=0.01$, the first stathe ge of MTM algorithm is not strong enough to find $\Tilde{K}=2$ (finds $\tilde K = 3$) and the second stage of the algorithm, which applies when $c>0$, directly jumps to one cluster $\Tilde{K}=1$.

\section{Discussion}
We studied the finite and infinite mixture models with well-specified kernels applied to data generated from a mixture with a finite number of components. In this setting, we have proved that Gibbs-type process mixtures are inconsistent a posteriori for the number of clusters. 
It is also the case for some finite-dimensional representations of  Gibbs-type priors such as the Dirichlet multinomial, Pitman--Yor multinomial and normalized generalized gamma multinomial processes. 
However, we did not prove inconsistency in general for \textsc{NIDM} \citep{lijoi2020finite-dim}. Further, we discussed the different approaches to solving inconsistency problems for both finite and infinite mixtures. 

For overfitted mixtures, \cite{rousseau2011asymptotic} prove that for some parameter specifications, the weights for extra components vanish, but it does not guarantee the posterior consistency of the number of clusters. We show that this guides prior specification for some of the models that are inconsistent a posteriori, such as overfitted mixtures based on the Dirichlet multinomial. On the other hand, we also proved that the Pitman--Yor multinomial process does not satisfy the conditions of \cite[Theorem 1 in][]{rousseau2011asymptotic}.
When the Wasserstein convergence rate of the mixing measure is known, the Merge-Truncate-Merge (MTM) algorithm proposed by \cite{guha2021posterior} allows obtaining a consistent estimate of the number of components in Bayesian nonparametric and overfitted mixtures. 
In particular, we showed that in contrast to the results of \cite{rousseau2011asymptotic}, the Merge-Truncate-Merge algorithm can be applied to the Dirichlet multinomial and Pitman--Yor multinomial processes without parameter constraints. 
Moreover, we also proved that Merge-Truncate-Merge can be applied to the Pitman--Yor process in the case of location mixtures.



Even if it seems possible to recover some consistency with, for example, the Merge-Truncate-Merge procedure, our inconsistency results suggest that Gibbs-type process mixture models face challenges when employed to estimate a finite number of components. This can be related to the fact that this usage corresponds to model misspecification, as these models assume an infinite number of components or a number of clusters growing with the sample size.
When it is known that the number of components is finite, we can also use a Mixture of Finite Mixtures which is better specified for this case. MFM are consistent for the number of components as proved in \cite{guha2021posterior}. 
However, MFM are notoriously more computationally challenging than Dirichlet process mixtures, for instance,  when the number of components is large  \citep[see remark in Section 3.2][]{guha2021posterior}. This might be a motivation to favour using misspecified Gibbs-type process mixture models in conjunction with the Merge-Truncate-Merge algorithm for instance in place of MFM.
However, recent works introduced new samplers for MFM which appear more computationally efficient than the usual ones \citep{miller2018,fruhwirth2021generalized}.

It is known that the Dirichlet process mixture model tends to create some extra little clusters which are linked to the inconsistency result \citep[see \textcolor{forestgreen}{e.g.}][and references therein]{miller2014inconsistency}. To avoid these clusters, some authors propose to use repulsive mixture models \citep[see \textcolor{forestgreen}{e.g.}][]{petralia2012repulsive}. Such models introduce a dependence on the components to better spread them out in the parameter space. \cite{xie_bayesian_2020} prove consistency for the density and the mixing measure for repulsive mixture models with Gaussian kernel. As for the number of components, no consistency is proven, but it is shown that some shrinkage effect occurs.

Another way to solve the inconsistency problem of the posterior number of clusters in the Dirichlet process mixture is introduced by \cite{ohn2020optimal}. 
Their solution is to make the concentration parameter $\alpha$ decrease when the sample size increases. 
With this assumption, they obtain a nearly tight upper bound on the true number of \textcolor{forestgreen}{components} through the posterior number of clusters. They also present a simulation study showing posterior consistency for the number of components.
We can wonder if control over the concentration parameter $\alpha$ when the sample size increases can allow posterior consistency for the number of components.
Indeed, \cite{ascolani2022clustering} proposes a way to control this parameter through a prior, which gives consistency for the number of \textcolor{forestgreen}{clusters} for a Dirichlet process mixture.

We investigate empirically these two directions, with a simulation study for Dirichlet multinomial mixtures where (i) we fix the expected number of clusters a priori when the sample size increases, implying that $\alpha$ decreases (Figure~\ref{fig:DPM-varying-alpha} (a) and (b)) and (ii) we use a Gamma prior on the concentration parameter (Figure~\ref{fig:DPM-varying-alpha} (c) and (d)) (see details in Appendix \ref{an:simulation}). As illustrated in Figure~\ref{fig:DPM-varying-alpha}, the posterior number of clusters in both cases seems to estimate the true number of components well even for large sample sizes, and the posterior seems to be consistent. This observation is corroborated by the posterior distribution of the weights shown in Figure~\ref{fig:DPM-varying-alpha} (b) and (d). However, there are no theoretical guarantees for consistency or inconsistency since the results of respectively \cite{ohn2020optimal} and \cite{ascolani2022clustering} do not apply in both cases. 
Although we cannot directly compare our experimental results with results obtained by \cite{ohn2020optimal} due to different theoretical assumptions,  we can note that the theoretical results obtained by \cite{ohn2020optimal} require that $\Bar{\alpha}$ decreases as $n^{-a_{0}}$, where $a_0>0$, which is faster than the \textcolor{forestgreen}{$1/\log(n)$} decrease induced by fixing the expectation.  So, the obtained results suggest that a slower decrease rate for $\alpha$ might be enough to obtain consistency. 

\begin{figure}[hp]
\thispagestyle{empty}
\centering
\begin{tabular}{cc}
  \includegraphics[width=.5\textwidth]{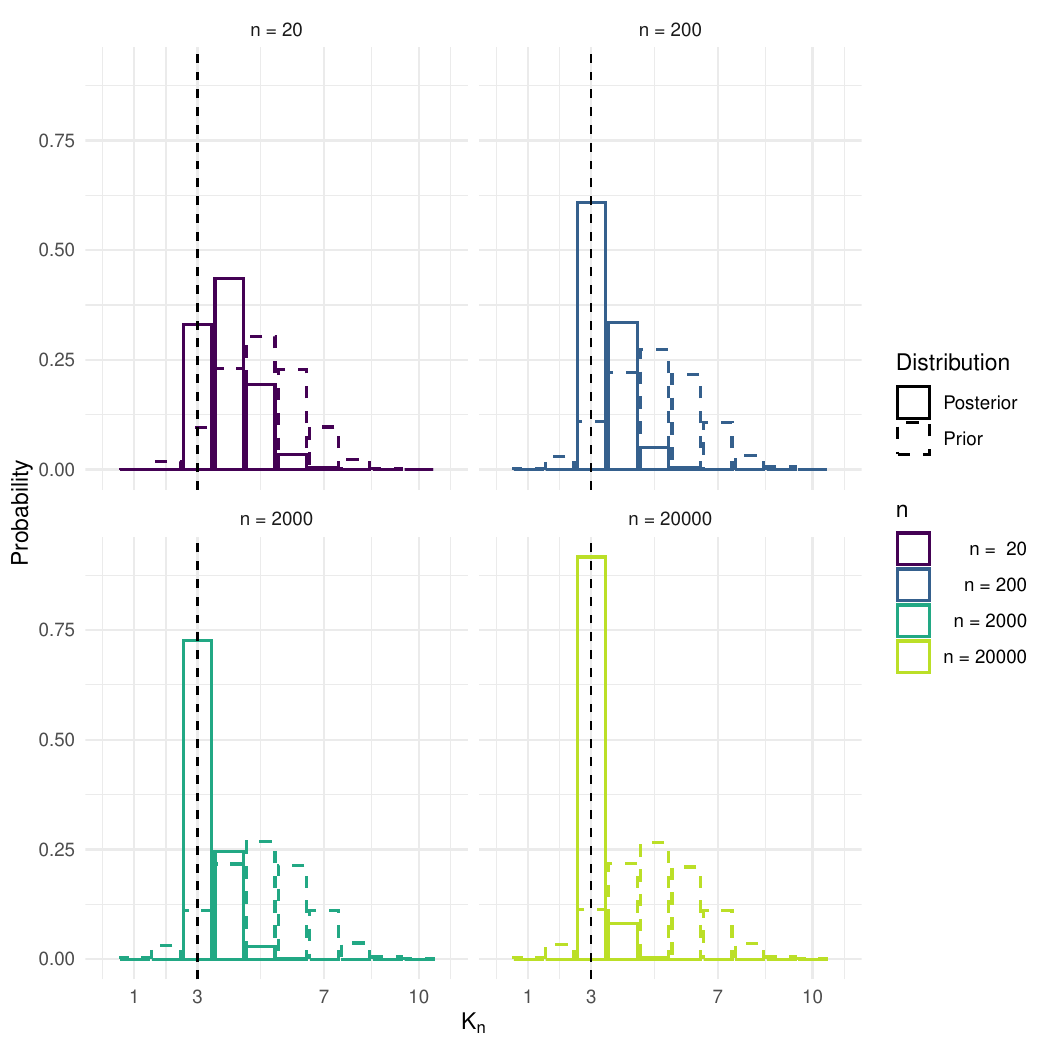} &   \includegraphics[width=.5\textwidth]{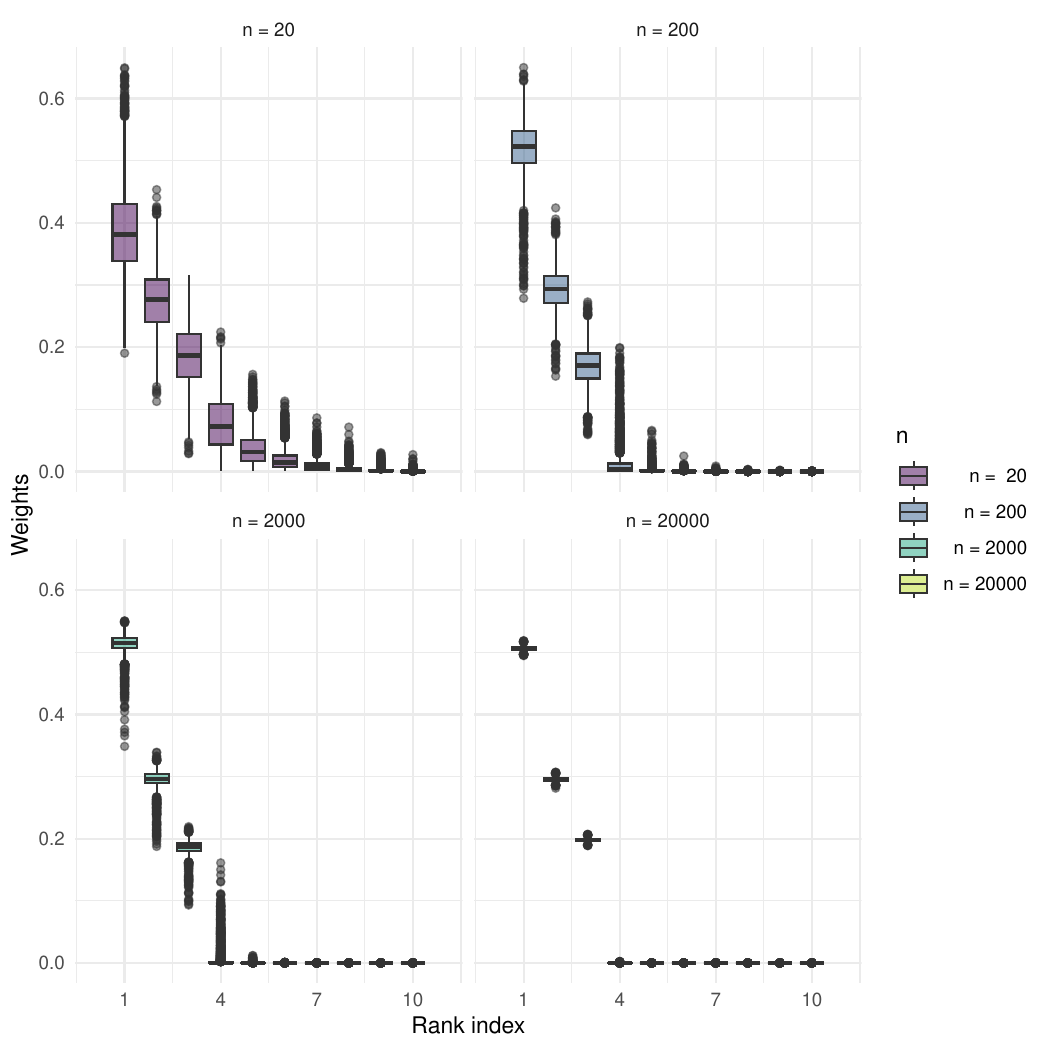}  \\
(a) $\Bar{\alpha}_n :$ $\mathbb{E}[K_n]= 5$  & (b)   $\Bar{\alpha}_n :$ $\mathbb{E}[K_n]= 5$ \\[6pt]
 \includegraphics[width=.5\textwidth]{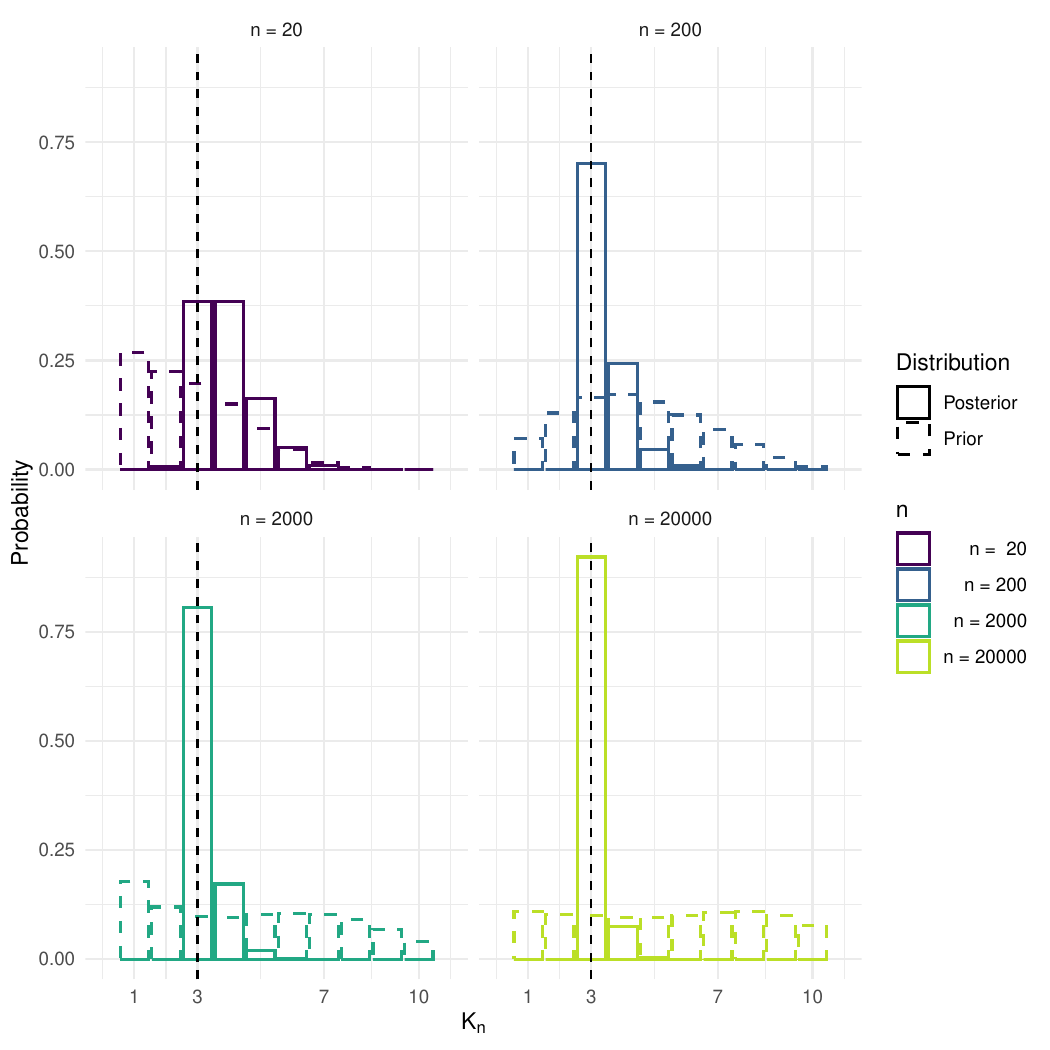} &
 \includegraphics[width=.5\textwidth]{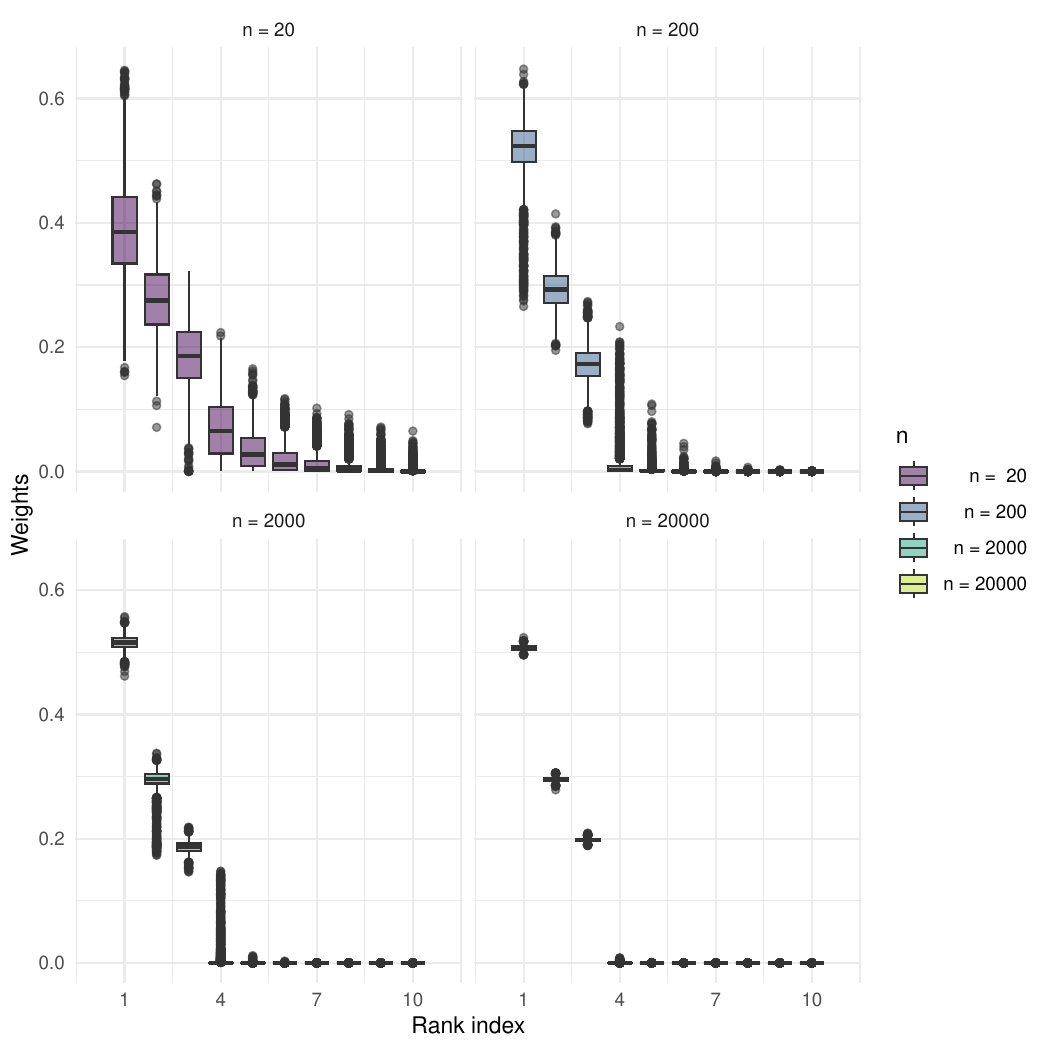} \\
 (c) $\Bar{\alpha} \sim \text{Ga}(a,bK)$ & (d)   $\Bar  \alpha \sim \text{Ga}(a,bK)$\\[6pt]
\end{tabular}
\caption{Dirichlet multinomial process mixtures varying concentration parameter $\Bar{\alpha}$. 
(a) and (b): $\Bar{\alpha}$ chosen such that $\mathbb{E}[K_n]=5$ for various choices of $n$. 
(c) and (d): a $\text{Ga}(a,bK)$ prior is used on $\Bar{\alpha}$. 
(a) and (c): Prior and posterior distributions of the number of clusters $K_n$.
(b) and (d): Boxplots of mixture weights.}
\label{fig:DPM-varying-alpha}
\end{figure}

Another way to estimate the number of components is to use the approach of \cite{wade_bayesian_2018}. This approach consists of a point estimation of the partition of the data and is commonly used in practice. 
As it is widely used in practice, it would be interesting to investigate the consistency in this case. 
\cite{chaumeny_bayesian_2022} investigates this question from a practical point of view, using a simulation study, some positive results are found, but no theory is provided.

Bayesian nonparametric or overfitted mixtures are often used in practical applications. In our work, we showed that the number of clusters estimated using these models is inconsistent for some of these models. We have also discussed possible ways to obtain consistent estimates in practice, using either prior- or post-processing procedures. However, throughout the paper, we considered a well-specified kernel case, where the data is generated from the finite mixture of distributions that belong to the considered kernel family. In practice, this condition can be easily violated, and an interesting avenue of research would be to investigate misspecified settings.

\section*{Acknowledgements}

We would like to thank the two anonymous Referees and the Associate Editor for their valuable comments, which significantly improved the manuscript. We also wish to thank Pierpaolo de Blasi for highlighting an inaccuracy in an earlier version of this work.

%% file: appendix.tex
\appendix
\section{Proofs of the results of Section \ref{sec:inconsistency}}
\label{an:proof}

\begin{proof}[Proof of Proposition \ref{prop:Gibbs}]
    For all $k\in\{1,2,\ldots\}$, we want to prove that 
    \[ \limsup_{n\to\infty}c_n(k) = \limsup_{n\to\infty}\frac{1}{n} \max_{A\in\A{k}}\max_{B\in\Z_A} \frac{p(A)}{p(B)} <\infty, \]
    where $\Z_A$ and $\A{k}$ are defined in Section \ref{sec:partition}.
    \\
    So, it is sufficient to prove that for any fixed $k$, there exists a constant $C$ such that for any $n$, for all $A\in\A{k}$ and $B = B(A,j)$ with $j\in A_\ell$,  $\frac{1}{n}\frac{p(A)}{p(B)}\leq C$.
    
    We consider the Gibbs-type prior case with $\sigma>0$, as case,  $\sigma=0$ is a Dirichlet process and is already proven in \cite{miller2014inconsistency}.
    As we are in the Gibbs-type prior case, we have, for $A\in\A{k}$,
    $p(A) = \frac{V_{n,k}}{k!} \prod_{i=1}^k (1-\sigma)_{n_i-1}$, and so
    \begin{align*}
        \frac{1}{n} \frac{p(A)}{p(B)} & = \frac{1}{n} \frac{V_{n,k}}{k!} \prod_{i=1}^k (1-\sigma)_{|A_i|-1} \frac{(k+1)!}{V_{n,k+1}} \left(\prod_{i=1}^{k+1} (1-\sigma)_{|B_i|-1}\right)^{-1} \\
        & =\frac{k+1}{n}\frac{V_{n,k}}{V_{n,k+1}} \underbrace{(1-\sigma+|A_\ell|-2)}_{\leq n} \\
        & \leq \frac{V_{n,k}}{V_{n,k+1}} (k+1).
    \end{align*}
    Therefore, we just have to prove that the sequence $\left(\frac{V_{n,k}}{V_{n,k+1}}\right)_{n\geq1}$ is bounded.\\
    \\    
    Using the recurrence relation \eqref{eq:rec_Vnk}, we have 
    \begin{align}
        V_{n,k} = V_{n+1,k+1} + (n-\sigma k) V_{n+1,k} &\iff  \frac{V_{n,k}}{V_{n+1,k+1}} = \frac{V_{n+1,k+1}}{V_{n+1,k+1}} + (n-\sigma k) \frac{V_{n+1,k}}{V_{n+1,k+1}} \nonumber\\
        \label{eq:ratio_Vnk}
        & \iff \frac{V_{n+1,k}}{V_{n+1,k+1}} = \left( \frac{V_{n,k}}{V_{n+1,k+1}}-
        1\right) \frac{1}{n-\sigma k}.
    \end{align} 
    
    We denote by $f_n(p,t) = t^{-\sigma k}p^{n-1-k\sigma} h(t) f_\sigma(t(1-p))$ the integrand function of Equation \eqref{eq:Vnk2}. From the definition of the $V_{n,k}$ in \eqref{eq:Vnk2}, we can write
    \[\frac{V_{n+1,k}}{V_{n,k}} = \frac{1}{n-\sigma k} \frac{\iint p f_n}{\iint f_n}. \]
    Using again the recurrence relation \eqref{eq:rec_Vnk}, we have 
    \[  \frac{V_{n+1,k+1}}{V_{n,k}} = 1-(n-\sigma k) \frac{V_{n+1,k}}{V_{n,k}}. \]
    Then, applying the Laplace approximation method twice and by setting $(t_n,p_n)$ the mode of $f_n$, we obtain as in \cite{arbel2021approximating}
    \begin{equation}
        \label{eq:ratio_lapl}
        \frac{V_{n+1,k+1}}{V_{n,k}} = g(t_n, p_n) + o\left(\frac{1}{n}\right),
    \end{equation}
    with $g(t_n, p_n) = 1-p_n$. 
    Indeed, to use the Laplace approximation, we write the integrand as $f_n = e^{n\ell_n}$, then
    \[ \frac{V_{n+1,k+1}}{V_{n,k}} = \frac{\iint g e^{n\ell_n}}{\iint e^{n\ell_n}}. \]
    As the exponential term is the same in both integrands of this ratio, by applying the Laplace approximation method to both integrals, we obtain
    \[\frac{V_{n+1,k+1}}{V_{n,k}} = \frac{g(t_n, p_n)+a(t_n,p_n)/n+\mathcal{O}\left(\frac{1}{n^2}\right)}{1+\mathcal{O}\left(\frac{1}{n}\right)}, \]
    where $a(t_n,p_n)$ is a second order term such that $a(t_n,p_n) = o(1/n)$. Hence, the previous ratio finally simplifies to \eqref{eq:ratio_lapl}.
    
    
    Let $\varphi_h(t) = -th'(t)/h(t)$, we can finally write using the partial derivatives above 
    \begin{equation}
        \frac{V_{n+1,k+1}}{V_{n,k}} = \frac{\sigma k + \varphi_h(t_n)}{n+ \varphi_h(t_n)-1} + o\left(\frac{1}{n}\right).
    \end{equation}
    Thus, if $\varphi_h(t_n)$ converges as $n$ tends to infinity, we have that $\frac{V_{n+1,k+1}}{V_{n,k}} \times \frac{n}{\sigma k}\to 1$ as $ n\to\infty$, so with the relation \eqref{eq:ratio_Vnk}, $\frac{V_{n+1,k}}{V_{n+1,k+1}} \underset{n\to\infty}{\longrightarrow} \frac{1}{\sigma k}$.
    If $\varphi_h(t_n)$ diverges as $n$ tends to infinity, we have that
    \[
    \lim_{n\to\infty} \frac{V_{n+1,k+1}}{V_{n,k}} = \left\{
        \begin{array}{ll}
            \frac{1}{c+1} & \text{ if } \frac{n}{\varphi_h(t_n)} \underset{n\to\infty}{\longrightarrow} c, \; c\in\mathbb{R}, \\
            0  & \text{ if } \frac{n}{\varphi_h(t_n)} \underset{n\to\infty}{\longrightarrow} \pm\infty.
        \end{array}
        \right.
    \]
    And then, using again \eqref{eq:ratio_Vnk}, $\frac{V_{n+1,k}}{V_{n+1,k+1}} \underset{n\to\infty}{\longrightarrow} 0$. Hence, 
    \[
    \lim_{n\to\infty} \frac{V_{n+1,k}}{V_{n+1,k+1}} = \left\{
        \begin{array}{ll}
            \frac{1}{\sigma k} & \text{ if }  \varphi_h(t_n) \text{ converges,} \\
            0  & \text{ if }  \varphi_h(t_n) \text{ diverges.}
        \end{array}
        \right.
    \]
    Thus, the sequence $\left(\frac{V_{n,k}}{V_{n,k+1}}\right)_{n\geq1}$ is bounded and Condition \ref{cond:part} is satisfied.
    
\end{proof}

\begin{proof}[Proof of Proposition \ref{prop:PYM}]
We consider $A\in\A{k}$ and $B = B(A,j)$, and we assume for simplicity, and without loss of generality, that the cluster in $A$ which contains the element $j$ is the $k$-th cluster $A_k$. As in the previous proof, we want to bound the ratio $\frac{p(A)}{p(B)}$ for the three different partition probabilities considered in the proposition. 

First, we consider the Dirichlet multinomial process, which is a special case of the Pitman--Yor multinomial process and normalized generalized gamma when $\sigma = 0$. Then we consider the Pitman--Yor multinomial process and the normalized generalized gamma process with $\sigma>0$.\\

(a) Dirichlet multinomial process: 
using \eqref{eq:EPPF_DPM}, we have
    \begin{equation*}
        \frac{1}{n}\frac{p(A)}{p(B)} = \frac{1}{n} \frac{p(n_1,\ldots,n_k)}{p(n_1,\ldots,n_k-1,1)}.
    \end{equation*}
So,
    \begin{align*}
        \frac{1}{n}\frac{p(A)}{p(B)} &= \frac{1}{n}\frac{(k+1)!(K-k-1)! \prod_{j=1}^{k}(c/K)_{n_j}(c)_n}{k!(K-k)! \prod_{i=1}^{k+1}(c/K)_{n_i}(c)_n}\\
        &= \frac{(k+1)(c/K + n_k -1)  }{n(K-k)c/K}\leq \frac{K(k+1)}{c(K-k)}.
    \end{align*}
    Thus, Condition \ref{cond:part} is satisfied for the Dirichlet multinomial process.

(b) Pitman--Yor multinomial process with $\sigma >0$: 
we denote by $q_{\ell^{(k)}} = \prod_{i=1}^{k} C(n_i,\ell_i;\sigma)/K^{\ell_i}$.
    Using \eqref{eq:EPPF_PYM}, we have
    \begin{align*}
        \frac{1}{n}\frac{p(A)}{p(B)} &= \frac{1}{n} \frac{p(n_1,\ldots,n_k)}{p(n_1,\ldots,n_k-1,1)}\\
        &= \frac{(k+1)!(K-(k+1))!}{nk!(K-k)!}\frac{\sum_{(\ell_1,\ldots,\ell_k)} \frac{\Gamma(\alpha/\sigma+|\ell^{(k)}|)}{\sigma\Gamma(\alpha/\sigma+1)}~q_{\ell^{(k)}}}
        {\sum_{(\ell_1,\ldots,\ell_{k+1})} \frac{\Gamma(\alpha/\sigma+|\ell^{(k+1)}|)}{\sigma\Gamma(\alpha/\sigma+1)}~q_{\ell^{(k+1)}}} \\[0.3cm]
        &= \frac{k+1}{n(K-k)} \frac{\sum_{(\ell_1,\ldots,\ell_{k-1})} \sum_{\ell_k=1}^{n_k}\Gamma(\alpha/\sigma+|\ell^{(k)}|)~q_{\ell^{(k)}}}
        {\sum_{(\ell_1,\ldots,\ell_{k-1})}\sum_{\ell_k=1}^{n_k-1}\sum_{n_{k+1} = 1}^1 \Gamma(\alpha/\sigma+|\ell^{(k+1)}|)~q_{\ell^{(k+1)}}} \\[0.3cm]
        &= \frac{k+1}{n(K-k)} \frac{\sum_{(\ell_1,\ldots,\ell_{k-1})} \sum_{\ell_k=1}^{n_k}\Gamma(\alpha/\sigma+|\ell^{(k)}|)~q_{\ell^{(k)}}}
        {\sum_{(\ell_1,\ldots,\ell_{k-1})}\sum_{\ell_k=1}^{n_k-1}\sum_{n_{k+1} = 1}^1 \Gamma(\alpha/\sigma+|\ell^{(k+1)}|)~q_{\ell^{(k)}}~ \frac{C(1,1;\sigma)}{K^{\ell_{k+1}}}} \\[0.3cm]
        &= \frac{K(k+1)}{n\sigma(K-k)} \frac{\sum_{(\ell_1,\ldots,\ell_{k-1})} \sum_{\ell_k=1}^{n_k}\Gamma(\alpha/\sigma+|\ell^{(k)}|)~q_{\ell^{(k)}}}
        {\sum_{(\ell_1,\ldots,\ell_{k-1})}\sum_{\ell_k=1}^{n_k-1} \Gamma(\alpha/\sigma+|\ell^{(k)}|+1)~q_{\ell^{(k)}}}\\[0.3cm]
        &=: \frac{K(k+1)}{n\sigma(K-k)}(R_1 + R_2).
    \end{align*}
    We separate the sum over $\ell_k$ in the numerator in two, $R_1$ corresponds to the first $n_k-1$ terms and $R_2$ to the last one.
    We compute separately $R_1$ and $R_2$. 
    \begin{align*}
        R_1 &= \frac{\sum_{(\ell_1,\ldots,\ell_{k-1})}\sum_{\ell_k=1}^{n_k-1} \Gamma(\alpha/\sigma+|\ell^{(k)}|)~q_{\ell^{(k)}}}
        {\sum_{(\ell_1,\ldots,\ell_{k-1})}\sum_{\ell_k=1}^{n_k-1} \Gamma(\alpha/\sigma+|\ell^{(k)}|+1)~q_{\ell^{(k)}}} \\[0.3cm]
        &= \frac{\sum_{(\ell_1,\ldots,\ell_{k-1})}\sum_{\ell_k=1}^{n_k-1} \Gamma(\alpha/\sigma+|\ell^{(k)}|)~q_{\ell^{(k)}}}
        {\sum_{(\ell_1,\ldots,\ell_{k-1})}\sum_{\ell_k=1}^{n_k-1} (\alpha/\sigma+|\ell^{(k)}|)\Gamma(\alpha/\sigma+|\ell^{(k)}|)~q_{\ell^{(k)}}}\\[0.3cm]
        &\leq \frac{\sum_{(\ell_1,\ldots,\ell_{k-1})}\sum_{\ell_k=1}^{n_k-1} \Gamma(\alpha/\sigma+|\ell^{(k)}|)~q_{\ell^{(k)}}}
        {(\alpha/\sigma+k)~\sum_{(\ell_1,\ldots,\ell_{k-1})}\sum_{\ell_k=1}^{n_k-1} \Gamma(\alpha/\sigma+|\ell^{(k)}|)~q_{\ell^{(k)}}}\\[0.3cm]
        &\leq \frac{1}{\alpha/\sigma+k}.
    \end{align*}
    Using twice the fact that $k\mapsto C(n,k;\sigma)$ is non increasing for $k\in\{1,\ldots,n\}$ \citep[see][]{bystrova2021approximating}, so $C(n_k,1;\sigma)\geq C(n_k,\ell_k;\sigma)\geq C(n_k,n_k;\sigma)$, and that $\Gamma(\alpha/\sigma+|\ell^{(k-1)}|+n_k)\leq \sum_{\ell_k=1}^{n_k-1} \Gamma(\alpha/\sigma+|\ell^{(k-1)}|+\ell_k+1)$, we obtain
    
    \begin{align*}
        R_2 &= \frac{\sum_{(\ell_1,\ldots,\ell_{k-1})}\sum_{\ell_k=n_k}^{n_k} \Gamma(\alpha/\sigma+|\ell^{(k)}|)~q_{\ell^{(k)}}}
        {\sum_{(\ell_1,\ldots,\ell_{k-1})}\sum_{\ell_k=1}^{n_k-1} \Gamma(\alpha/\sigma+|\ell^{(k)}|+1)~q_{\ell^{(k)}}} \\[0.3cm]
        & = \frac{\sum_{(\ell_1,\ldots,\ell_{k-1})} \Gamma(\alpha/\sigma+|\ell^{(k-1)}|+n_k) ~q_{\ell^{(k-1)}} \frac{C(n_k,n_k;\sigma)}{K^{n_k}}}
        {\sum_{(\ell_1,\ldots,\ell_{k-1})}\sum_{\ell_k=1}^{n_k-1} \Gamma(\alpha/\sigma+|\ell^{(k-1)}|+\ell_k+1)~q_{\ell^{(k-1)}}~ \frac{C(n_k,\ell_k;\sigma)}{K^{\ell_k}}}\\[0.3cm]
        &\leq \frac{C(n_k,n_k;\sigma)}{K^{n_k}} \frac{K^{n_k-1}}{C(n_k,1;\sigma)}
        \frac{\sum_{(\ell_1,\ldots,\ell_{k-1})} q_{\ell^{(k-1)}} \Gamma(\alpha/\sigma+|\ell^{(k-1)}|+n_k)}
        {\sum_{(\ell_1,\ldots,\ell_{k-1})} q_{\ell^{(k-1)}} \sum_{\ell_k=1}^{n_k-1} \Gamma(\alpha/\sigma+|\ell^{(k-1)}|+\ell_k+1)} \\[0.3cm]
        &\leq \frac{C(n_k,n_k;\sigma)}{K~ C(n_k,1;\sigma)} \leq \frac{1}{K}.
    \end{align*}
    
    Finally, we have that 
    \begin{align*}
        \frac{1}{n}\frac{p(A)}{p(B)}&=\frac{K(k+1)}{n\sigma(K-k)}(R_1 + R_2)
        \leq \frac{K(k+1)}{n\sigma(K-k)}\left(\frac{1}{\alpha/\sigma+k} + \frac{1}{K}\right).
    \end{align*}  
    So Condition \ref{cond:part} is satisfied for the Pitman--Yor multinomial process.\\
        
(c) Normalized generalized gamma multinomial process:
using \eqref{eq:EPPF_NGGM} and following the same way as for the Pitman--Yor case, we have
        \begin{align*}
        \frac{1}{n}\frac{p(A)}{p(B)} &= \frac{1}{n}\frac{p(n_1,\ldots,n_k)}{p(n_1,\ldots,n_k-1,1)}\\
        & = \frac{k+1}{n(K-k)}\left(\sum_{(\ell_1,\ldots,\ell_k)} \frac{V_{n,|\ell^{(k)}|}}{K^{|\ell^{(k)}|}}\prod_{i=1}^{k} \frac{C(n_i,\ell_i;\sigma)}{\sigma^{\ell_i}}\right)\left(\sum_{(\ell_1,\ldots,\ell_{k+1})} \frac{V_{n,|\ell^{(k+1)}|}}{K^{|\ell^{(k+1)}|}}\prod_{i=1}^{k+1} \frac{C(n_i,\ell_i;\sigma)}{\sigma^{\ell_i}}\right)^{-1} \\
        &= \frac{k+1}{n(K-k)}\left(\sum_{(\ell_1,\ldots,\ell_k)} \frac{V_{n,|\ell^{(k)}|}}{K^{|\ell^{(k)}|}}\prod_{i=1}^{k} \frac{C(n_i,\ell_i;\sigma)}{\sigma^{\ell_i}}\right)\left(\sum_{(\ell_1,\ldots,\ell_{k})} \frac{V_{n,|\ell^{(k)}|+1}}{K^{|\ell^{(k)}|+1}}\prod_{i=1}^{k} \frac{C(n_i,\ell_i;\sigma)}{\sigma^{\ell_i}}\right)^{-1}\\ 
        &=: \frac{K(k+1)}{n(K-k)} (R_1+R_2).
    \end{align*}
    As in \textsc{PYM} (b) proof, we separate the sum over $\ell_k$ in the numerator in two, $R_1$ corresponds to the first $n_k-1$ terms and $R_2$ to the last one.
    
    In the proof of Proposition \ref{prop:Gibbs}, we have shown that the ratio $\left(\frac{V_{n,k}}{V_{n,k+1}}\right)_{n\geq1}$ is bounded. Let $B\in\R^\star_+$ denote an upper bound of this sequence. Then
    \begin{align*}
        R_1 &= \left(\sum_{(\ell_1,\ldots,\ell_{k-1})}\sum_{\ell_k = 1}^{n_k-1} \frac{V_{n,|\ell^{(k)}|}}{K^{|\ell^{(k)}|}}\prod_{i=1}^{k} \frac{C(n_i,\ell_i;\sigma)}{\sigma^{\ell_i}}\right)\left(\sum_{(\ell_1,\ldots,\ell_{k-1})}\sum_{\ell_k = 1}^{n_k-1} \frac{V_{n,|\ell^{(k)}|+1}}{K^{|\ell^{(k)}|}}\prod_{i=1}^{k} \frac{C(n_i,\ell_i;\sigma)}{\sigma^{\ell_i}}\right)^{-1} \\
        &\leq B \left(\sum_{(\ell_1,\ldots,\ell_{k-1})}\sum_{\ell_k = 1}^{n_k-1} \frac{V_{n,|\ell^{(k)}|}}{K^{|\ell^{(k)}|}}\prod_{i=1}^{k} \frac{C(n_i,\ell_i;\sigma)}{\sigma^{\ell_i}}\right)\left(\sum_{(\ell_1,\ldots,\ell_{k-1})}\sum_{\ell_k = 1}^{n_k-1} \frac{V_{n,|\ell^{(k)}|}}{K^{|\ell^{(k)}|}}\prod_{i=1}^{k} \frac{C(n_i,\ell_i;\sigma)}{\sigma^{\ell_i}}\right)^{-1}\\
        &\leq B.
    \end{align*}
    Combining $\frac{V_{n,|\ell^{(k-1)}|+n_k}}{K^{|\ell^{(k-1)}|+n_k}} \leq \sum_{\ell_k = 1}^{n_k-1} \frac{V_{n,|\ell^{(k)}|+1}}{K^{|\ell^{(k)}|}}$ with similar arguments to the bounding of $R_2$ term in \textsc{pym} (b) above yield $R_2\leq\frac{1}{\sigma}$ 
    Finally, we obtain 
    \[\frac{1}{n}\frac{p(A)}{p(B)} \leq  \frac{K(k+1)(\sigma B+1)}{n\sigma(K-k)},\]
    so Condition \ref{cond:part} is satisfied for the normalized generalized gamma multinomial processes.

    Hence, there is inconsistency in the sense of Theorem \ref{thm} for the  Pitman--Yor multinomial process, the Dirichlet multinomial process, and the \textsc{NGGM} process.
    
\end{proof}

\section{Details on the results of Section \ref{sec:consistency}}
\label{an:bg}
\subsection{Theorem~1 of \cite{rousseau2011asymptotic}}
We recall the main result of \citet[Theorem~1]{rousseau2011asymptotic}. 
This result holds under some conditions on the mixture density, the kernel and the prior of the mixture model. Conditions \ref{hyp:1} and \ref{hyp:5}, stated previously, are conditions on the mixture density and on the prior density.
Under the notations used here, we stated the conditions on the kernel density, which need to have some regularity, integrability and strong identifiability properties. 
As a reminder, $\theta = (\theta_1,\ldots,\theta_K)$ denotes the set of component parameters, $w = (w_1,\ldots,w_K)$ denotes the weights of the mixing measure, $f(\cdot\mid\theta_i)$ denotes a component specific kernel density and $G=\sum_i w_i\delta_{\theta_i}$ denotes the mixing measure.
We have data observations $X_1,\ldots, X_n$ assumed to be independent and identically distributed from a mixture model with $K_0$ components, where $K_0<K$:
\[ f_0^X(x) = \sum_{k=1}^{K_0} w^0_k f(x\mid\theta_k^0), \qquad \theta_k^0\in \Theta. \]
\begin{Cond}[\citealp{rousseau2011asymptotic}, Assumption 2]
\label{hyp:2}
    The kernel function $\Tilde{\theta}\in\Theta\to f(\cdot\mid\Tilde{\theta})$ is three time differentiable and regular in the sense that for all $\Tilde{\theta}\in \Theta$ the Fisher information matrix that is associated with $f(\cdot\mid\Tilde{\theta})$ is positive definite at $\Tilde{\theta}$.
    Denote by $\nabla f(x\mid\theta)$ and $\mathrm{D}^2 f(x\mid\theta)$ respectively the vector of the first derivatives and the matrix of second derivatives of $f(x\mid\theta)$ with respect to $\theta$. Denote also by $\mathrm{D}^{(3)}f(x\mid\theta)$ the array whose components are $\frac{\partial^3 f(x\mid\theta)}{\partial\theta_{i_1}\partial\theta_{i_2}\partial\theta_{i_3}}$.

    For all $i\leq K_0$, there exists $\delta>0$ such that
    \[
    \int f_0^X(x) \frac{\sup_{|\theta_i^0-\theta|\leq \delta} f(x\mid\theta)^3}{\inf_{|\theta_i^0-\theta|\leq \delta} f(x\mid\theta)^3} \mathrm{d}x <\infty, 
    \quad \int f_0^X(x) \frac{\sup_{|\theta-\theta_i^0|\leq \delta} |\nabla f(x\mid\theta)|^3}{\inf_{|\theta_i^0-\theta|\leq \delta} f(x\mid\theta)^3} \mathrm{d}x <\infty, \]
    \[\quad \int f_0^X(x) \frac{|\nabla f(x\mid\theta_i^0)|^4}{f_0^X(x)^4} \mathrm{d}x <\infty, \]
    \[
    \int f_0^X(x) \frac{\sup_{|\theta-\theta_i^0|\leq \delta} |\mathrm{D}^2 f(x\mid\theta)|^2}{\inf_{|\theta_i^0-\theta|\leq \delta} f(x\mid\theta)^2} \mathrm{d}x <\infty, 
    \quad \int f_0^X(x) \frac{\sup_{|\theta-\theta_i^0|\leq \delta} |\mathrm{D}^{(3)} f(x\mid\theta)|^2}{\inf_{|\theta_i^0-\theta|\leq \delta} f(x\mid\theta)} \mathrm{d}x <\infty.
    \]
   Assume also that for all $i=1,\ldots,K_0$, $\theta_i^0\in \mathrm{int}(\Theta)$ the interior of $\Theta$.
\end{Cond}

\begin{Cond}[\citealp{rousseau2011asymptotic}, Assumption 3]
\label{hyp:3}
    There exists $\Theta_0\subset \Theta$ satisfying $\lambda(\Theta_0)>0$, where $\lambda(A)$ denotes the Lebesgue measure of $A$, and for all $i\leq K_0$, 
    \[ d(\theta_i^0,\Theta_0) = \inf_{\theta\in\Theta_0} |\theta-\theta_i^0|>0, \]
    and such that for $\theta\in\Theta_0$ there exists a $\delta>0$,
    \[ \int f^X_0(x) \frac{f(x\mid\theta)^4}{f_0(x)^4} \mathrm{d}x <\infty, \quad \int f^X_0(x) \frac{f(x\mid\theta)^3}{\sup_{|\theta'-\theta_i^0|\leq\delta}f(x\mid\theta')^3} \mathrm{d}x <\infty, \; \forall i\leq K_0. \]
\end{Cond}

\begin{Cond}[\citealp{rousseau2011asymptotic}, Assumption 4]
\label{hyp:4}
    For all ordered partitions $\mathbf{t}$ of $\{1,\ldots,K\}$ in $K_0+1$ clusters defined by the cardinality of each cluster, $\mathbf{t} = (t_i)_{i=0}^{K_0}$ with $0=t_0<t_1<\cdots<t_{K_0}\leq K$, let $(w,\theta) = (w_1,\ldots,w_K,\theta_1,\ldots,\theta_K)$ and write $(w,\theta)$ as $(\phi_\mathbf{t},\psi_\mathbf{t})$, where 
    \[ \phi_\mathbf{t} = ((\theta_j)_{j=1,\ldots,t_{K_0}}, (s_i)_{i=1,\ldots, K_0-1}, (w_j)_{j=t_{K_0}+1,\ldots,K}) \in \mathbb{R}^{dt_{K_0}+K_0+K+t_{K_0}-1}, \qquad s_i = \sum_{j=t_{i-1}+1}^{t_i} w_j -w_i^0, \, i=1,\ldots,K_0,\]
    and 
    \[ \psi_\mathbf{t} = ((q_j)_{j=1,\ldots,t_{K_0}}, \theta_{t_{K_0}+1},\ldots,\theta_{t_K}), \qquad q_i = w_i/\sum_{j=t_{i-1}+1}^{t_i} w_j, \, \text{where } i\in \{t_{i-1}+1,\ldots,t_i\}.\]
    We denote by $g^X(x)$ the density associated to the parameterization $(\phi^0_\mathbf{t}, \psi_\mathbf{t})$ of $(w,\theta)$.
    Then 
    \begin{multline*}
    (\phi_\mathbf{t}-\phi_\mathbf{t}^0)^T {g^X}'+\frac{1}{2} (\phi_\mathbf{t}-\phi_\mathbf{t}^0)^T {g^X}'' (\phi_\mathbf{t}-\phi_\mathbf{t}^0) = 0 \quad \Leftrightarrow 
    \\
    \forall i \leq K_0,\; s_i=0 \; \text{and} \; \forall j\in \{t_{i-1}+1,\ldots,t_i\}, \; q_j(\theta_j-\theta^0_j) = 0, \quad \forall i\geq t_{K_0}+1, \; w_i = 0.
    \end{multline*}
    Assuming also that if $\theta\not\in\{\theta_1,\ldots\theta_k\}$ then for all functions $h(\cdot\mid\theta)$ which are linear combinations of derivatives of $f(\cdot\mid\theta)$ of order less than or equal to 2 with respect to $\theta$, and all functions $h_1$ which are also linear combinations of derivatives of $f(\cdot\mid\theta_j)$, $j=1,\ldots,K$, and its derivatives of order less than or equal to 2, then $a h(\cdot\mid\theta)+ b h_1(\cdot) = 0$ if and only if $a h(\cdot\mid\theta)= b h_1(\cdot) = 0$.
\end{Cond}
This last condition can be extended to the non-compact cases if $\Theta$ is not compact as explained in \cite{rousseau2011asymptotic}.

We recall that $d$ denotes the dimension of $\theta$.
Under the three conditions detailed above, Condition \ref{hyp:1} and Condition \ref{hyp:5}, we can state the main result in \cite{rousseau2011asymptotic} in the following theorem.
\begin{Thm}[\citealp{rousseau2011asymptotic}, Theorem 1]
    Under all the five conditions recalled previously that the prior distribution satisfies, let $\mathcal{S}_K$ be the set of permutations of $\{1,\ldots,K\}$, $\alpha_{\mathrm{max}} = \max(\alpha_j,\,j\leq K)$ and $\alpha_{\mathrm{min}} = \min(\alpha_j,\,j\leq K)$.
    \begin{enumerate}[label=(\roman*)]
      \item If $\alpha_{\mathrm{max}}<d/2$, set $\rho = \left[dK_0+K_0-1+ \alpha_{\mathrm{max}}(K-K_0)\right]/(d/2-\alpha_{\mathrm{max}})$, then
      \[
      \lim_{M\to\infty} \limsup_n\left(\mathbb{E}_0^n\left[ \Pi\left\{ \min_{\sigma\in\mathcal{S}_K} \sum_{i=K_0+1}^K w_{\sigma(i)} > Mn^{-1/2} \log(n)^{q(1+\rho)} \,\middle|\, X_{1:n} \right\} \right]\right) = 0.
      \]
      \item If $\alpha_{\mathrm{min}}>d/2$, set $\rho' = \left[dK_0+K_0-1+ d(d-K_0)/2\right]/(\alpha_{\mathrm{min}}-d/2)(K-K_0)$, then
      \[
      \lim_{\epsilon\to0} \limsup_n\left(\mathbb{E}_0^n\left[ \Pi\left\{ \min_{\sigma\in\mathcal{S}_K} \sum_{i=K_0+1}^K w_{\sigma(i)} < \epsilon \log(n)^{-q(1+\rho')} \,\middle|\, X_{1:n} \right\} \right]\right) = 0.
      \]
    \end{enumerate}
\end{Thm}

\subsection{Merge-Truncate-Merge Algorithm of \cite{guha2021posterior}}
We recall the Merge-Truncate-Merge algorithm in \cite{guha2021posterior} used in Section \ref{sec:MTM}. This algorithm is a post-processing procedure applied on a posterior sample of the mixing measure $G$.
Applying this algorithm, a posterior contraction rate for the mixing measure under the Wasserstein metric, denoting $\omega_n$, is mandatory. More precisely, we need $G$ such that 
\[\Pi\left(G\,:\, W_r(G,G_0)\leq\delta\omega_n\mid X_{1:n}\right)\overset{p_{G_0}}{\longrightarrow} 1,\]
with $\omega_n=o(1)$ a vanishing rate, $r\geq1$.
We also need to choose a constant $c$ used in the second stage of the algorithm. There is no explicit way of choosing this constant in \cite{guha2021posterior}, we describe it as a regularisation parameter, which we illustrate in figure \ref{fig:DPM-reg-path}.
\renewcommand{\algorithmicrequire}{\textbf{Input:}}
\renewcommand{\algorithmicensure}{\textbf{Output:}}
\begin{algorithm}[H]
\caption{Recall of Merge-Truncate-Merge Algorithm (MTM) \citep{guha2021posterior}}
    \begin{algorithmic}[1]
        \REQUIRE{Posterior sample $G=\sum_i w_i \delta_{\theta_i}$, rate $\omega_n$, constant $c$.}
        \ENSURE{Discrete measure $\Tilde{G}$ and its number of supporting atoms $\Tilde{k}$.}
        
        \COMMENT{\textbf{Stage 1: Merge procedure}}
        
        \STATE Reorder atoms $\{\theta_1,\theta_2,\ldots\}$ by simple random sampling without replacement with corresponding weights $\{w_1,w_2,\ldots\}$, let $\tau_1,\tau_2,\ldots$ denote the new indices and set $\mathcal{E}=\{\tau_j\}_j$ as the existing set of atoms.
        
        \STATE Sequentially for each index $\tau_j\in\mathcal{E}$, if there exists an index $\tau_i<\tau_j$ such that $\|\theta_{\tau_i}-\theta_{\tau_j}\|\leq\omega_n$, then update $w_{\tau_i}=w_{\tau_i}+w_{\tau_j}$, and remove $\tau_j$ from $\mathcal{E}$.
        
        \STATE Collect $G'=\sum_{j:\,\tau_j\in\mathcal{E}} w_{\tau_j}\delta_{\theta_{\tau_j}}$, write $G'$ as $\sum_{i>1} q_i\delta_{\gamma_i}$ so that $q_1\geq q_2\geq\cdots$.

        \COMMENT{\textbf{Stage 2: Truncate-Merge procedure}}
        
        \STATE Set $\mathcal{A}=\left\{i:\, q_i>\left(c\omega_n\right)^r\right\}$ and $\mathcal{N}=\left\{i:\, q_i\leq\left(c\omega_n\right)^r\right\}$.
        
        \STATE For each index $i\in\mathcal{A}$, if there is $j\in\mathcal{A}$ such that $j<i$ and $q_i\|\gamma_i-\gamma_j\|^r\leq\left(c\omega_n\right)^r$, then remove $i$ from $\mathcal{A}$ and add it to $\mathcal{N}$.

        \STATE For each $i\in\mathcal{N}$, find the atom $\gamma_j$ among $j\in \mathcal{A}$ that is nearest to $\gamma_i$, update $q_j = q_j+q_i$.

        \STATE Return $\Tilde{G} = \sum_{j\in\mathcal{A}} q_j\delta_{\gamma_j}$ and $\Tilde{K} = |\mathcal{A}|$.
    \end{algorithmic}
\end{algorithm}

As recalled in Theorem \ref{thm:cons}, \cite{guha2021posterior} prove that the output $\tilde{K}$ of the MTM algorithm consistently estimates the number of clusters for any $c>0$.
This result holds under the assumption that there exists a contraction rate for the mixing measure. In order to have a contraction rate the kernel $f(\cdot\mid \theta)$ needs to satisfy some assumptions presented below.
\begin{Cond}[Second-order identifiability]
\label{hyp:id}
    The family of densities $\{f(\cdot\mid \theta), \;\theta\in\Theta\}$ is identifiable in the second order if $f(x\mid \theta)$ is twice differentiable in $\theta$ and for any finite $k$ different $\theta_1,\ldots,\theta_k\in\Theta$, the equality 
    \[\sup_x \left\vert \sum_{j=1}^k (\alpha_j f(x\mid \theta_j)+\beta_j^T \frac{\partial f}{\partial\theta}(x\mid \theta_j)+ \gamma_j^T \frac{\partial^2 f}{\partial\theta^2}(x\mid \theta_j)\gamma_j \right\vert=0\]
    implies that $\alpha_j=0\in\mathbb{R}$, $\beta_j=\gamma_j=0\in\mathbb{R}^d$ for $j=1,\ldots,k$.
\end{Cond}

\begin{Cond}[Uniform Lipschitz-continuity]
\label{hyp:lip}
    The family of densities $\{f(\cdot\mid \theta), \;\theta\in\Theta\}$ is uniformly Lipschitz continuous up to the second order if there exists a positive constant $\delta$ such that for any $R>0$ $\|\theta\|\leq R$, $\gamma\in\mathbb{R}^d$, $\theta_1,\theta_2\in\Theta$, there is a positive constant $C>0$ depending on $R$ such that for all $x\in\mathcal{X}$
    \[
    \left\vert \gamma^T\left(\frac{\partial^2 f}{\partial\theta^2}(x\mid\theta_1) -\frac{\partial^2 f}{\partial\theta^2}(x\mid\theta_2)\right)\gamma \right\vert \leq C \|\theta_1-\theta_2\|^\delta_1 \|\gamma\|_2^2.
    \]
\end{Cond}
For more details on these conditions and on contraction rate see \cite{chen_optimal_1995, ho_strong_2016}.

\subsection{Theorem~1 of \cite{scricciolo2014adaptive}}
In Section \ref{sec:consistency}, we introduce Lemma \ref{lem:Contract} which is a corollary of Theorem 1 in \cite{scricciolo2014adaptive}. 
This theorem gives a posterior contraction rate for the mixing measure of a Pitman--Yor mixture model relative to the $L^p$-metric.
We detailed below the conditions for this theorem and then we recalled the result.

Here, we assume that $\Theta\subset\mathbb{R}$. The model is less general than in the rest of the paper, we assume that the model is a location mixture defined as:
\begin{equation}
    \label{eq:loc_mix}
    f^X(x) = \int f(x\mid\theta,\tau)G(\mathrm{d}\theta)=  \int \tau^{-1} f((x-\theta_k)/\tau)G(\mathrm{d}\theta),    
\end{equation}
where $\tau$ is a scale parameter and $f(\cdot\mid\theta) = f(\cdot)$ denotes the kernel density. 
In this Section, we will assume that the scale parameter $\tau_0$ is known as the true scale parameter $\tau_0$. This can also be seen as $\tau$ following the prior distribution $\delta_{\tau_0}$.

The theorem holds under three conditions on the kernel density, the true mixing measure $G_0$ and the base measure of the Pitman--Yor process.


\begin{Cond}[\citealp{scricciolo2014adaptive}, Assumption A1]
\label{hyp:A1}
    The kernel probability density $f(\cdot\mid\theta):\mathbb{R}\to\mathbb{R}^+$ is symmetric around $0$, monotone decreasing in $|x|$ and satisfies the tail condition $f(x\mid\theta) \gtrsim e^{-c|x|^\kappa}$ as $|x|\to \infty$, for some constants $0<c,\, \kappa<\infty$.
\end{Cond}

\begin{Cond}[\citealp{scricciolo2014adaptive}, Assumption A2]
\label{hyp:A2}
    The true mixing measure $G_0$ satisfies the tail condition $G_0(\theta: \, |\theta|>t)\gtrsim e^{-c_0t^\varpi}$ as $t\to \infty$, for some constants $0<c_0<\infty$ and $0<\varpi\leq\infty$.
\end{Cond}

\begin{Cond}[\citealp{scricciolo2014adaptive}, Assumption A3]
\label{hyp:A3}
    The base measure $P$ has a continuous and positive density $p'$ on $\mathbb{R}$ such that $p'(\theta)\propto e^{-b|\theta|^\delta}$ as $|\theta|\to\infty$, for some constants $0<b,\,\delta<\infty$.
\end{Cond}
We also introduce the following set,
\[\mathcal{A}^{\rho,\eta,L} := \left\{f:\mathbb{R}\to\mathbb{R}^+ \, |\, \|f\|_1 = 1, \, \int e^{2(\rho|t|)^\eta}|\hat{f}(t)|^2\mathrm{d}t \leq 2\pi L^2\right\},\]
where $\hat{f}$ denotes the Fourier transform of $f$ and $\rho,L,\eta$ are some positive constants. 

We can now recall Theorem~1 from \cite{scricciolo2014adaptive}. Here, we state a simplified version of the theorem as we assumed the scale parameter to be known. The general statement requires additional conditions on the scale parameter. 
\begin{Thm}[\citealp{scricciolo2014adaptive}, Theorem 1]
    Let $f(\cdot\mid\theta)\in\mathcal{A}^{\rho,\eta,L}(\mathbb{R})$, $0<\rho,\eta,L<\infty$, be as in Condition \ref{hyp:A1}. Suppose that the true mixture density $f_0^X$, with
    \begin{enumerate}[label=(\roman*)]
        \item $G_0$ satisfying Condition \ref{hyp:A2} for some constants $0<c_0<\infty$ and given numbers $0<\kappa,\,\eta<\infty$ $\varpi$ be such that $0<\max\left\{\kappa, \,[1+\mathbb{1}_{(1,\infty)}(\eta)/(\eta-1)]\right\}\leq\varpi\leq\infty$.
    \end{enumerate}
    Let $G\sim\mathrm{PY}(\alpha,\sigma; P)$, with $0\leq\sigma<1$, $-\sigma<\alpha<\infty$ and a base measure $P$. Assume that
    \begin{enumerate}[label=(\roman*), resume]
        \item $P$ satisfies Condition \ref{hyp:A3} for constants $0<b,\,\delta<\infty$, with $\delta\leq\varpi$ when $\varpi<\infty$;
    \end{enumerate}
    Then, the posterior contraction rate relative to the $L^p$-metric, $1\leq p\leq\infty$, denoted by $\omega_{n,p}$, is $n^{-1/2}\textcolor{forestgreen}{\log (n)}^\mu$, with a constant $0<\mu\infty$ possibly depending on $p$.
\end{Thm}

\section{Proofs of the results of Section \ref{sec:consistency}}

\label{an:prop4}
\begin{proof}[Proof of Proposition \ref{prop:over}]
    In the Dirichlet multinomial process case, the prior on the weights $w=(w_1,\ldots,w_K)$ is a finite-dimensional Dirichlet distribution which is of the form
    \[ \pi(w)= \frac{\Gamma(\alpha)}{\Gamma(\alpha/K)^K} w_1^{\alpha/K-1}w_2^{\alpha/K-1}\cdots w_K^{\alpha/K-1} \mathbb{I}(w \in \Delta_K), \]
    where $\Delta_K$ denotes the $K$-dimensional simplex. So, the prior is of the same form as in Condition \ref{hyp:5} with $C(w) = \Gamma(\alpha)/\Gamma(\alpha/K)^K\,\mathbb{I}(w \in \Delta_K)$ which is a constant on the simplex. Condition \ref{hyp:1} is verified using Theorem 4.1 from \cite{rousseau2019bayesian} which can also be applied to overfitted mixtures. Hence, the result \cite{rousseau2011asymptotic} applies in this case.

\textcolor{forestgreen}{    
    In the Pitman--Yor multinomial case, the prior on the weights is a ratio-stable distribution defined in \cite{carlton2002family} and denoted by $w\sim RS(\sigma,\tilde{\alpha};1/K,\ldots,1/K)$. In the general case, no density is available so it is not possible to satisfy \ref{hyp:5} and the \cite{rousseau2019bayesian} results cannot give us any guarantee. 
    In the interesting $\sigma=1/2$ case, the density is
    \[ \pi(w) = \frac{\left(1/K\right)^K}{\pi^{\frac{K-1}{2}}} \frac{\Gamma(\tilde{\alpha}+K/2)}{\Gamma(\tilde{\alpha}+1/2)} \frac{w_1^{-3/2}\cdots w_K^{-3/2}}{\left(\frac{1}{w_1K^2}+\cdots+\frac{1}{w_KK^2}\right)^{\tilde{\alpha}+K/2}} \mathbb{I}(w\in\Delta_K). \]
    To write this density in the form  $ \pi(w) = C(w) \prod_{i=1}^Kw_i^{\alpha_i-1}$ we must set:
    \[C(w) \propto \frac{\prod_{i=1}^Kw_i^{-\alpha_i-1/2}}{\left(\sum_{i=1}^Kw_i^{-1}\right)^{\tilde \alpha+K/2}}\]
    Condition \ref{hyp:5} from requires $C(w)$ to be bounded above and below. A necessary condition is obtained by studying the limit of $C(w)$ for any $w_i \to 0$ with the others remaining bounded away from 0.
    \[ \text{As } w_i \to 0; C(w) = \mathcal{O}\left(w_i^{\tilde \alpha - \alpha_i + \frac{K-1}{2}}\right)\]
    For $C(w)$ to remain bounded above and below in this limit requires $\alpha_i = \tilde \alpha + \frac{K-1}{2}$.
    Enforcing this necessary condition for each $\alpha_i$ independently requires rewriting $C(w)$ as:
        \begin{align*}
        C(w) &\propto \frac{\prod_{i=1}^Kw_i^{-(\alpha+K/2)}}{\left(\sum_{i=1}^Kw_i^{-1}\right)^{\tilde \alpha+K/2}} \\
        & \propto \frac{1}{(w_1\cdots w_K)^{\tilde{\alpha}+K/2}}\times\left(\frac{w_1\cdots w_K}{w_2\cdots w_K+w_1w_3\cdots w_K+\cdots+w_1\cdots w_{K-1}}\right)^{\tilde{\alpha}+K/2} \\
        & \propto \left(\frac{1}{w_2\cdots w_K+w_1w_3\cdots w_K+\cdots+w_1\cdots w_{K-1}}\right)^{\tilde{\alpha}+K/2} 
    \end{align*}
    This quantity is bounded from below for all $\tilde \alpha>0$, for any $w_i \to 0$ independently, but not for two $w_i, w_j \to 0$ at different rates, in which case all terms at the denominator vanish and $C(w)$ diverges.
    We have found a set of necessary conditions which are incompatible, hence there is no choice of $\alpha_i$ such that Condition \ref{hyp:5} can be satisfied, and in this case too the \cite{rousseau2019bayesian} results cannot give us any guarantee. 
    }

\end{proof}

\begin{proof}[Proof of Lemma \ref{lem:Contract}]
This is a direct application of Corollary 1 from \cite{scricciolo2014adaptive}. To apply this corollary, we must check that the kernel $f(\cdot\mid\theta)$ associated with the mixing measure $G$ is a symmetric probability density such that, for some constants $0<\rho<\infty$ and $0<\eta\leq2$, the Fourier transform $\hat{f}$ of $f(\cdot\mid\theta)$ satisfies:
\[|\hat{f}(t)|\sim e^{-\left(\rho|t|\right)^\eta} \; \text{as } |t|\to\infty.\]
This is satisfied by assumption. 
In Condition \ref{hyp:A1}, the kernel $f(\cdot\mid\theta)$ is assumed to be symmetric, monotone decreasing in $|x|$ and to satisfy a tail condition. The kernel $f(\cdot\mid\theta)$ also belongs to the set 
\[\mathcal{A}^{\rho,\eta,L} := \left\{f:\mathbb{R}\to\mathbb{R}^+ \, |\, \|f\|_1 = 1, \, \int e^{2(\rho|t|)^\eta}|\hat{f}(t)|^2\mathrm{d}t \leq 2\pi L^2\right\},\]
where $\hat{f}$ denotes the Fourier transform of $f$ and $\rho,L,\eta$ are some positive constants. 

We also need to check that for a sequence $\Tilde{\varepsilon}_n >0$ such that $\Tilde{\varepsilon}_n\to0$ as $n\to \infty$ and $n\Tilde{\varepsilon}_n^2\gtrsim\textcolor{forestgreen}{\log (n)}^{1/\eta}$, we have
\[ \Pi(B_\text{KL}(f_0^X; \Tilde{\varepsilon}_n^2))\gtrsim \exp(-Cn\Tilde{\varepsilon}_n^2) \; \text{for some constant } 0<C<\infty,\]
where $B_\text{KL}(f_0^X; \varepsilon^2) := \left\{f:\; \int f_0^X\log(f_0^X/f)\leq \varepsilon^2,\, \int f_0^X(\log(f_0^X/f))^2\leq\varepsilon^2\right\}$ denotes Kullback--Leibler neighbourhoods of the true density $f_0^X$. 
This condition is verified in the second part of the proof of Theorem 1 in \cite{scricciolo2014adaptive}. 
\end{proof}

\section{Details on the simulation study of Section \ref{sec:simulation}}
\label{an:simulation}
We consider the mixture model, with $K=10$: 
\begin{align*}
  f^X(x) = \sum_{k = 1}^{K} w_k f(x \mid \mu_k, \Sigma_k).
  \end{align*}
  Parameters have the following prior distributions:
  \begin{align*}
    (w_1, \ldots,w_K) \sim  \text{Dir}_K(\Bar{\alpha},\ldots,\Bar{\alpha}), \quad \Bar{\alpha}= \alpha/K,\\
   \mu_k \sim \mathcal{N}(b_0, B_0), \quad k = 1,\ldots, K,\\
   \Sigma_k^{-1} \sim \mathcal{W}(c_0, C_0), \quad C_0 \sim \mathcal{W}(h_0, H_0).
\end{align*}
Parameters for Wishart distribution are defined as in \cite{malsiner-walli_model-based_2016}: $c_0= 2.5 +\frac{r-1}{2}$, $h_0= 0.5 +\frac{r-1}{2}$, $H_0 = \frac{100\,h_0}{c_0}\text{diag}(1/R_1^2,\ldots,1/R_r^2)$, and $B_0 = \text{diag}(R_1^2,\ldots,R^2_r)$, where $r$ is dimension of $\Sigma$ matrix, and $R_j$ is the range of the data in each dimension.
Parameter $b_0$ is set to the median of the data. 

We run two MCMC chains of 20 000 iterations each, with 10 000 burn-in iterations.
Convergence assessment was done through the calculation of Gelman-Rubin diagnostics \citep{gelman1992inference} and visual
inspection of the trace plots. \\

We provide here some more details on Figure \ref{fig:DPM-varying-alpha} in Discussion.
Figure \ref{fig:DPM-varying-alpha} illustrates two different cases where the parameter $\Bar{\alpha}$ is not fixed. First, we consider the fixed prior expected number of clusters, such as  $\mathbb{E}[K_n]= 5$, which leads to decreasing of the parameter $\Bar{\alpha}$ with $n$. Posterior distribution of the number of clusters is presented in Figure~\ref{fig:DPM-varying-alpha} (a). 
The second approach consists in using the hyperprior for parameter $\Bar{\alpha}$. We consider  $\Bar{\alpha} \sim \text{Ga}(a,bK)$, where parameters $a =1$, $b = 0.1$ and $K=10$ is the number of components, which leads to less informative prior distribution of the number of clusters. This simulation setting is also different from theoretical assumptions required by \cite{ascolani2022clustering}. 

\section{Details on the real-data analysis of Section \ref{sec:real}}
\label{an:real}
We consider two different mixture models in Section \ref{sec:real}.
The first one is of the form:
\[ f^X(x) = \int_\Theta f(x\mid \theta) G(\mathrm{d}\theta),\]
where $\theta_k=(\mu_k, \Tilde{\sigma}_k^2)$ and $f(x\mid \theta)=\mathcal{N}(x\mid \mu,\Tilde{\sigma}^2)$. Here, the mixing measure is distributed as a Pitman--Yor process, $G \sim \mathrm{PY}(\alpha, \sigma; P)$, where $P$ is the base measure defined hierarchically as the following:
\begin{align*}
    \Tilde{\sigma}_k^2 &\overset{\mathrm{iid}}{\sim} \mathrm{IG}(a_0, b_0), \quad\;\; k >1,\\
    \mu_k \mid \Tilde{\sigma}_k^2 &\overset{\mathrm{ind}}{\sim} \mathcal{N}(m_0,\Tilde{\sigma}_k^2), \quad k >1.
\end{align*}
$\mathrm{IG}$ denotes the Inverse-Gamma distribution.
Parameters for the Inverse-Gamma distribution are defined as the default values of the \texttt{BNPmix} package \citep[see][]{corradin_bnpmix_2021}: $a_0 = 2$, $b_0$ is set as the sample variance of the data and $m_0$ as the sample mean of the data.
We used various values of  $\alpha \in\{0.01,0.5\}$ and $\sigma\in\{0.1,0.25\}$

We run four MCMC chains of 20 000 iterations each, with 10 000 burn-in iterations using the marginal sampler of the \texttt{BNPmix} package.

The second model is of the following form:
\[
 f^X(x) = \sum_{k = 1}^{K} w_k \mathcal{N}(x \mid \mu_k, \Tilde{\sigma}_k^2),
\]
with $K=10$. In this case, the mixing measure is distributed as a Dirichlet multinomial process. The parameters have the following prior distributions:
  \begin{align*}
    (w_1, \ldots,w_K) &\sim  \text{Dir}_K(\Bar{\alpha},\ldots,\Bar{\alpha}), \quad \Bar{\alpha}= \alpha/K,\\
   \mu_k &\sim \mathcal{N}(b_0, B_0), \quad k = 1,\ldots, K,\\
   \Tilde{\sigma}_k^2 &\sim \mathrm{IG}(c_0, C_0), \quad C_0 \sim \mathrm{Ga}(h_0, H_0).
\end{align*}
The parameters for the Gamma distribution are defined as previously but in a univariate form: $c_0= 2.5$, $h_0= 0.5$, $H_0 = \frac{100\,h_0}{c_0\,R^2}$, and $B_0 = R^2$ where $R$ is the range of the data. Parameter $b_0$ is set to the median of the data.
We used various values of $\Bar{\alpha}=\alpha/K\in\{0.01,0.5,1,2\}$.

We run two MCMC chains of 80 000 iterations each, with 30 000 burn-in iterations. For both models, convergence assessment was done through the calculation of Gelman--Rubin diagnostics \citep{gelman1992inference} and visual inspection of the trace plots.

%% file: main_neutral.bbl
\begin{thebibliography}{73}
\newcommand{\enquote}[1]{``#1''}
\expandafter\ifx\csname natexlab\endcsname\relax\def\natexlab#1{#1}\fi
\expandafter\ifx\csname url\endcsname\relax
  \def\url#1{{\tt #1}}\fi
\expandafter\ifx\csname urlprefix\endcsname\relax\def\urlprefix{URL }\fi

\bibitem[{Arbel and Favaro(2021)}]{arbel2021approximating}
Arbel, J. and Favaro, S. (2021).
\newblock \enquote{{Approximating predictive probabilities of Gibbs-type
  priors}.}
\newblock {\em Sankhya A\/}, 83(1), 496--519.

\bibitem[{Arbel et~al.(2017)Arbel, Favaro, Nipoti, and Teh}]{arbel2017bayesian}
Arbel, J., Favaro, S., Nipoti, B., and Teh, Y.~W. (2017).
\newblock \enquote{Bayesian nonparametric inference for discovery
  probabilities: Credible intervals and large sample asymptiotics.}
\newblock {\em Statistica Sinica\/}, 839--858.

\bibitem[{Argiento and De~Iorio(2022)}]{argiento2019infinity}
Argiento, R. and De~Iorio, M. (2022).
\newblock \enquote{{Is infinity that far? A Bayesian nonparametric perspective
  of finite mixture models}.}
\newblock {\em The Annals of Statistics\/}.

\bibitem[{Ascolani et~al.(2022)Ascolani, Lijoi, Rebaudo, and
  Zanella}]{ascolani2022clustering}
Ascolani, F., Lijoi, A., Rebaudo, G., and Zanella, G. (2022).
\newblock \enquote{{Clustering consistency with Dirichlet process mixtures}.}
\newblock {\em Biometrika. In press\/}.

\bibitem[{Attorre et~al.(2020)Attorre, Cambria, Agrillo, Alessi, Alf{\`o},
  De~Sanctis, Malatesta, Sitzia, Guarino, Marcen{\`o}
  et~al.}]{attorre2020finite}
Attorre, F., Cambria, V.~E., Agrillo, E., Alessi, N., Alf{\`o}, M., De~Sanctis,
  M., Malatesta, L., Sitzia, T., Guarino, R., Marcen{\`o}, C., et~al. (2020).
\newblock \enquote{{Finite Mixture Model-based classification of a complex
  vegetation system}.}
\newblock {\em Vegetation Classification and Survey\/}, 1, 77.

\bibitem[{Bacallado et~al.(2017)Bacallado, Battiston, Favaro, and
  Trippa}]{Sufficientness}
Bacallado, S., Battiston, M., Favaro, S., and Trippa, L. (2017).
\newblock \enquote{{Sufficientness Postulates for Gibbs-Type Priors and
  Hierarchical Generalizations}.}
\newblock {\em Statistical Science\/}, 32(4), 487 -- 500.

\bibitem[{Bystrova et~al.(2021)Bystrova, Arbel, Kon Kam~King, and
  Deslandes}]{bystrova2021approximating}
Bystrova, D., Arbel, J., Kon Kam~King, G., and Deslandes, F. (2021).
\newblock \enquote{Approximating the clusters' prior distribution in {B}ayesian
  nonparametric models.}
\newblock In {\em Third Symposium on Advances in Approximate Bayesian
  Inference\/}.

\bibitem[{Cai et~al.(2021)Cai, Campbell, and Broderick}]{cai2021finite}
Cai, D., Campbell, T., and Broderick, T. (2021).
\newblock \enquote{Finite mixture models do not reliably learn the number of
  components.}
\newblock In {\em International Conference on Machine Learning\/}, 1158--1169.
  PMLR.

\bibitem[{Carlton(2002)}]{carlton2002family}
Carlton, M.~A. (2002).
\newblock \enquote{{A family of densities derived from the three-parameter
  Dirichlet process}.}
\newblock {\em Journal of applied probability\/}, 39(4), 764--774.

\bibitem[{Caron and Fox(2017)}]{caron2017sparse}
Caron, F. and Fox, E.~B. (2017).
\newblock \enquote{Sparse graphs using exchangeable random measures.}
\newblock {\em Journal of the Royal Statistical Society Series B: Statistical
  Methodology\/}, 79(5), 1295--1366.

\bibitem[{Celeux et~al.(2000)Celeux, Hurn, and
  Robert}]{celeux2000computational}
Celeux, G., Hurn, M., and Robert, C.~P. (2000).
\newblock \enquote{Computational and inferential difficulties with mixture
  posterior distributions.}
\newblock {\em Journal of the American Statistical Association\/}, 95(451),
  957--970.

\bibitem[{Cesari et~al.(2014)Cesari, Favaro, and Nipoti}]{CESARI201479}
Cesari, O., Favaro, S., and Nipoti, B. (2014).
\newblock \enquote{{Posterior analysis of rare variants in Gibbs-type species
  sampling models}.}
\newblock {\em Journal of Multivariate Analysis\/}, 131, 79--98.

\bibitem[{Chambaz and Rousseau(2008)}]{chambaz:hal-00262054}
Chambaz, A. and Rousseau, J. (2008).
\newblock \enquote{{Bounds for Bayesian order identification with application
  to mixtures}.}
\newblock {\em {The Annals of Statistics}\/}, 36(2), 938--962.

\bibitem[{Chaumeny et~al.(2022)Chaumeny, van~der Molen~Moris, Davison, and
  Kirk}]{chaumeny_bayesian_2022}
Chaumeny, Y., van~der Molen~Moris, J., Davison, A.~C., and Kirk, P. D.~W.
  (2022).
\newblock \enquote{Bayesian nonparametric mixture inconsistency for the number
  of components: {How} worried should we be in practice?}
\newblock ArXiv:2207.14717 [stat].

\bibitem[{Chen(1995)}]{chen_optimal_1995}
Chen, J. (1995).
\newblock \enquote{Optimal {Rate} of {Convergence} for {Finite} {Mixture}
  {Models}.}
\newblock {\em The Annals of Statistics\/}, 23(1), 221--233.
\newblock Publisher: Institute of Mathematical Statistics.

\bibitem[{Corradin et~al.(2021)Corradin, Canale, and
  Nipoti}]{corradin_bnpmix_2021}
Corradin, R., Canale, A., and Nipoti, B. (2021).
\newblock \enquote{{BNPmix}: {An} {R} {Package} for {Bayesian} {Nonparametric}
  {Modeling} via {Pitman}-{Yor} {Mixtures}.}
\newblock {\em Journal of Statistical Software\/}, 100, 1--33.

\bibitem[{De~Blasi et~al.(2015)De~Blasi, Favaro, Lijoi, Mena, Pruenster, and
  Ruggiero}]{de_blasi_are_2015}
De~Blasi, P., Favaro, S., Lijoi, A., Mena, R.~H., Pruenster, I., and Ruggiero,
  M. (2015).
\newblock \enquote{Are {Gibbs}-type priors the most natural generalization of
  the {Dirichlet} process?}
\newblock {\em IEEE Transactions on Pattern Analysis and Machine
  Intelligence\/}, 37(2), 212--229.

\bibitem[{Dudley et~al.(1991)Dudley, Giuffra, Raine, and
  Reeders}]{dudley_assessing_1991}
Dudley, C.~R., Giuffra, L.~A., Raine, A.~E., and Reeders, S.~T. (1991).
\newblock \enquote{Assessing the role of {APNH}, a gene encoding for a human
  amiloride-sensitive {Na}+/{H}+ antiporter, on the interindividual variation
  in red cell {Na}+/{Li}+ countertransport.}
\newblock {\em Journal of the American Society of Nephrology\/}, 2(4), 937.

\bibitem[{Favaro et~al.(2009)Favaro, Lijoi, Mena, and Prünster}]{favaro_2009}
Favaro, S., Lijoi, A., Mena, R.~H., and Prünster, I. (2009).
\newblock \enquote{{Bayesian non-parametric inference for species variety with
  a two-parameter Poisson–Dirichlet process prior}.}
\newblock {\em Journal of the Royal Statistical Society: Series B (Statistical
  Methodology)\/}, 71(5), 993--1008.

\bibitem[{Favaro et~al.(2012)Favaro, Lijoi, and Pr{\"u}nster}]{favaro2012new}
Favaro, S., Lijoi, A., and Pr{\"u}nster, I. (2012).
\newblock \enquote{A new estimator of the discovery probability.}
\newblock {\em Biometrics\/}, 68(4), 1188--1196.

\bibitem[{Ferguson(1973)}]{ferguson1973bayesian}
Ferguson, T. (1973).
\newblock \enquote{{A Bayesian analysis of some nonparametric problems}.}
\newblock {\em The Annals of Statistics\/}, 1(2), 209--230.

\bibitem[{Fraley and Raftery(2002)}]{fraley_model-based_2002}
Fraley, C. and Raftery, A.~E. (2002).
\newblock \enquote{Model-{Based} {Clustering}, {Discriminant} {Analysis}, and
  {Density} {Estimation}.}
\newblock {\em Journal of the American Statistical Association\/}, 97(458),
  611--631.

\bibitem[{Fr{\"u}hwirth-Schnatter(2006)}]{fruhwirth2006finite}
Fr{\"u}hwirth-Schnatter, S. (2006).
\newblock {\em Finite mixture and Markov switching models\/}, volume 425.
\newblock Springer.

\bibitem[{Fr{\"u}hwirth-Schnatter et~al.(2021)Fr{\"u}hwirth-Schnatter,
  Malsiner-Walli, and Gr{\"u}n}]{fruhwirth2021generalized}
Fr{\"u}hwirth-Schnatter, S., Malsiner-Walli, G., and Gr{\"u}n, B. (2021).
\newblock \enquote{{Generalized Mixtures of Finite Mixtures and Telescoping
  Sampling}.}
\newblock {\em Bayesian Analysis\/}, 16(4), 1279 -- 1307.

\bibitem[{Fr{\"u}hwirth-Schnatter et~al.(2012)Fr{\"u}hwirth-Schnatter,
  Pamminger, Weber, and Winter-Ebmer}]{FrhwirthSchnatter_2012}
Fr{\"u}hwirth-Schnatter, S., Pamminger, C., Weber, A., and Winter-Ebmer, R.
  (2012).
\newblock \enquote{{Labor market entry and earnings dynamics: Bayesian
  inference using mixtures-of-experts Markov chain clustering}.}
\newblock {\em Journal of Applied Econometrics\/}, 27(7), 1116--1137.

\bibitem[{Frühwirth-Schnatter et~al.(2019)Frühwirth-Schnatter, Celeux, and
  Robert}]{fruhwirth-schnatter_handbook_2019}
Frühwirth-Schnatter, S., Celeux, G., and Robert, C.~P. (eds.) (2019).
\newblock {\em Handbook of Mixture Analysis\/}.
\newblock CRC Press, Taylor \& Francis Group.

\bibitem[{Gelman and Rubin(1992)}]{gelman1992inference}
Gelman, A. and Rubin, D.~B. (1992).
\newblock \enquote{Inference from iterative simulation using multiple
  sequences.}
\newblock {\em Statistical Science\/}, 7(4), 457--472.

\bibitem[{Ghosal et~al.(1999)Ghosal, Ghosh, and Ramamoorthi}]{ghosal1999}
Ghosal, S., Ghosh, J.~K., and Ramamoorthi, R.~V. (1999).
\newblock \enquote{{Posterior consistency of Dirichlet mixtures in density
  estimation}.}
\newblock {\em The Annals of Statistics\/}, 27(1), 143 -- 158.

\bibitem[{Ghosal and Van Der~Vaart(2007)}]{ghosal2007posterior}
Ghosal, S. and Van Der~Vaart, A. (2007).
\newblock \enquote{{Posterior convergence rates of Dirichlet mixtures at smooth
  densities}.}
\newblock {\em The Annals of Statistics\/}, 35(2), 697--723.

\bibitem[{Ghosal and Van~der Vaart(2017)}]{ghosal2017fundamentals}
Ghosal, S. and Van~der Vaart, A. (2017).
\newblock {\em {Fundamentals of nonparametric Bayesian inference}\/},
  volume~44.
\newblock Cambridge University Press.

\bibitem[{Gnedin and Pitman(2006)}]{gnedin2006exchangeable}
Gnedin, A. and Pitman, J. (2006).
\newblock \enquote{Exchangeable {G}ibbs partitions and {S}tirling triangles.}
\newblock {\em Journal of Mathematical Sciences\/}, 138(3), 5674--5685.

\bibitem[{Greve et~al.(2022)Greve, Gr{\"u}n, Malsiner-Walli, and
  Fr{\"u}hwirth-Schnatter}]{greve2020spying}
Greve, J., Gr{\"u}n, B., Malsiner-Walli, G., and Fr{\"u}hwirth-Schnatter, S.
  (2022).
\newblock \enquote{{Spying on the prior of the number of data clusters and the
  partition distribution in Bayesian cluster analysis}.}
\newblock {\em Australian \& New Zealand Journal of Statistics\/}, 64(2),
  205--229.

\bibitem[{Guha et~al.(2021)Guha, Ho, and Nguyen}]{guha2021posterior}
Guha, A., Ho, N., and Nguyen, X. (2021).
\newblock \enquote{{On posterior contraction of parameters and interpretability
  in Bayesian mixture modeling}.}
\newblock {\em Bernoulli\/}, 27(4), 2159--2188.

\bibitem[{Ho and Nguyen(2016)}]{ho_strong_2016}
Ho, N. and Nguyen, X. (2016).
\newblock \enquote{On strong identifiability and convergence rates of parameter
  estimation in finite mixtures.}
\newblock {\em Electronic Journal of Statistics\/}, 10(1), 271--307.

\bibitem[{Ishwaran and James(2001)}]{Ishwaran_2001}
Ishwaran, H. and James, L.~F. (2001).
\newblock \enquote{Gibbs Sampling Methods for Stick-Breaking Priors.}
\newblock {\em Journal of the American Statistical Association\/}, 96(453),
  161--173.

\bibitem[{Ishwaran and Zarepour(2000)}]{ishwaran2000markov}
Ishwaran, H. and Zarepour, M. (2000).
\newblock \enquote{{Markov chain Monte Carlo in approximate Dirichlet and beta
  two-parameter process hierarchical models}.}
\newblock {\em Biometrika\/}, 87(2), 371--390.

\bibitem[{Ishwaran and Zarepour(2002)}]{ishwaran_exact_2002}
--- (2002).
\newblock \enquote{Exact and approximate sum representations for the
  {Dirichlet} process.}
\newblock {\em Canadian Journal of Statistics\/}, 30(2), 269--283.

\bibitem[{James et~al.(2009)James, Lijoi, and
  Pr{\"u}nster}]{james2009posterior}
James, L.~F., Lijoi, A., and Pr{\"u}nster, I. (2009).
\newblock \enquote{Posterior analysis for normalized random measures with
  independent increments.}
\newblock {\em Scandinavian Journal of Statistics\/}, 36(1), 76--97.

\bibitem[{Jara et~al.(2010)Jara, Lesaffre, Iorio, and Quintana}]{Jara_2010}
Jara, A., Lesaffre, E., Iorio, M.~D., and Quintana, F. (2010).
\newblock \enquote{{Bayesian semiparametric inference for multivariate
  doubly-interval-censored data}.}
\newblock {\em The Annals of Applied Statistics\/}, 4(4), 2126 -- 2149.

\bibitem[{Kruijer et~al.(2010)Kruijer, Rousseau, and
  Vaart}]{kruijer_adaptive_2010}
Kruijer, W., Rousseau, J., and Vaart, A. v.~d. (2010).
\newblock \enquote{Adaptive {Bayesian} density estimation with location-scale
  mixtures.}
\newblock {\em Electronic Journal of Statistics\/}, 4(none), 1225--1257.

\bibitem[{Legramanti et~al.(2022)Legramanti, Rigon, Durante, and
  Dunson}]{Legramanti_2022}
Legramanti, S., Rigon, T., Durante, D., and Dunson, D.~B. (2022).
\newblock \enquote{{Extended stochastic block models with application to
  criminal networks}.}
\newblock {\em The Annals of Applied Statistics\/}, 16(4), 2369 -- 2395.

\bibitem[{Lijoi et~al.(2007{\natexlab{a}})Lijoi, Mena, and
  Pr{\"u}nster}]{lijoi2007controlling}
Lijoi, A., Mena, R.~H., and Pr{\"u}nster, I. (2007{\natexlab{a}}).
\newblock \enquote{{Controlling the reinforcement in Bayesian non-parametric
  mixture models}.}
\newblock {\em Journal of the Royal Statistical Society: Series B (Statistical
  Methodology)\/}, 69(4), 715--740.

\bibitem[{Lijoi et~al.(2005{\natexlab{a}})Lijoi, Mena, and
  Prünster}]{Lijoi_2005}
Lijoi, A., Mena, R.~H., and Prünster, I. (2005{\natexlab{a}}).
\newblock \enquote{{Hierarchical Mixture Modeling With Normalized
  Inverse-Gaussian Priors}.}
\newblock {\em Journal of the American Statistical Association\/}, 100(472),
  1278--1291.

\bibitem[{Lijoi et~al.(2007{\natexlab{b}})Lijoi, Mena, and
  Prünster}]{lijoi_species_sampling2007}
--- (2007{\natexlab{b}}).
\newblock \enquote{{Bayesian Nonparametric Estimation of the Probability of
  Discovering New Species}.}
\newblock {\em Biometrika\/}, 94(4), 769--786.

\bibitem[{Lijoi et~al.(2020{\natexlab{a}})Lijoi, Pr{\"u}nster, and
  Rigon}]{lijoi2020finite-dim}
Lijoi, A., Pr{\"u}nster, I., and Rigon, T. (2020{\natexlab{a}}).
\newblock \enquote{{Finite-dimensional discrete random structures and Bayesian
  clustering}.}
\newblock {\em Preprint\/}.

\bibitem[{Lijoi et~al.(2020{\natexlab{b}})Lijoi, Pr{\"u}nster, and
  Rigon}]{lijoi2020pitman}
--- (2020{\natexlab{b}}).
\newblock \enquote{{The Pitman--Yor multinomial process for mixture
  modelling}.}
\newblock {\em Biometrika\/}, 107(4), 891--906.

\bibitem[{Lijoi et~al.(2005{\natexlab{b}})Lijoi, Pr{\"u}nster, and
  Walker}]{lijoi2005consistency}
Lijoi, A., Pr{\"u}nster, I., and Walker, S.~G. (2005{\natexlab{b}}).
\newblock \enquote{{On consistency of nonparametric normal mixtures for
  Bayesian density estimation}.}
\newblock {\em Journal of the American Statistical Association\/}, 100(472),
  1292--1296.

\bibitem[{Lijoi and Prünster(2010)}]{lijoi_prunster_2010}
Lijoi, A. and Prünster, I. (2010).
\newblock {\em Models beyond the Dirichlet process\/}, 80–136.
\newblock Cambridge Series in Statistical and Probabilistic Mathematics.
  Cambridge University Press.

\bibitem[{Lo(1984)}]{lo1984class}
Lo, A.~Y. (1984).
\newblock \enquote{{On a class of Bayesian nonparametric estimates: I. Density
  estimates}.}
\newblock {\em The Annals of Statistics\/}, 351--357.

\bibitem[{Malsiner-Walli et~al.(2016)Malsiner-Walli, Frühwirth-Schnatter, and
  Grün}]{malsiner-walli_model-based_2016}
Malsiner-Walli, G., Frühwirth-Schnatter, S., and Grün, B. (2016).
\newblock \enquote{Model-based clustering based on sparse finite {Gaussian}
  mixtures.}
\newblock {\em Statistics and Computing\/}, 26(1-2), 303--324.

\bibitem[{Miller(2014)}]{miller_nonparametric_nodate}
Miller, J.~W. (2014).
\newblock \enquote{Nonparametric and {Variable}-{Dimension} {Bayesian}
  {Mixture} {Models}: {Analysis}, {Comparison}, and {New} {Methods}.}
\newblock Ph.D. thesis, Brown University, Division of Applied Mathematics.

\bibitem[{Miller(2023)}]{miller2022consistency}
--- (2023).
\newblock \enquote{Consistency of mixture models with a prior on the number of
  components.}
\newblock {\em Dependence Modeling\/}, 11(1), 20220150.

\bibitem[{Miller and Dunson(2019)}]{miller2018robust}
Miller, J.~W. and Dunson, D.~B. (2019).
\newblock \enquote{{Robust Bayesian Inference via Coarsening}.}
\newblock {\em Journal of the American Statistical Association\/}, 114(527),
  1113--1125.

\bibitem[{Miller and Harrison(2014)}]{miller2014inconsistency}
Miller, J.~W. and Harrison, M.~T. (2014).
\newblock \enquote{{Inconsistency of Pitman-Yor process mixtures for the number
  of components}.}
\newblock {\em The Journal of Machine Learning Research\/}, 15(1), 3333--3370.

\bibitem[{Miller and Harrison(2018)}]{miller2018}
--- (2018).
\newblock \enquote{{Mixture Models With a Prior on the Number of Components}.}
\newblock {\em Journal of the American Statistical Association\/}, 113(521),
  340--356.

\bibitem[{Muliere and Secchi(1995)}]{muliere1995note}
Muliere, P. and Secchi, P. (1995).
\newblock \enquote{{A note on a proper Bayesian bootstrap}.}

\bibitem[{Muliere and Secchi(2003)}]{muliere2003weak}
--- (2003).
\newblock \enquote{Weak {Convergence} of a {Dirichlet}-{Multinomial}
  {Process}.}
\newblock {\em Georgian Mathematical Journal\/}, 10(2), 319--324.

\bibitem[{M{\"u}ller et~al.(1996)M{\"u}ller, Erkanli, and
  West}]{muller1996bayesian}
M{\"u}ller, P., Erkanli, A., and West, M. (1996).
\newblock \enquote{Bayesian curve fitting using multivariate normal mixtures.}
\newblock {\em Biometrika\/}, 83(1), 67--79.

\bibitem[{Nguyen(2013)}]{nguyenConvergenceLatentMixing2013}
Nguyen, X. (2013).
\newblock \enquote{Convergence of latent mixing measures in finite and infinite
  mixture models.}
\newblock {\em The Annals of Statistics\/}, 41(1), 370--400.

\bibitem[{Nobile(1994)}]{nobile_1994}
Nobile, A. (1994).
\newblock \enquote{Bayesian Analysis of Finite Mixture Distributions.}
\newblock Ph.D. thesis, Department of Statistics, Carnegie Mellon University,
  Pittsburgh, PA.

\bibitem[{Ohn and Lin(2023)}]{ohn2020optimal}
Ohn, I. and Lin, L. (2023).
\newblock \enquote{{Optimal Bayesian estimation of Gaussian mixtures with
  growing number of components}.}
\newblock {\em Bernoulli\/}, 29(2), 1195--1218.

\bibitem[{Petralia et~al.(2012)Petralia, Rao, and
  Dunson}]{petralia2012repulsive}
Petralia, F., Rao, V., and Dunson, D. (2012).
\newblock \enquote{Repulsive mixtures.}
\newblock {\em Advances in Neural Information Processing Systems\/}, 25.

\bibitem[{Pitman(2003)}]{pitman2003poisson}
Pitman, J. (2003).
\newblock \enquote{Poisson-{Kingman} partitions.}
\newblock {\em Statistics and science: a Festschrift for Terry Speed\/}, 40,
  1--35.
\newblock Publisher: Institute of Mathematical Statistics.

\bibitem[{Ram{\'\i}rez et~al.(2019)Ram{\'\i}rez, Forbes, Arbel, Arnaud, and
  Dojat}]{Ramrez_2019}
Ram{\'\i}rez, V.~M., Forbes, F., Arbel, J., Arnaud, A., and Dojat, M. (2019).
\newblock \enquote{Quantitative {MRI} {Characterization} of {Brain}
  {Abnormalities} in {DE} {NOVO} {Parkinsonian} {Patients}.}
\newblock In {\em 2019 {IEEE} 16th {International} {Symposium} on {Biomedical}
  {Imaging} ({ISBI} 2019)\/}, 1572--1575.

\bibitem[{Regazzini et~al.(2003)Regazzini, Lijoi, and
  Prünster}]{regazzini_distributional_2003}
Regazzini, E., Lijoi, A., and Prünster, I. (2003).
\newblock \enquote{Distributional results for means of normalized random
  measures with independent increments.}
\newblock {\em The Annals of Statistics\/}, 31(2), 560--585.

\bibitem[{Richardson and Green(1997)}]{Richardson_1997}
Richardson, S. and Green, P.~J. (1997).
\newblock \enquote{On {B}ayesian Analysis of Mixtures with an Unknown Number of
  Components (with discussion).}
\newblock {\em Journal of the Royal Statistical Society: Series B (Statistical
  Methodology)\/}, 59(4), 731--792.

\bibitem[{Rousseau et~al.(2019)Rousseau, Grazian, and
  Lee}]{rousseau2019bayesian}
Rousseau, J., Grazian, C., and Lee, J.~E. (2019).
\newblock \enquote{Bayesian mixture models: Theory and methods.}
\newblock In Fruhwirth-Schnatter, S., Celeux, G., and Robert, C.~P. (eds.),
  {\em Handbook of Mixture Analysis\/}, 53--72. Chapman and Hall/CRC.

\bibitem[{Rousseau and Mengersen(2011)}]{rousseau2011asymptotic}
Rousseau, J. and Mengersen, K. (2011).
\newblock \enquote{Asymptotic behaviour of the posterior distribution in
  overfitted mixture models.}
\newblock {\em Journal of the Royal Statistical Society: Series B (Statistical
  Methodology)\/}, 73(5), 689--710.

\bibitem[{Scricciolo(2014)}]{scricciolo2014adaptive}
Scricciolo, C. (2014).
\newblock \enquote{Adaptive {Bayesian} {Density} {Estimation} in {Lp}-metrics
  with {Pitman}-{Yor} or {Normalized} {Inverse}-{Gaussian} {Process} {Kernel}
  {Mixtures}.}
\newblock {\em Bayesian Analysis\/}, 9(2).

\bibitem[{Teh and Jordan(2010)}]{teh2010hierarchical}
Teh, Y.~W. and Jordan, M.~I. (2010).
\newblock \enquote{{Hierarchical Bayesian nonparametric models with
  applications}.}
\newblock {\em Bayesian nonparametrics\/}, 1, 158--207.

\bibitem[{Ullah and Mengersen(2019)}]{ullah2019bayesian}
Ullah, I. and Mengersen, K. (2019).
\newblock \enquote{{Bayesian mixture models and their Big Data implementations
  with application to invasive species presence-only data}.}
\newblock {\em Journal of Big Data\/}, 6(1), 1--25.

\bibitem[{Wade and Ghahramani(2018)}]{wade_bayesian_2018}
Wade, S. and Ghahramani, Z. (2018).
\newblock \enquote{Bayesian {Cluster} {Analysis}: {Point} {Estimation} and
  {Credible} {Balls} (with {Discussion}).}
\newblock {\em Bayesian Analysis\/}, 13(2), 559--626.
\newblock Publisher: International Society for Bayesian Analysis.

\bibitem[{Xie and Xu(2020)}]{xie_bayesian_2020}
Xie, F. and Xu, Y. (2020).
\newblock \enquote{Bayesian {Repulsive} {Gaussian} {Mixture} {Model}.}
\newblock {\em Journal of the American Statistical Association\/}, 115(529),
  187--203.

\end{thebibliography}
